\documentclass[onefignum,onetabnum]{siamart171218}

\usepackage{lipsum}
\usepackage{amsfonts}
\usepackage{graphicx}
\usepackage{epstopdf}
\usepackage{algorithmic}
\ifpdf
  \DeclareGraphicsExtensions{.eps,.pdf,.png,.jpg}
\else
  \DeclareGraphicsExtensions{.eps}
\fi

\newsiamremark{remark}{Remark}
\newsiamremark{hypothesis}{Hypothesis}
\crefname{hypothesis}{Hypothesis}{Hypotheses}
\newsiamthm{claim}{Claim}

\newsiamthm{assumption}{Assumption}

\headers{Running Algorithm for Time-Varying Optimization}{Y. Tang, E. Dall'Anese, A. Bernstein, and S. Low}

\title{Running Primal-Dual Gradient Method for Time-Varying Nonconvex Problems}

\author{Yujie Tang\thanks{
California Institute of Technology, Pasadena, CA (\email{ytang2@caltech.edu}, \email{slow@caltech.edu}).
}\and 
Emiliano Dall'Anese\thanks{
University of Colorado Boulder, Boulder, CO (\email{emiliano.dallanese@colorado.edu}).
}\and
Andrey Bernstein\thanks{
National Renewable Energy Laboratory, Golden, CO (\email{andrey.bernstein@nrel.gov}).
}\and 
Steven Low\footnotemark[1]
}

\usepackage{amsopn}

\ifpdf
\hypersetup{
  pdftitle={Title},
  pdfauthor={Authors}
}
\fi

\usepackage{enumitem}
\usepackage{mathrsfs}
\usepackage{hyperref}

\DeclareMathOperator*{\argmin}{arg\,min}
\DeclareMathOperator*{\esssup}{ess\,sup}

\makeatletter
\newenvironment{subtheorem}[1]{%
  \def\subtheoremcounter{#1}%
  \refstepcounter{#1}%
  \protected@edef\theparentnumber{\csname the#1\endcsname}%
  \setcounter{parentnumber}{\value{#1}}%
  \setcounter{#1}{0}%
  \expandafter\def\csname the#1\endcsname{\theparentnumber\alph{#1}}%
  \ignorespaces
}{%
  \setcounter{\subtheoremcounter}{\value{parentnumber}}%
  \ignorespacesafterend
}
\makeatother
\newcounter{parentnumber}

\begin{document}

\maketitle
\begin{abstract}
  This paper considers a nonconvex optimization problem that evolves over time, and addresses the synthesis and analysis of regularized primal-dual gradient methods to track a Karush--Kuhn--Tucker (KKT) trajectory. The proposed regularized primal-dual gradient methods are implemented in a running fashion, in the sense that  the underlying optimization problem changes during the iterations of the algorithms. For a problem with twice continuously differentiable cost and constraints, and under a generalization of the Mangasarian-Fromovitz constraint qualification, sufficient conditions are derived for the running algorithm to track a KKT trajectory. Further, asymptotic bounds for the tracking error (as a function of the time-variability of a KKT trajectory) are obtained. A continuous-time version of the algorithm, framed as a system of differential inclusions, is also considered and analytical convergence results are derived. For the continuous-time setting, a set of sufficient conditions for the KKT trajectories not to bifurcate or merge is proposed. Illustrative numerical results inspired by a real-world application are provided. 
\end{abstract}

\begin{keywords}
time-varying optimization, nonconvex optimization, running algorithms, tracking, gradient methods, differential inclusion.
\end{keywords}


\section{Introduction}
\label{sec:introduction}

This paper focuses on continuous-time nonconvex optimization problems of the form:
\begin{equation}
\label{eq:main_problem_cont}
\begin{aligned}
\min_{x \in \mathbb{R}^n}\quad & c(x,t) \\
\textrm{s.t.}\quad & f(x,t)\leq 0 \\
& x\in\mathcal{X}(t)
\end{aligned}
\end{equation}
where $t\in [0,S]$ is the temporal index, with some positive $S\in\mathbb{R}$, $c:\mathbb{R}^n\times\mathbb{R} \rightarrow\mathbb{R}$, and $f:\mathbb{R}^n\times\mathbb{R}\rightarrow\mathbb{R}^m$. For each $t\in [0,S]$, $c(\cdot,t)$ and $f(\cdot,t)$ are assumed twice continuously differentiable. The function $f(x,t)$ is decoupled as: 
\begin{equation}
f(x,t) := f^c(x,t)+f^{nc}(x,t),
\end{equation}
where $f^c(\cdot, t)$ is twice continuously differentiable and has convex components for each $t\in[0,S]$, and $f^{nc}(x,t)$ collects nonconvex functions. Finally, $\mathcal{X}(t)\subset\mathbb{R}^n$ is a time-varying convex set (additional modeling details and pertinent assumptions will be provided in the ensuing sections). The optimization problem~\cref{eq:main_problem_cont} can be associated with systems or networked systems governed by possibly nonlinear physical or logical models, and it can capture performance objectives and constraints that evolve over time. The interdisciplinary nature of time-varying optimization problems~\cite{Simonetto_Asil14} is evident through a number of works in the domains of power systems~\cite{dall2018optimal,tang2017distributed,hauswirth2018time}, communication systems~\cite{Low99,chen2012convergence}, robotic networks~\cite{Bullo18}, and online learning~\cite{mokhtari2016online,yi2016tracking} just to mention a few. On the other hand, suitable modifications of~\cref{eq:main_problem_cont} can model time-varying data processing problems under streaming of measurements such as matrix factorization~\cite{Ling12} and sparse signal recovery~\cite{balavoine2015discrete}. 

Accordingly, under given constraint qualification conditions (that will be discussed shortly in \Cref{sec:ProblemFormulation}), the continuous-time optimization problem implicitly defines \emph{trajectories} for its Karush--Kuhn--Tucker (KKT) points. Letting $(z^*(t))_{t\in [0,S]}$ denote a KKT trajectory, the objective of the paper is to synthesize algorithms that can \emph{track}  $(z^*(t))_{t\in [0,S]}$, and to establish bounds for the tracking error on a per-iteration basis and asymptotically, and to derive results for the convergence rate. 

To this end, this paper first focuses on algorithmic solutions with discrete-time updates, that allow one to account for non-infinitesimal computational times in the updates of the algorithm and delays in the collection of problem inputs. Assume that the interval $[0,S]$ is divided into $T$ time slots of duration $\Delta_T:=S/T$, and consider the following sampled version of~\cref{eq:main_problem_cont}: 
\begin{equation}\label{eq:main_problem}
\begin{aligned}
\min_{x}\quad & c_\tau(x) \\
\textrm{s.t.}\quad & f^c_\tau(x)+f^{nc}_\tau(x)\leq 0 \\
& x\in\mathcal{X}_\tau
\end{aligned}
\end{equation}
for $\tau \in \{1, \ldots, T\}$, where 
$$
\begin{aligned}
c_\tau(x)&:=c(x,\tau\Delta_T), &\quad
f^c_\tau(x)&:=f^c(x,\tau\Delta_T), \\
\mathcal{X}_\tau&:=\mathcal{X}(\tau\Delta_T),
&\quad 
f^{nc}_\tau(x)&:=f^{nc}(x,\tau\Delta_T)
\end{aligned}
$$
are sampled versions of the functions and sets  in~\cref{eq:main_problem_cont} and, similarly, a sampled KKT trajectory is $\big(z^*_\tau := z^*(\tau\Delta_T)\big)_{\tau = 1}^{T}$. One way to obtain $(z^*_\tau)_{\tau = 1}^{T}$ is to solve problem~\cref{eq:main_problem} in a \emph{batch} setting at each time $t$; that is, an iterative algorithm is utilized, and iterations are sequentially performed within each time interval $(\tau\Delta_T, (\tau+1)\Delta_T]$ until convergence to a KKT point $z^*_\tau$. However, a batch solution might not be appropriate in a time-varying setting, since underlying communication and computational complexity requirements may prevent the algorithm from converging within a time interval $\Delta_T$; this is especially the case when the sampling time is chosen small enough to fully model fast-varying costs and constraints and closely approximate the continuous-time optimal trajectory~\cite{SimonettoGlobalsip2014}. In lieu of a batch solution, this paper leverages \emph{running} primal-dual gradient methods to track a sequence of optimal points $(z_\tau^*)_{\tau=1}^T$. 

In particular, the paper offers the following technical approaches and contributions. 

(i) The paper utilizes a regularized primal-dual method, where the regularization comes in the form of a strongly concave term in the dual vector variable that is added to the Lagrangian function~\cite{koshal2011multiuser}; see also~\cite{Rockafellar76} and~\cite{khuzani2016distributed,HongPerturbedPrimalDual18}. The strongly concave regularization term plays a critical role in establishing Q-linear convergence of the proposed algorithm. However, as an artifact of this regularization, existing works for time-invariant convex programs~\cite{koshal2011multiuser}, time-varying convex programs~\cite{onlineOptFeedback}, and for static nonconvex problems~\cite{HongPerturbedPrimalDual18} could prove that gradient-based iterative methods approach an approximate KKT point~\cite{Andreani11}. On the other hand, the present paper provides analytical results in terms of tracking a KKT point (as opposed to an approximate KKT point) of the nonconvex problem~\cref{eq:main_problem} and provides bounds for the distance of the algorithmic iterates from a KKT trajectory. The bounds are obtained by finding conditions under which the regularized primal-dual gradient step exhibits a contraction-like behavior; the bounds are shown to be a function of pertinent algorithmic parameters such as the step size and the coefficient of the regularization function, as well as the maximum temporal variability of a KKT trajectory.

(ii) The paper provides sufficient conditions under which there exists a set of algorithmic parameters that guarantees that the asymptotic error bound holds. From a qualitative standpoint, the sufficient conditions for the existence of feasible parameters suggests that the problem should be ``sufficiently convex'' around a KKT trajectory to overcome the nonlinearity of the nonconvex constraints. 

(iii) The paper analyzes convergence and performance for the case where the algorithm can afford -- in the limit -- an update in infinitesimal time (see, e.g.,~\cite{Cherukuri17,cherukuri2017role,Fazlyab2016a,Rahili2015,Elia-CDC11} and pertinent references therein for continuous-time algorithmic platforms for convex problems); that is, the paper studies the continuous-time limit of the proposed regularized primal-dual gradient algorithm. We show that the continuous-time counterpart of the discrete-time algorithm is given by a system of differential inclusions that has been studied in the literature under the name of \emph{perturbed sweeping processes} \cite{adly2014convex,castaing1993evolution}; the discrete-time algorithm was referred to as \emph{catching algorithm} of the perturbed sweeping processes. The tracking performance of the system of differential inclusions is analytically established, and similarly to the discrete-time case, sufficient conditions under which there exists a set of algorithmic parameters that guarantees bounded tracking error is derived. The continuous-time tracking error bound shares a similar form with the discrete-time tracking error bound, and they are unified under the notion of \emph{eventual tracking error bound}. 

(iv) The results -- for both the discrete-time and the continuous-time algorithms -- are stated in terms of tracking of a KKT trajectory. The paper derives conditions under which the KKT points for a given time instant will always be isolated; that is, bifurcations of KKT trajectories do not happen; similarly, merging of KKT trajectories do not happen either. 

The paper provides contributions over existing works on discrete-time algorithms for time-varying convex optimization~\cite{onlineOptFeedback,dallanese2016optimal,onlineSaddle,Ling14,Simonetto_Asil14,SimonettoGlobalsip2014} by investigating the case of time-varying nonconvex problems in the format of~\cref{eq:main_problem_cont}, and by providing the contributions (i)--(iv) above. With respect to time-varying nonconvex problems (with differentiable cost),~\cite{tang2017real} provided regret-type results for the case where constraints are lifted via approximate barrier functions. The problem of tracking a time-dependent manifold was tackled in~\cite{Zavala2010}, by leveraging a parametric generalized equation;~\cite{Zavala2010} demonstrated that if points along a solution manifold are consistently strongly regular, it is possible to track the manifold approximately by solving a single linear complementarity problem at each time step. On the other hand, the present paper focuses on more general optimization problems, and provides the contributions (i)--(iv) above for a regularized primal-dual gradient method. An Euler--Newton continuation method was proposed in~\cite{dontchev2013euler} for tracking solution trajectories of parametric variational inequalities. The strong
regularity assumption (that renders the Jacobian invertible) in~\cite{dontchev2013euler} is related to the findings in this paper; notice also that the present work focuses on first-order methods, and does not consider Newton corrector steps. 

With respect to the continuous-time algorithm (cast as a system of differential inclusions), the paper offers contributions over existing works on continuous-time saddle-point flows (see e.g.,~\cite{Cherukuri17,cherukuri2017role,Rahili2015,Elia-CDC11} and pertinent references therein) by considering non-convex problems, providing sufficient conditions for the existence of feasible parameters, and by investigating conditions under which KKT trajectories do not merge or bifurcate.

Examples of works dealing with gradient-based or proximal-based methods for (static) non-convex problems include, for example,~\cite{Carmon18,HongPerturbedPrimalDual18,HongPrimalDual18,Ochs14,HongDistributed18,Zeng18,Martinez12} (see also references therein); in particular,~\cite{HongPerturbedPrimalDual18} considers a perturbed primal-dual method applied to a problem with nonconvex cost and (linear) consensus constraints. The present paper considers a more general case, investigates running algorithms for time-varying settings, and it analyzes  both discrete-time and  continuous-time algorithms.  

The rest of the paper is organized as follows: \Cref{sec:ProblemFormulation} provides more details regarding the problem formulation and presents the discrete-time algorithm; \Cref{sec:MainResult} presents the main convergence results, while \Cref{sec:ContinuousTimeLimit} deals with the continuous-time algorithm and elaborates on the KKT trajectories; \Cref{sec:results} provides illustrative numerical results, and \Cref{sec:conclusions} concludes the paper. Some proofs are provided in the Appendix.

\subsection*{Notation} The following notation is utilized throughout the paper. 

For a twice continuously differentiable real-valued function $f$, the Hessian of $f$ at $x$ will be denoted by $\nabla^2 f(x)$. For a function $f(x,t)$ that is continuously differentiable in the first argument, its gradient with respect to $x$ at $(u,s)$ will be denoted by $\nabla_x f(u,s)$. If $f(x,s)$ is twice differentiable with respect to $x$, its Hessian with respect to $x$ at $(u,s)$ will be denoted by $\nabla^2_{xx} f(u,s)$.

For a continuously differentiable vector-valued function $f$, its Jacobian matrix evaluated at $x$ will be denoted by $J_f(x)$. On the other hand, for vector-valued $f(x,t)$ that is continuously differentiable in the first argument, its Jacobian matrix with respect to $x$ at $(u,s)$ is denote by $J_{f,x}(u,s)$.

Given a convex set $C$, its normal cone at $x\in C$, defined by $\{y: y^T(z-x)\leq 0,\,\forall z\in C\}$, is denoted by $N_C(x)$. The projection operator onto $C$ will be denote by $\mathcal{P}_C$. 

We denote $\mathbb{R}_+=[0,+\infty)$ and $\mathbb{R}_{++}=(0,+\infty)$. The identity matrix is denoted by $I$, or $I_n\in\mathbb{R}^{n\times n}$ when the dimensions need to be specified to avoid confusion. For $x\in \mathbb{R}^n$, the $\ell_2$-norm of $x$ is denoted by $\|x\|:=\sqrt{x^T x}$. The $n$-dimensional closed ball with radius $r$ centered at the origin is denoted by $B_n(r)$.

\section{Problem formulation and algorithm}
\label{sec:ProblemFormulation}

We make the following assumptions regarding the problem \cref{eq:main_problem_cont}.
\begin{assumption}
\label{as:set_and_function}
For the problem \cref{eq:main_problem_cont}, it holds that:
\begin{enumerate}[itemindent=18pt,leftmargin=0pt]
\item[(a)] $\mathcal{X}(t)$ is convex and closed at each $t \in [0, S]$.

\item[(b)] For each $t\in[0,S]$, the functions $c(x, t)$, $f^c(x, t)$ and $f^{nc}(x, t)$ are twice continuously differentiable over $x\in\mathbb{R}^n$. In addition, $\nabla^2_{xx}c(x,t)$, $\nabla^2_{xx}f^c_i(x,t)$ and $\nabla^2_{xx}f^{nc}_i(x,t)$ for each $i=1,\ldots,m$ are continuous over $(x,t)\in\mathbb{R}^n\times[0,S]$.

\item[(c)] There exists a Lipschitz continuous trajectory $z^\ast(t)=\big(x^\ast(t),\lambda^\ast(t)\big),\,t\in[0,S]$ such that
\begin{subequations}
\label{eq:KKT_form1}
\begin{align}
(x^\ast(t),\lambda^\ast(t)) & \in \mathcal{X}(t)\times\mathbb{R}^m_+, \\
\nabla_x c(x^\ast(t),t)
+J_{f,x}(x^\ast(t),t)^T\lambda^\ast(t) & \in -N_{\mathcal{X}(t)}(x^\ast(t)),
\label{eq:KKT_form1:primal}\\
f(x^\ast(t),t) & \in N_{\mathbb{R}^m_+}(\lambda^\ast(t)).
\label{eq:KKT_form1:dual1}
\end{align}
\end{subequations}
for each $t\in[0,S]$.
\end{enumerate}
\end{assumption}
The conditions \cref{eq:KKT_form1} are just the KKT conditions of \cref{eq:main_problem_cont} for each $t\in[0,S]$. Indeed, \cite[Theorem 2A.9]{dontchev2014implicit} shows that, if we assume the following constraint qualification condition for every $t\in[0,S]$:
\begin{equation}\label{eq:cq}
\begin{aligned}
&\textrm{There is no }\lambda\in\mathbb{R}^m_+\backslash\{0\}\textrm{ such that} \\
&\lambda^T f(x^\ast(t),t)=0
\textrm{ and }
-J_{f,x}(x^\ast(t),t)^T\lambda
\in N_{\mathcal{X}(t)}(x^\ast(t)),
\end{aligned}
\end{equation}
where $x^\ast(t)$ is a local optimal solution to \eqref{eq:main_problem_cont}, then there exists an associated optimal Lagrange multiplier $\lambda^\ast(t)$ that satisfy the KKT conditions \cref{eq:KKT_form1} for each $t\in[0,S]$.

\begin{remark}
The constraint qualification given in \cref{eq:cq} is a generalization of the Mangasarian-Fromovitz constraint qualification (MFCQ) \cite{bertsekas1999nonlinear}. Indeed, Gordan's theorem of alternative \cite{bertsekas2003convex} states that $Ax<0$ for a given $A\in\mathbb{R}^{p\times n}$ has a solution $x$ if and only if there is no $y \in \mathbb{R}_+^p\backslash\{0\}$ such that $A^T y=0$. Then in the case where $\mathcal{X}(t)=\mathbb{R}^n$, by Gordan's theorem of alternative, we directly see that the constraint qualification \cref{eq:cq} is equivalent to MFCQ.
\end{remark}
\begin{remark}
In general, there can be multiple KKT points of \cref{eq:main_problem_cont} that move in $\mathbb{R}^{n}\times\mathbb{R}^m_+$ as time proceeds, which form multiple trajectories that can appear, terminate, bifurcate or merge during the period $(0,S)$. Reference \cite{guddat1990parametric} presents a comprehensive theory of the structures and singularities of trajectories of KKT points for time-varying optimization problems. In \cite{dontchev2013euler}, the authors show that strong regularity for generalized equations is a key concept for establishing the existence of Lipschitz continuous KKT trajectories over a given finite period. Here, we arbitrarily select one of these trajectories that is well defined and Lipschitz continuous for $t\in[0,S]$, denote it by $z^\ast(t)$, and mainly focus on this trajectory in most part of our study. How to deal with KKT trajectories that exhibits discontinuities is not considered in this work, and it is the subject of current investigations.  
\end{remark}

As stated in \Cref{sec:introduction}, we sample the time interval $[0,S]$ and get a sequence of problems \cref{eq:main_problem}. Since \cref{as:set_and_function} holds over the whole interval $[0, S]$, it is clear that the sampled quantities in \cref{eq:main_problem} enjoy the same assumptions for each discrete time instant $\tau \in \{1, \ldots, T\}$.

Now we introduce the running regularized primal-dual gradient method for solving the time-varying nonconvex optimization problem \cref{eq:main_problem}. Let $\hat z_0=(\hat x_0,\hat\lambda_0)\in\mathcal{X}(0)\times\mathbb{R}^m_+$ be the initial point. The regularized primal-dual gradient algorithm produces a primal-dual pair $(\hat x_\tau,\hat\lambda_\tau)$ iteratively by
\begin{subequations}\label{eq:RegPPD}
\begin{align}
\hat x_\tau &= \mathcal{P}_{\mathcal{X}_\tau}
\left[
\hat x_{\tau-1}
-\alpha \left(\nabla c_\tau(\hat x_{\tau-1})
+
J_{f_\tau}(\hat x_{\tau-1})^T
\hat\lambda_{\tau-1}\right)
\right],
\label{eq:RegPPD:primal} \\
\hat\lambda_\tau &=
\mathcal{P}_{\mathbb{R}^m_+}\left[
\hat\lambda_{\tau-1}
+\eta\alpha
\left(f_\tau(\hat x_{\tau-1})
-\epsilon\big(\hat\lambda_{\tau-1}-\lambda_{\mathrm{prior}}\big)
\right)\right]
\label{eq:RegPPD:dual1}
\end{align}
\end{subequations}
for each $\tau\in\{1,2,\ldots,T\}$. Here $\alpha>0$, $\eta>0$, $\epsilon>0$ and $\lambda_{\mathrm{prior}}\in\mathbb{R}^m_+$ are parameters of the algorithm.

The regularized primal-dual gradient algorithm is closely related to the following saddle point problem
\begin{equation}\label{eq:saddle_point1}
\min_{x\in\mathcal{X}_\tau}\max_{\lambda\in\mathbb{R}^m_+}
\mathcal{L}^{\epsilon}_\tau(x,\lambda)
\end{equation}
where $\mathcal{L}^{\epsilon}_\tau(x,\lambda)$ denotes the regularized Lagrangian defined by
\begin{equation}\label{eq:reg_lagrange1}
\mathcal{L}^{\epsilon}_\tau(x,\lambda)
= c_\tau(x)
+\lambda^T f_\tau(x)
-\frac{\epsilon}{2}\left\|\lambda
-\lambda_{\mathrm{prior}}\right\|^2.
\end{equation}
Comparing with the ordinary Lagrangian, we can see that there is an additional term $\frac{\epsilon}{2}\left\|\lambda
-\lambda_{\mathrm{prior}}\right\|^2$ in the regularized Lagrangian, which represents regularization that drives the dual variables towards $\lambda_{\mathrm{prior}}$.
Then, \cref{eq:RegPPD} can be viewed as applying a single iteration of projected gradient descent on the primal variable and projected gradient ascent on the dual variables in \cref{eq:reg_lagrange1}, with step sizes being $\alpha$ for the primal update and $\eta\alpha$ for the dual updates respectively. The parameter $\epsilon$ then controls the amount of regularization on the dual variables. The constant vector $\lambda_{\mathrm{prior}}$ can be viewed as a prior estimate of the optimal dual variable; it   can be set to zero if such prior estimates cannot be obtained. The dual update \cref{eq:RegPPD:dual1} can also be equivalently written as
$$
\hat\lambda_{\tau}
=\mathcal{P}_{\mathbb{R}^m_+}
\left[\hat\lambda_{\tau-1}
+\eta\alpha
\left(\tilde f_\tau(\hat x_{\tau-1})
-\epsilon\hat\lambda_{\tau-1}\right)\right]
$$
with $\tilde f_\tau(x)=f_\tau(x)+\epsilon\lambda_{\mathrm{prior}}$, and thus we can also view \cref{eq:RegPPD:dual1} as employing no prior estimate of $\lambda^\ast_\tau$ but a more conservative version of the inequality constraint given by $f_\tau(x)+\epsilon\lambda_{\mathrm{prior}}\leq 0$.

By adding a quadratic regularization on the dual variables, the (regularized) Lagrangian \cref{eq:reg_lagrange1} becomes strongly concave in $\lambda$, which, as we shall show, is important for the regularized primal-dual gradient method to achieve bounded tracking error. On the other hand, the solution to the saddle point problem \cref{eq:saddle_point1} is in general different from the optimal solution to \cref{eq:main_problem}, which suggests that there could be further sub-optimality introduced by regularization in the time-varying setting. This will be  illustrated in \Cref{sec:MainResult}.

We denote the primal-dual pair produced by the proposed algorithm as $\hat z_\tau=(\hat x_\tau,\hat\lambda_\tau)$. Further, we define the norm
\begin{align}
\|z\|_{\eta} := \sqrt{\|x\|^2+\eta^{-1}\|\lambda\|^2}
\end{align}
for $z=(x,\lambda)\in\mathbb{R}^{n+m}$ and given scalar $\eta > 0$.

\section{Main tracking results}
\label{sec:MainResult}
In this section, we study the tracking performance of the regularized primal-dual gradient algorithm.

We define the \emph{tracking error} to be
$$
\left\|\hat z_\tau-z^\ast_\tau\right\|_{\eta}
=\sqrt{\left\|\hat x_\tau-x^\ast_\tau\right\|^2+
\eta^{-1}\big\|\hat \lambda_\tau-\lambda^\ast_\tau\big\|^2},
\qquad \tau=1,2,\ldots,T,
$$
which represents the distance between the true optimal solution and the solution generated by \cref{eq:RegPPD}; a small tracking error implies good tracking performance. We are interested in the factors that affect the tracking error, and especially the conditions under which a bounded tracking error can be guaranteed.

Before proceeding, we first define some quantities that will be used in our analysis. The temporal variability of the KKT pair $z^\ast(t)=(x^\ast(t),\lambda^\ast(t))$ is captured by defining the following quantity: 
\begin{align}
\sigma_{\eta} := \sup_{\substack{t_1,t_2\in[0,S],\\ t_1\neq t_2}}\frac{\|z^\ast(t_2)-z^\ast(t_1)\|_{\eta}}{|t_2-t_1|},
\end{align}
which represents the maximum rate of change of the optimal primal-dual pair with respect to the norm $\|\cdot\|_\eta$. We assume that $\sigma_\eta>0$ for some (and thus for any) $\eta>0$.

We then define
\begin{align}
M_{\lambda} &:= \sup_{t\in[0,S]}\left\|
\lambda^\ast(t)-\lambda_{\mathrm{prior}}\right\|, \\
M_{nc}(\delta) &:= \sup_{t\in[0,S]}\sup_{u:\|u\|\leq\delta} \left\|D^2_{xx}
f^{nc}_t(x^\ast(t)+u,t)\right\|,
\label{eq:def_Mnc}\\
M_{c}(\delta) &:= \sup_{t\in[0,S]}\sup_{u:\|u\|\leq\delta} \left\|D^2_{xx} f^c(x^\ast(t)+u,t)\right\|, \\
L_f(\delta) &:= \sup_{t\in[0,S]}\sup_{u:\|u\|\leq\delta} \left\|
J_{f,x}(x^\ast(t)+u,t)
\right\|, \\
D(\delta,\eta)&:=
\sqrt{\eta}L_f(\delta)+M_{c}(\delta)\sup_{t\in[0,S]}\left\|\lambda^\ast(t)\right\|.
\label{eq:def_capD}
\end{align}
Here for a vector-valued $f(x,t)$ that is twice continuously differentiable in $x$ for a fixed $t$, we use $D^2_{xx} f(x,t)$ to denote the bilinear map that maps a pair of vectors $(h_1,h_2)$ to a vector whose $i$'th entry is given by $h_2^T\nabla^2_{xx} f_i(x,t)h_1$, i.e.,
$$
D^2_{xx} f(x,t):\quad (h_1,h_2)\mapsto \left(h_2^T\nabla^2_{xx} f_i(x,t)h_1\right)_{i},
$$
and $\left\|D^2_{xx} f(x,t)\right\|$ is defined by
$$
\left\|D^2_{xx} f(x,t)\right\| := \sup_{h_1,h_2\neq 0} \frac{\left\|D^2_{xx} f(x,t)(h_1,h_2)\right\|}{\|h_1\|\|h_2\|}
=\sup_{\|h_1\|=\|h_2\|=1} \left\|D^2_{xx} f(x,t)(h_1,h_2)\right\|.
$$
The following mean-value type lemma indicates that $\left\|D^2_{xx} f(x,t)\right\|$ characterizes the nonlinearity of function $f$ with respect to $x$ at time $t$.
\begin{lemma}\label{lemma:curve_lemma}
Let $t\in[0,S]$ be fixed. Then for any $x$ and $u$ in $\mathbb{R}^n$,
\begin{equation}\label{eq:curve_bound_J}
\frac{\left\|J_{f,x}(x+u,t)-J_{f,x}(x,t)\right\|}{\|u\|}
\leq
\sup_{\theta\in[0,1]}
\left\|D^2_{xx} f(x+\theta u,t)\right\|,
\end{equation}
and
\begin{equation}\label{eq:curve_bound_f}
\frac{\left\|f(x+u,t)-f(x,t)-J_{f,x}(x,t)u\right\|}{\|u\|^2}
\leq
\frac{1}{2}\sup_{\theta\in[0,1]}
\left\|D^2_{xx} f(x+\theta u,t)\right\|.
\end{equation}
\end{lemma}
The argument $\delta\in\mathbb{R}_{++}$ in the definitions \cref{eq:def_Mnc}--\cref{eq:def_capD} represents the radius of the ball centered at the KKT point, i.e., the local region around the KKT point we are interested in.

Let the ``nonconvex part'' of the Lagrangian be
$$
\begin{aligned}
\mathcal{L}^{nc}(x,\lambda,t) &:=c(x,t)+\lambda^T f^{nc}(x,t),
&\qquad &t\in[0,S], \\
\mathcal{L}^{nc}_\tau(x,\lambda) &:=
\mathcal{L}^{nc}(x,\lambda,\tau\Delta_T),
&\qquad &\tau=1,2,\ldots,T.
\end{aligned}
$$
We also define
\begin{align}
\overline H_{\mathcal{L}^{nc}}(u,t)
&:=\int_0^1 \nabla_{xx}^2 \mathcal{L}^{nc}(x^\ast(t)+\theta u,\lambda^\ast(t),t)\,d\theta,
\label{eq:def_HL}\\
\overline{H}_{f^c_{i}}(u,t)
&:=\int_0^1 2(1-\theta)\nabla^2_{xx} f^c_{i}(x^\ast (t) +\theta u,t)\,d\theta,
\label{eq:def_Hfc}
\end{align}
and
$$
\begin{aligned}
\rho^{(\mathrm{P})}(\delta,\alpha,\eta,\epsilon) &:= \sup_{\substack{t\in[0,S], \\ u:\|u\|\leq\delta}}
\left\|
\left(I-\alpha\overline H_{\mathcal{L}^{nc}}(u,t)\right)^2
-\alpha(1-\eta\alpha\epsilon)\sum_{i=1}^m \lambda^\ast_{i}(t)\overline H_{f^c_{i}}(u,t)
\right\|, \\
\rho(\delta,\alpha,\eta,\epsilon) &:= \Bigg[
\max\left\{
\rho^{(\mathrm{P})}(\delta,\alpha,\eta,\epsilon),
(1-\eta\alpha\epsilon)^2
\right\}+ 
\alpha(1-\eta\alpha\epsilon)
\frac{\sqrt{\eta}\delta M_{nc}(\delta)}{2} \\
&\qquad+\alpha^2\,
\Bigg(2\sup_{\substack{t\in[0,S], \\u:\|u\|\leq\delta}}\left\|\eta\epsilon I-\overline H_{\mathcal{L}^{nc}}(u,t)\right\|D(\delta,\eta)
+D^2(\delta,\eta)\Bigg)\Bigg]^{1/2}, \\
\kappa(\delta,\alpha,\eta,\epsilon)
&:=\max\left\{
1,
\, \frac{1-\eta\alpha\epsilon}{\rho(\delta,\alpha,\eta,\epsilon)},
\, \frac{\sqrt{\eta} \alpha L_f(\delta)}{\rho(\delta,\alpha,\eta,\epsilon)}
\right\}.
\end{aligned}
$$

Now we present the main result regarding the tracking error bound of the regularized primal-dual gradient algorithm \cref{eq:RegPPD}.
\begin{theorem}\label{theorem:tracking_performance}
Suppose there exist parameters $\delta>0$, $\alpha>0$, $\eta>0$ and $\epsilon>0$ such that
\begin{subequations}\label{eq:main_thm_cond}
\begin{equation}\label{eq:main_thm_cond1}
\sigma_{\eta}\Delta_T \leq \left(1-\rho(\delta,\alpha,\eta,\epsilon)\right)\delta 
-\kappa(\delta,\alpha,\eta,\epsilon) \sqrt{\eta}\alpha\epsilon M_\lambda.
\end{equation}
Let the initial point $\hat z_0=(\hat x_0,\hat\lambda_0)$ be sufficiently close to $z^\ast_1$ so that
\begin{equation}\label{eq:main_thm_cond2_a}
\left\|\hat z_0-z^\ast_1\right\|_{\eta}\leq \delta.
\end{equation}
\end{subequations}
Then the sequence $(\hat z_\tau)_{\tau=1,\ldots,T}$ produced by the regularized primal-dual gradient algorithm \cref{eq:RegPPD} satisfies
\begin{equation}\label{eq:main_tracking_error}
\begin{aligned}
\left\|\hat z_\tau-z^\ast_\tau\right\|_{\eta}
\leq\ &
\frac{
\rho(\delta,\alpha,\eta,\epsilon)\, \sigma_{\eta}\Delta_T
+\kappa(\delta,\alpha,\eta,\epsilon)\sqrt{\eta}\alpha\epsilon M_\lambda
}{1-\rho(\delta,\alpha,\eta,\epsilon)} \\
&+\rho^\tau(\delta,\alpha,\eta,\epsilon)
\left(\left\|\hat z_0-z^\ast_1\right\|_{\eta}
-\frac{ \sigma_{\eta}\Delta_T
+\kappa(\delta,\alpha,\eta,\epsilon)\sqrt{\eta}\alpha\epsilon M_\lambda
}{1-\rho(\delta,\alpha,\eta,\epsilon)}\right)
\end{aligned}
\end{equation}
for all $\tau=1,2,\ldots,T$.

Moreover, we have
$$
\lim_{\alpha\rightarrow 0^+} \kappa\left(\delta,\alpha,\eta,\epsilon\right)=1
\qquad
\textrm{and}\qquad
\kappa\left(\delta,\alpha,\eta,\epsilon\right) \leq\sqrt{2}.
$$
\end{theorem}

The proof of \cref{theorem:tracking_performance} is based on the following lemma whose proof is postponed to \Cref{sec:proof_approx_contraction}.
\begin{lemma}\label{lemma:approx_contraction}
Let $\tau\in\{1,2,\ldots,T\}$ be arbitrary. If $\left\|\hat z_{\tau-1}-z^\ast_\tau\right\|_{\eta}\leq\delta$ and $\hat z_{\tau}$ is generated by \cref{eq:RegPPD}, then
\begin{equation}
\label{eq:approx_contraction}
\left\|\hat z_{\tau}-z^\ast_\tau\right\|_{\eta}
\leq \rho(\delta,\alpha,\eta,\epsilon)
\left\|\hat z_{\tau-1}-z^\ast_\tau\right\|_{\eta}
+\kappa(\delta,\alpha,\eta,\epsilon)
\sqrt{\eta}\alpha\epsilon M_{\lambda},
\end{equation}
where $\kappa(\delta,\alpha,\eta,\epsilon)$ is upper bounded by $\sqrt{2}$ and satisfies
$$
\lim_{\alpha\rightarrow 0^+}\kappa(\delta,\alpha,\eta,\epsilon)=1.
$$
\end{lemma}

Note that \cref{lemma:approx_contraction} asserts that if the radius $\delta$ is chosen such that $\rho := \rho(\delta,\alpha,\eta,\epsilon) < 1$, iteration \cref{eq:RegPPD} is an approximate local contraction with  coefficient $\rho$ and error term $e := \kappa(\delta,\alpha,\eta,\epsilon)
\sqrt{\eta}\alpha\epsilon M_{\lambda}$. The idea of the proof of \cref{theorem:tracking_performance} is then based on showing that condition \cref{eq:main_thm_cond1} is sufficient to guarantee that the trajectory generated by the algorithm is confined within the contraction region at every time step. Note that \cref{eq:main_thm_cond1} implies 
\begin{equation} \label{eq:ball_contained}
B(z^\ast_\tau, \rho \delta + e) \subseteq B(z^\ast_{\tau+1}, \delta),
\end{equation}
where $B(z, \delta)$ is the ball centered at $z$ with radius $\delta$ (with respect to the norm $\|\cdot \|_{\eta}$). Then, it is possible to show by induction that if $\hat z_0 \in B(z^\ast_{1}, \delta)$, then $\hat z_{\tau-1} \in B(z^\ast_{\tau}, \delta)$ for all $\tau$. The  idea of condition \cref{eq:ball_contained} is illustrated in \cref{fig:approx_contraction} and the formal proof of \cref{theorem:tracking_performance} is given below.

\begin{figure}[tbhp]
    \centering
    \includegraphics[width=.4\textwidth]{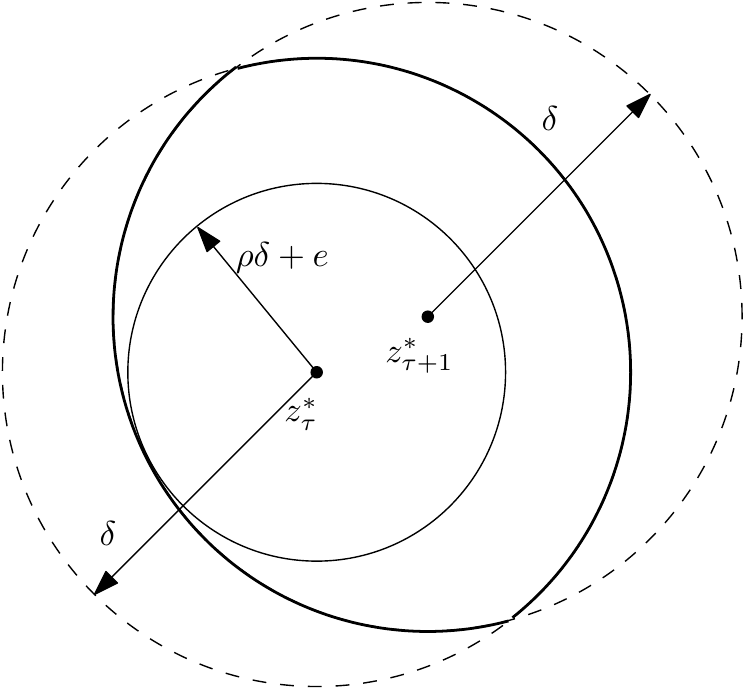}
    \caption{Illustration of condition \cref{eq:ball_contained}. This condition is essentially on: (i) time-variability of a KKT point, (ii) extent of the contraction, and (iii) maximum error in each iteration. Note that in the error-less static case (i.e., $\sigma_\eta = e  = 0$), this condition is trivially satisfied.}
    \label{fig:approx_contraction} 
\end{figure}

\begin{proof}[Proof of \cref{theorem:tracking_performance}]

To simplify the notation, we will just use $\rho$ to denote $\rho(\delta,\alpha,\eta,\epsilon)$ and use $\kappa$ to denote $\kappa(\delta,\alpha,\eta,\epsilon)$. The condition \cref{eq:main_thm_cond2_a} guarantees that we can use \cref{lemma:approx_contraction} to get
$$
\left\|\hat z_1-z^\ast_1\right\|_{\eta}
\leq
\rho\left\|\hat z_0-z^\ast_1\right\|_{\eta}
+\kappa\sqrt{\eta}\alpha\epsilon M_{\lambda},
$$
which shows that \cref{eq:main_tracking_error} holds for $\tau=1$. Now suppose that
\cref{eq:main_tracking_error} holds for some $\tau\in\{1,2,\ldots,T\}$. We have
\begin{equation}\label{eq:proof_main_tracking_error:temp1}
\begin{aligned}
&\left\|\hat z_\tau-z^\ast_{\tau+1}\right\|_{\eta}
\leq \left\|\hat z_\tau-z^\ast_{\tau}\right\|_{\eta}
+\left\|z^\ast_\tau-z^\ast_{\tau+1}\right\|_{\eta}
 \\
\leq\ &
\left(1-\rho^\tau\right)\frac{
\rho\sigma_{\eta}\Delta_T
+\kappa\sqrt{\eta}\alpha\epsilon M_{\lambda}
}{1-\rho}
+\rho^\tau
\left(\left\|\hat z_0-z^\ast_1\right\|_{\eta}
-\sigma_{\eta}\Delta_T\right)
+\sigma_{\eta}\Delta_T
 \\
=\ &
\left(1-\rho^\tau\right)\frac{\sigma_{\eta}\Delta_T
+\kappa\sqrt{\eta}\alpha\epsilon M_{\lambda}
}{1-\rho}
+\rho^\tau
\left\|\hat z_0-z^\ast_1\right\|_{\eta}.
\end{aligned}
\end{equation}
By the conditions \cref{eq:main_thm_cond}, we get
$$
\left\|\hat z_\tau-z^\ast_{\tau+1}\right\|_{\eta}
\leq
\left(1-\rho^\tau\right)\delta
+\rho^\tau\delta
\leq\delta.
$$
We can then use \cref{lemma:approx_contraction} and \cref{eq:proof_main_tracking_error:temp1} to get
$$
\begin{aligned}
&\left\|\hat z_{\tau+1}-z^\ast_{\tau+1}\right\|_{\eta}
\leq \rho\left\|\hat z_\tau-z^\ast_{\tau+1}\right\|_{\eta}
+\kappa\sqrt{\eta}\alpha\epsilon M_{\lambda} \\
\leq\ &
\rho
\left(\left(1-\rho^\tau\right)\frac{\sigma_{\eta}\Delta
+\kappa\sqrt{\eta}\alpha\epsilon M_{\lambda}
}{1-\rho}
+\rho^\tau
\left\|\hat z_0-z^\ast_1\right\|_{\eta}\right)
+\kappa\sqrt{\eta}\alpha\epsilon M_{\lambda} \\
=\ &
\frac{\rho\sigma_{\eta}\Delta_T
+\kappa\sqrt{\eta}\alpha\epsilon M_{\lambda}
}{1-\rho}
+\rho^{\tau+1}
\left(\left\|\hat z_0-z^\ast_1\right\|_{\eta}
-\frac{\sigma_{\eta}\Delta_T
+\kappa\sqrt{\eta}\alpha\epsilon M_{\lambda}
}{1-\rho}\right),
\end{aligned}
$$
and by induction we get \cref{eq:main_tracking_error} for all $\tau\in\{1,2,\ldots,T\}$.
\end{proof}

It can be seen that the bound on the tracking error in \cref{eq:main_tracking_error} consists of a constant term and a term that decays geometrically with $\tau$, as the condition \cref{eq:main_thm_cond1} implies $\rho\left(\delta,\alpha,\eta,\epsilon\right)<1$. We call the constant term
\begin{equation}\label{eq:main_tracking_error2}
\frac{
\rho(\delta,\alpha,\eta,\epsilon)\, \sigma_{\eta}\Delta_T
+\kappa(\delta,\alpha,\eta,\epsilon)\sqrt{\eta}\alpha\epsilon M_\lambda
}{1-\rho(\delta,\alpha,\eta,\epsilon)}
\end{equation}
the \emph{eventual tracking error bound}, which can be further split into two parts:

\noindent (a) The first part
$$
\frac{
\rho(\delta,\alpha,\eta,\epsilon)
}{1-\rho(\delta,\alpha,\eta,\epsilon)}\sigma_{\eta}\Delta_T
$$
is proportional to $\sigma_\eta$, the maximum speed of the optimal primal-dual pair. In time-varying optimization, such terms are common in the tracking error bound \cite{dallanese2016optimal,tang2017real,simonetto2017time,mokhtari2016online}. This term also decreases as one reduces the sampling interval $\Delta_T$.

\noindent (b) The second part
$$
\frac{\kappa(\delta,\alpha,\eta,\epsilon)\sqrt{\eta}\alpha\epsilon M_\lambda
}{1-\rho(\delta,\alpha,\eta,\epsilon)}
$$
is proportional to $M_\lambda$, the maximum distance between the optimal Lagrange multiplier $\lambda^\ast(t)$ and the prior estimate $\lambda_{\mathrm{prior}}$. This term represents the discrepancy introduced by adding regularization on the dual variable; similar behavior has also been observed in \cite{koshal2011multiuser}.

In addition, the first part has a multiplicative factor $\rho(\delta,\alpha,\eta,\epsilon)/(1-\rho(\delta,\alpha,\eta,\epsilon))$, while the second part has a multiplicative factor $1/(1-\rho(\delta,\alpha,\eta,\epsilon))$, which are all strictly increasing in $\rho(\delta,\alpha,\eta,\epsilon)$. This implies that a smaller $\rho(\delta,\alpha,\eta,\epsilon)$ will lead to better tracking performance. The condition \cref{eq:main_thm_cond1} is also more likely to be satisfied when $\rho(\delta,\alpha,\eta,\epsilon)$ is smaller.

\subsection{Feasible algorithmic parameters}
\label{subsec:FeasibleAlgorithmicParameters}

In order that \cref{eq:main_thm_cond1} can be satisfied and the bound \cref{eq:main_tracking_error2} can be as small as possible, one needs to find an appropriate set of the parameters $\alpha,\eta,\epsilon$. However, the expression that defines $\rho(\delta,\alpha,\eta,\epsilon)$ is rather complicated, making it difficult to analyze how to achieve a smaller bound \cref{eq:main_tracking_error2}; even the existence of parameters that can guarantee the condition \cref{eq:main_thm_cond1} is not readily available. In this section, we give a preliminary study of the conditions under which \cref{eq:main_thm_cond1} can be satisfied.

\begin{definition}
We say that $(\delta,\alpha,\eta,\epsilon)\in\mathbb{R}_{++}^4$ is a tuple of \emph{feasible parameters} if
\begin{equation}\label{eq:def_feas_param}
\left(1-\rho(\delta,\alpha,\eta,\epsilon)\right)\delta
-\kappa(\delta,\alpha,\eta,\epsilon)\sqrt{\eta}\alpha\epsilon M_\lambda>0.
\end{equation}
\end{definition}
It can be seen that, if \cref{eq:def_feas_param} is satisfied, then for sufficiently small sampling interval $\Delta_T$ the condition \cref{eq:main_thm_cond1} can be satisfied; otherwise \cref{eq:main_thm_cond1} cannot be satisfied no matter how one reduces the sampling interval $\Delta_T$. The quantity $\delta$ has been added to the tuple of parameters for convenience.

Define
\begin{equation}\label{eq:def_Lambda_m}
\Lambda_m(\delta):= \inf_{t\in[0,S]}\inf_{u:\|u\|\leq\delta}\lambda_{\min}\left(
\overline{H}_{\mathcal{L}
^{nc}}(u,t)+\frac{1}{2}\sum_{i=1}^m\lambda^\ast_{i}(t)\overline{H}_{f^c_{i}}(u,t)\right).
\end{equation}
Roughly speaking, $\Lambda_m(\delta)$ characterizes how convex the problem \cref{eq:main_problem} is in the neighborbood of the optimal point. It's easy to see that $\Lambda_m(\delta)$ is a nonincreasing function in $\delta$.

\begin{theorem}\label{theorem:feas_param}
Suppose there exists some $\bar\delta>0$ such that
\begin{equation}\label{eq:feas_param_cond}
\Lambda_m\big(\bar\delta\big)>M_\lambda\, M_{nc}\big(\bar\delta\big).
\end{equation}
Let
$$
S_{\mathrm{fp}}:=\left\{
(\delta,\alpha,\eta,\epsilon)\in\mathbb{R}_{++}^4:
\left(1-\rho(\delta,\alpha,\eta,\epsilon)\right)\delta
-\kappa(\delta,\alpha,\eta,\epsilon)\sqrt{\eta}\alpha\epsilon M_\lambda>0\right\}.
$$
Then $S_{\mathrm{fp}}$ is a non-empty open subset of $\mathbb{R}_{++}^4$.
\end{theorem}
\begin{proof}
We first introduce an auxiliary lemma, whose proof is given in \Cref{sec:proof_sup_continuous}.
\begin{lemma}\label{lemma:sup_continuous}
The following two statements hold:
\begin{enumerate}[itemindent=13.5pt,leftmargin=0pt]
\item Suppose $X$ and $Y$ are metric spaces with $X$ being compact, and $f:X\times Y\rightarrow\mathbb{R}$ is a continuous function. Let $g:Y\rightarrow\mathbb{R}$ be given by
$$
g(y)=\sup_{x\in X}f(x,y).
$$
Then $g$ is a continuous function.
\item Suppose $f:B_n(R)\times V\rightarrow\mathbb{R}$ is a continuous function for some $R>0$ and $V$ is a metric space. Define $g:(0,R)\times V\rightarrow\mathbb{R}$ by
$$
g(r,v)=\sup_{u:\|u\|\leq r} f(u,v)
$$
Then $g$ is a continuous function on $(0,R)\times V$.
\end{enumerate}
\end{lemma}

Denote
$$
\Lambda_M(\delta):= \sup_{t\in[0,S]}\sup_{u:\|u\|\leq\delta}\lambda_{\max}\left(
\overline{H}_{\mathcal{L}^{nc}}(u,t)+\frac{1}{2}\sum_{i=1}^m\lambda^\ast_{i}(t)\overline{H}_{f^c_{i}}(u,t)\right).
$$
Let $R>\bar\delta$ be arbitrary. We first show that $\overline{H}_{\mathcal{L}^{nc}}(u,t)$ and each $\overline{H}_{f^c_i}(u,t)$ are continuous over $(u,t)\in {B}_n(R)\times[0,S]$. Indeed, by their definitions \cref{eq:def_HL} and \cref{eq:def_Hfc}, any entry of these matrices can be written in the form
$$
\int_{0}^1 g(\theta,u,t)\,d\theta,
$$
where $g$ is some continuous function over $(\theta,u,t)\in[0,1]\times {B}_n(R)\times[0,S]$ as we have assumed the continuity of $\nabla^2_{xx} c(x,t)$ and each $\nabla^2_{xx}f^c_i(x,t)$, $\nabla^2_{xx}f^{nc}_i(x,t)$, $\nabla^2_{xx}f^{eq}_j(x,t)$. Since $[0,1]\times {B}_n(R)\times[0,S]$ is a compact set, $g(\theta,u,t)$ is bounded above. By the dominated convergence theorem, $\int_{0}^1 g(\theta,u,t)\,d\theta$ is continuous over $(u,t)\in {B}_n(R)\times[0,S]$. 

As a consequence of the continuity of $\overline{H}_{\mathcal{L}^{nc}}(u,t)$ and each $\overline{H}_{f^c_i}(u,t)$, $\Lambda_M(\delta)$ is finite for all $\delta\leq R$.

Let
$$
f_R(\delta,\alpha,\eta,\epsilon)
:=\left(1-\rho(\delta,\alpha,\eta,\epsilon\right))\delta
-\kappa(\delta,\alpha,\eta,\epsilon)\sqrt{\eta}\alpha\epsilon M_{\lambda}
$$
for $\left(\delta,\alpha,\eta,\epsilon\right) \in (0,R)\times\mathbb{R}_{++}^3$, and let us consider the Taylor expansion of $f_R(\delta,\alpha,\eta,\epsilon)$ with respect to $\alpha$ as $\alpha\rightarrow 0^+$. We have
\begin{equation}\label{eq:proof_feas_param_rhomat}
\begin{aligned}
&\left(I-\alpha\overline{H}_{\mathcal{L}^{nc}}(u,t)\right)^2-\alpha(1-\eta\alpha\epsilon)\sum_{i=1}^m\lambda^\ast_{i}(t)\overline{H}_{f^c_{i}}(u,t) \\
=\,&
I-2\alpha\left(\overline{H}_{\mathcal{L}^{nc}}(u,t)+\frac{1}{2}\sum_{i=1}^m\lambda^\ast_{i}(t)\overline{H}_{f^c_{i}}(u,t)\right)+ \alpha^2Q(u,t,\eta,\epsilon),
\end{aligned}
\end{equation}
where $Q(u,t,\eta,\epsilon)$ is some positive semidefinite matrix that depends continuously on $(u,t,\eta,\epsilon)$. It can be checked that for $\alpha
<\left(2\Lambda_M(R)\right)^{-1}$, we have
$$
0\prec 2\alpha\left(\overline{H}_{\mathcal{L}^{nc}}(u,t)+\frac{1}{2}\sum_{i=1}^m\lambda^\ast_{i}(t)\overline{H}_{f^c_{i}}(u,t)\right)
\prec I
$$
whenever $\|u\|<R$, and so
$$
\begin{aligned}
&\left\|I-2\alpha\left(\overline{H}_{\mathcal{L}^{nc}}(u,t)+\frac{1}{2}\sum_{i=1}^m\lambda^\ast_{i}(t)\overline{H}_{f^c_{i}}(u,t)\right)\right\| \\
=\ &
1-2\alpha
\lambda_{\min}\left(\overline{H}_{\mathcal{L}^{nc}}(u,t)+\frac{1}{2}\sum_{i=1}^m\lambda^\ast_{i}(t)\overline{H}_{f^c_{i}}(u,t)\right)
\end{aligned}
$$
whenever $\|u\|< R$ and $\alpha
<\left(2\Lambda_M(R)\right)^{-1}$. By \cref{eq:proof_feas_param_rhomat}, we get
$$
\begin{aligned}
& \sup_{t\in[0,S]}\sup_{u:\|u\|\leq\delta}\left\|\left(I-\alpha\overline{H}_{\mathcal{L}^{nc}}(u,t)\right)^2-\alpha(1-\eta\alpha\epsilon)\sum_{i=1}^m\lambda^\ast_{i}(t)\overline{H}_{f^c_{i}}(u,t)\right\| \\
=\, & 
1-2\alpha\Lambda_m(\delta)
+O\big(\alpha^2\big)
\end{aligned}
$$
as $\alpha\rightarrow 0^+$ for any $\delta\in(0,R)$. Consequently, by Taylor expansion of $\rho\left(\delta,\alpha,\eta,\epsilon\right)$ with respect to $\alpha$ and using the properties of $\kappa\left(\delta,\alpha,\eta,\epsilon\right)$, we can show that
\begin{equation}\label{eq:proof_feas_param:taylor_rho}
\rho(\delta,\alpha,\eta,\epsilon)
=1-\alpha\left(\min\left\{\Lambda_m(\delta),\eta\epsilon\right\}-\frac{\sqrt{\eta}}{4}\delta M_{nc}(\delta)\right)
+O\big(\alpha^2\big)
\end{equation}
and
\begin{equation}\label{eq:proof_feas_param:taylor}
 f_R(\delta,\alpha,\eta,\epsilon)
= 
\alpha\left(\delta\left(\min\left\{\Lambda_m(\delta),\eta\epsilon\right\}-\frac{\sqrt{\eta}}{4}\delta M_{nc}(\delta)\right)
-\sqrt{\eta}\epsilon M_{\lambda}\right)+O\big(\alpha^2\big).
\end{equation}

Now we consider two cases:
\begin{enumerate}[itemindent=13.5pt,leftmargin=0pt]
\item $M_{\lambda}>0$: Let $\delta_0=\bar\delta$ and
$$
\begin{aligned}
\eta_0 =\left(\frac{2\Lambda_m(\delta_0)}{\delta_0 M_{nc}(\delta_0)}\right)^2,\qquad 
\epsilon_0 = \frac{\Lambda_m(\delta_0)}{\eta_0}.
\end{aligned}
$$
By \cref{eq:proof_feas_param:taylor},
\begin{equation}\label{eq:proof_param_feas:case1}
\begin{aligned}
f_R(\delta_0,\alpha,\eta_0,\epsilon_0\big)
=\, &
\alpha\frac{\bar\delta}{2} \left(
\Lambda_m\big(\bar\delta\big)-M_{\lambda} M_{nc}\big(\bar\delta\big)
\right)+O\big(\alpha^2\big).
\end{aligned}
\end{equation}
Togeher with the condition \cref{eq:feas_param_cond}, we can see from \cref{eq:proof_param_feas:case1} there exists a sufficiently small $\alpha_0>0$ such that $f_R(\delta_0,\alpha_0,\eta_0,\epsilon_0)$ is positive, and consequently $S_{\mathrm{fp}}$ is non-empty.

\item $M_{\lambda}=0$: Let $\eta_0>0$ be arbitrary. Consider the function
$$
g(\delta)=\Lambda_m(\delta)-\frac{\sqrt{\eta_0}}{4}\delta M_{nc}(\delta).
$$
by the monotonicity of $\Lambda_m(\delta)$ and $M_{nc}(\delta)$,
$$
\lim_{\delta\rightarrow 0^+} g(\delta)=\lim_{\delta\rightarrow 0^+}\Lambda_m(\delta)\geq\Lambda_m\left(\bar\delta\right)>0.
$$
Therefore there exists some $\delta_0\in\left(0,\bar\delta\right]$ such that $g(\delta_0)>0$.

Now let
$\epsilon_0=\eta_0^{-1}\Lambda_m(\delta_0)$. 
By \cref{eq:proof_feas_param:taylor},
\begin{equation}\label{eq:proof_param_feas:case2}
\begin{aligned}
f_R(\delta_0,\alpha,\eta_0,\epsilon_0)
=\, &
\alpha\delta_0 g(\delta_0) +O\big(\alpha^2\big).
\end{aligned}
\end{equation}
Therefore we can find some $\alpha_0>0$ such that $f_R(\delta_0,\alpha_0,\eta_0,\epsilon_0)$ is positive, and consequently $S_{\mathrm{fp}}$ is non-empty.
\end{enumerate}

Finally, by \cref{lemma:sup_continuous}, it can be seen that $f_R(\delta,\alpha,\eta,\epsilon)$ is a continuous function over $(\delta,\alpha,\eta,\epsilon)\in(0,R)\times\mathbb{R}_{++}^3$. Therefore the set
$$
S_{\mathrm{fp}}
\cap\left((0,R)\times\mathbb{R}_{++}^3\right)=
\{(\delta,\alpha,\eta,\epsilon)\in(0,R)\times\mathbb{R}_{++}^3: f_R(\delta,\alpha,\eta,\epsilon)>0\}
$$
is an open subset of $\mathbb{R}_{++}^4$, and consequently
$$
S_{\mathrm{fp}}
=\bigcup_{R>\bar\delta}
S_{\mathrm{fp}}
\cap\left((0,R)\times\mathbb{R}_{++}^3\right)
$$
is an open subset of $\mathbb{R}_{++}^4$.
\end{proof}

The condition \cref{eq:feas_param_cond} for the existence of feasible parameters can be intuitively interpreted as follows: The problem should be sufficiently convex around the optimal trajectory to overcome the nonlinearity of the nonconvex constraints. It should be emphasized that this is only a sufficient condition.

In the proof of \cref{theorem:feas_param}, we consider the asymptotic behavior of $\rho(\delta,\alpha,\eta,\epsilon)$ as the step size $\alpha$ approaches zero, which greatly helps simplify the analysis. It is well known that when the step size is very small, the classical projected gradient descent can be viewed as good approximation of the continuous-time projected gradient flow \cite{nagurney2012projected} which has simpler analysis but still provides valuable results for understanding the discrete-time counterpart. This observation suggests that by studying the continuous-time limit of \cref{eq:RegPPD}, we may get a better understanding of the discrete-time algorithm.

\section{Continuous-time limit}
\label{sec:ContinuousTimeLimit}

In this section, we study the continuous-time limit of the regularized primal-dual gradient algorithm \cref{eq:RegPPD}.

Let $\hat z_0=(\hat x_0,\hat\lambda_0)$ be the initial primal-dual pair. For each $T\in\mathbb{N}$, let $\hat z_\tau^{(T)}=\big(\hat x_\tau^{(T)},\hat\lambda_\tau^{(T)}\big),\,\tau=0,1,\ldots,T$ be the sequence generated by \cref{eq:RegPPD} with sampling interval $\Delta_T=S/T$ and step size $\alpha=\Delta_T\beta$ for some fixed $\beta>0$. In other words, we let $\hat z^{(T)}_0=\hat z_0$ and $\hat z_\tau^{(T)}=\big(\hat x_\tau^{(T)},\hat\lambda_\tau^{(T)}\big)$ with
\begin{subequations}\label{eq:RegPPD_cont_setting}
\begin{align}
\hat x^{(T)}_\tau &= \mathcal{P}_{\mathcal{X}_\tau}
\left[
\hat x^{(T)}_{\tau-1}
-\Delta_T\beta\left(\nabla c_\tau\big(\hat x^{(T)}_{\tau-1}\big)
+
J_{f_\tau}\big(\hat x^{(T)}_{\tau-1}\big)^T
\lambda^{(T)}_{\tau-1}
\right)
\right], \\
\hat\lambda^{(T)}_\tau &=
\mathcal{P}_{\mathbb{R}^m_+}\left[
\hat\lambda^{(T)}_{\tau-1}
+\Delta_T\eta\beta
\left(f_\tau\big(\hat x^{(T)}_{\tau-1}\big)
-\epsilon\big(\hat\lambda^{(T)}_{\tau-1}-\lambda_{\mathrm{prior}}\big)
\right)\right],
\end{align}
\end{subequations}
for $\tau=1,2,\ldots,T$. Now we define the linear interpolation of $\big(z^{(T)}_\tau\big)_{\tau=1,\ldots,T}$ by
\begin{equation}\label{eq:discrete_interp_path}
\hat z^{(T)}(t)
=\frac{\tau\Delta_T -t}{\Delta_T}\hat z^{(T)}_{\tau-1} + \frac{t-(\tau-1)\Delta_T}{\Delta_T}\hat z^{(T)}_{\tau}
\end{equation}
if $t\in[(\tau-1)\Delta_T,\tau\Delta_T]$ for each $\tau=1,2,\ldots,T$. We are interested in the behavior of $\hat z^{(T)}(t)$ when $T\rightarrow\infty$.

Regarding the time-varying set $\mathcal{X}(t)$, we introduce the following definition on Lipschitz set-value maps. 
\begin{definition}
We say that the set-valued map $\mathcal{X}:[0,S]\rightarrow 2^{\mathbb{R}^n}$ is $\kappa$-Lipschitz if
$$
d_H(\mathcal{X}(t_1),\mathcal{X}(t_2))\leq \kappa|t_1-t_2|\quad\forall t_1,t_2\in[0,S],
$$
where $d_H$ denotes the Hausdorff distance.
\end{definition}

The following additional assumptions are then imposed.
\begin{assumption}
\label{as:set_lips}
\begin{enumerate}[itemindent=18pt,leftmargin=0pt]
\item[(a)] $\mathcal{X}: [0,S] \rightarrow  2^{\mathbb{R}^n}$ is a $\kappa_1$-Lipschitz set-valued map.
\item[(b)] The gradient $\nabla_x c(x,t)$ is continuous over $(x,t)\in {\bigcup_{t\in[0,S]}\mathcal{X}(t)\times[0,S]}$, and there exists $\kappa_2>0$ such that
\begin{equation}
\label{eq:cont_time_limit:nabla_c_linear_growth}
\|\nabla_x c(x,t)\|\leq \kappa_2(1+\|x\|),\qquad 
\forall (x,t)\in \bigcup_{t\in[0,S]}\mathcal{X}(t)\times[0,S].
\end{equation}
\item[(c)] $f(x,t)$ is continuous over $(x,t)\in\bigcup_{t\in[0,S]}\mathcal{X}(t)\times[0,S]$, and $J_{f,x}(x,t)$ is bounded and continuous over $(x,t)\in\bigcup_{t\in[0,S]}\mathcal{X}(t)\times[0,S]$.
\end{enumerate}
\end{assumption}

The next theorem formulates the continuous-time limit of the discrete-time algorithm \cref{eq:RegPPD}.
\begin{theorem}\label{theorem:cont_time_limit}
Let $T\rightarrow\infty$ while keeping $S$ constant. Then the sequence of trajectories $\big(\hat z^{(T)}\big)_{T\in\mathbb{N}}$ defined by \cref{eq:discrete_interp_path} has a convergent subsequence, and any convergent subsequence converges uniformly to some Lipschitz continuous $\hat z=\big(\hat x,\hat\lambda\big)$ that satisfies
\begin{subequations}\label{eq:cont_time_limit}
\begin{align}
-\frac{d}{dt}\hat x(t)-\beta\left(\nabla_x c(\hat x(t),t)+
J_{f,x}(\hat x(t),t)^T
\hat\lambda(t)\right)
&\in N_{\mathcal{X}(t)}(\hat x(t)),
\label{eq:cont_time_limit:primal}\\
-\frac{d}{dt}\hat\lambda(t)+
\eta\beta
\left(f(\hat x(t),t)-\epsilon\big(\hat\lambda(t)-\lambda_{\mathrm{prior}}\big)\right)
&\in N_{\mathbb{R}^m_+}\big(\hat\lambda(t)\big)
\label{eq:cont_time_limit:dual1}
\end{align}
\end{subequations}
for almost all $t\in[0,S]$.
\end{theorem}
The proof is provided in \Cref{sec:proof_cont_time_limit}.

\Cref{theorem:cont_time_limit} indicates that a continuous-time counterpart of the discrete-time algorithm \cref{eq:RegPPD} is given by the differential inclusions \cref{eq:cont_time_limit}, which can also be abbreviated as
\begin{equation}\label{eq:pert_sweep}
-\frac{d}{dt}\hat z(t)
+\Phi(\hat z(t),t)\in N_{\mathcal{C}(t)}(\hat z(t)),
\end{equation}
where $\mathcal{C}(t)=\mathcal{X}(t)\times\mathbb{R}^m_+$ and $\Phi$ is a map from $\mathbb{R}^n\times\mathbb{R}^{m}$ to $\mathbb{R}^n\times\mathbb{R}^{m}$. This form of differential inclusions has been studied under the name \emph{perturbed sweeping processes} in the literature \cite{castaing1993evolution,adly2014convex}, and the discrete-time algorithm has been called the \emph{catching algorithm} of the perturbed sweeping processes.

It should be noted that, when the convex set $\mathcal{C}(t)$ is time-varying, the perturbed sweeping process \cref{eq:pert_sweep} in general cannot be equivalently written in the form
\begin{subequations}
\begin{equation}\label{eq:proj_dynamic_form_1}
\frac{d}{dt}\hat z(t)
=\mathcal{P}_{T_{\mathcal{C}(t)}(\hat z(t))}\left[\Phi(\hat z(t),t)\right]
\end{equation}
where $T_{\mathcal{C}(t)}$ denotes the tangent cone of $\mathcal{C}(t)$, or
\begin{equation}\label{eq:proj_dynamic_form_2}
\frac{d}{dt}\hat z(t)
=\lim_{s\rightarrow 0^+}\frac{\mathcal{P}_{\mathcal{C}(t)}\left[\hat z(t)+s\Phi(\hat z(t),t)\right]-\hat z(t)}{s}
\end{equation}
\end{subequations}
as an ordinal projected dynamical system, as there may not exist solutions on $[0,S]$ that satisfy these equations almost everywhere. A simple example is given by $\mathcal{C}(t)=\{(x_1,x_2)\in\mathbb{R}^2:x_1\geq t\}$ with $\Phi(z,t)=0$. It can be seen that under the initial condition $(x_1(0),x_2(0))=(0,0)$, \cref{eq:pert_sweep} admits the solution
$x_1(t)=t,x_2(t)=0$, but \cref{eq:proj_dynamic_form_1} and \cref{eq:proj_dynamic_form_2} do not have solutions. In \cite{hauswirth2018time}, the authors introduced a formulation similar to \cref{eq:proj_dynamic_form_1} based on the notion of temporal tangent cones, which is a generalization of tangent cones in time-varying situations.

Next we study the tracking performance of the system of differential inclusions \cref{eq:cont_time_limit}, and present the following theorem which is the continous-time counterpart of \cref{theorem:tracking_performance}.
\begin{theorem}\label{theorem:tracking_cont_time}
Suppose there exists $\delta>0$, $\beta>0$, $\eta>0$ and $\epsilon>0$ such that
\begin{equation}\label{eq:cont_time_tracking_cond}
\beta^{-1}\sigma_{\eta}<\delta\, \gamma(\delta,\eta,\epsilon)-\sqrt{\eta}\epsilon{M}_{\lambda},
\end{equation}
where
\begin{equation}
\gamma(\delta,\eta,\epsilon) := \min\left\{{\Lambda}_m(\delta),\eta\epsilon\right\}
-\frac{\sqrt{\eta}}{4}\delta{M}_{nc}(\delta).
\end{equation}
Let $\hat z(t)$ be a Lipschitz continuous solution to \cref{eq:cont_time_limit} with $\|\hat z(0)-z^\ast(0)\|_{\eta^{-1}}<\delta$. Then for all $t\in[0,S]$,
\begin{equation}\label{eq:cont_time_bound}
\begin{aligned}
\|\hat z(t)-z^\ast(t)\|_{\eta}
<\ 
&\frac{\beta^{-1}{\sigma}_\eta+\sqrt{\eta}\epsilon{M}_{\lambda}}{\gamma(\delta,\eta,\epsilon)} \\
&+e^{-\beta\gamma(\delta,\eta,\epsilon)\,t}
\left(
\left\|\hat z(0)-z^\ast(0)\right\|_{\eta}
-\frac{\beta^{-1}{\sigma}_\eta+\sqrt{\eta}\epsilon{M}_{\lambda}}{\gamma(\delta,\eta,\epsilon)}\right).
\end{aligned}
\end{equation}
\end{theorem}
\begin{proof}
The following Gronwall-type lemma will be employed, whose proof is presented in \Cref{sec:proof_gronwall_type}.
\begin{lemma}\label{lemma:gronwall_type}
Let $v(t)$ be a nonnegative absolutely continuous function that satisfies
$$
\frac{1}{2}\frac{d}{dt}\left(v^2(t)\right)\leq a v(t)-b v^2(t)
$$
for almost all $t\in[0,S]$, where $a$ and $b$ are nonnegative constants. Then
$$
v(t)\leq e^{-b t}v(0)+\frac{a}{b}(1-e^{-b t}).
$$
\end{lemma}

First of all, we notice that \cref{eq:cont_time_limit:primal} implies
$$
\begin{aligned}
0\geq\ &
(x^\ast(t)-\hat x(t))^T\bigg(-\frac{d}{dt}{\hat x}(t)-\beta\nabla_x \mathcal{L}^{nc}(\hat x(t),\lambda^\ast(t),t) \\
&-\beta J_{f^c,x}(\hat x(t),t)^T\hat\lambda(t)
+\beta
J_{f^{nc},x}(\hat x(t),t)^T
\big(\lambda^\ast(t)-\hat\lambda(t)\big)\bigg),
\end{aligned}
$$
and by \cref{eq:KKT_form1:primal},
$$
\begin{aligned}
0\geq\ &
(\hat x(t)-x^\ast(t))^T
\left(
-\beta\nabla_x \mathcal{L}^{nc}(x^\ast(t),\lambda^\ast(t),t)
-\beta J_{f^c,x}(x^\ast(t),t)^T\lambda^\ast(t)
\right).
\end{aligned}
$$
Thus it can be shown that
$$
\begin{aligned}
& (\hat x(t)-x^\ast(t))^T\frac{d}{dt}\hat x(t) \\
\leq\, &
-\beta(\hat x(t)-x^\ast(t))^T
\Big(
\nabla_x \mathcal{L}^{nc}(\hat x(t),\lambda^\ast(t),t)
-\nabla_x \mathcal{L}^{nc}(x^\ast(t),\lambda^\ast(t),t) \\
&+J_{f^c,x}(\hat x(t),t)^T\hat\lambda(t)-J_{f^c,x}(x^\ast(t),t)^T\lambda^\ast(t)
+
J_{f^{nc},x}(\hat x(t),t)^T
\big(\hat\lambda(t)-\lambda^\ast(t)\big)
\Big) \\
=\ &
-\beta(\hat x(t)-x^\ast(t))^T B_{\mathcal{L}^{nc}}(t)\,(\hat x(t)-x^\ast(t))
\\
&-\beta(\hat x(t)-x^\ast(t))^T\Big(J_{f^c,x}(\hat x(t),t)^T\hat\lambda(t)-J_{f^c,x}(x^\ast(t),t)^T\lambda^\ast(t) \\
&\qquad\qquad +
J_{f^{nc},x}(\hat x(t),t)^T
\big(\hat\lambda(t)-\lambda^\ast(t)\big)
\Big),
\end{aligned}
$$
where $B_{\mathcal{L}^{nc}}(t):=\overline{H}_{\mathcal{L}^{nc}}(\hat x(t)-x^\ast(t),t)$. Next, by \cref{eq:cont_time_limit:dual1},
$$
\begin{aligned}
0\geq\ & (\lambda^\ast(t)-\hat\lambda(t))^T
\left(
-\frac{d}{dt}\hat\lambda(t)
+\eta\beta\left( f(\hat x(t),t)
-\epsilon(\hat\lambda(t)-\lambda_{\mathrm{prior}})\right)
\right),
\end{aligned}
$$
and by \cref{eq:KKT_form1:dual1}, $0\geq
\eta\beta\big(\hat\lambda(t)-\lambda^\ast(t)\big)^T
f(x^\ast(t),t)$. We then get
$$
\begin{aligned}
\big(\hat\lambda(t)-\lambda^\ast(t)\big)^T\frac{d}{dt}\hat\lambda(t)
\leq\ &
\eta\beta\big(\hat\lambda(t)-\lambda^\ast(t))^T
\big(f(\hat x(t),t)-f(x^\ast(t),t)\big) \\
&-\eta\beta\epsilon\big\|\hat\lambda(t)-\lambda^\ast(t)\big\|^2
-\eta\beta\epsilon\big(\hat\lambda(t)-\lambda^\ast(t)\big)^T(\lambda^\ast(t)-\lambda_{\mathrm{prior}}).
\end{aligned}
$$
Now it can be seen that
$$
\begin{aligned}
&\frac{1}{2}\frac{d}{dt}\|\hat z(t)-z^\ast(t)\|_{\eta}^2 \\
=\ &
(\hat x(t)-x^\ast(t))^T\frac{d}{dt}\hat x(t)
+\eta^{-1}
\big(\hat \lambda(t)-\lambda^\ast(t)\big)^T\frac{d}{dt}\hat\lambda(t)
\\
& -(\hat z(t)-z^\ast(t))^T\begin{bmatrix}
I_n & \\
 & \eta^{-1} I_{m}
\end{bmatrix}\frac{d}{dt}z^\ast(t) \\
\leq\ &
-\beta(\hat x(t)-x^\ast(t))^T B_{\mathcal{L}^{nc}}(t)(\hat x(t)-x^\ast(t)) -\beta\epsilon\left(\|\hat\lambda(t)-\lambda^\ast(t)\|^2\right)\\
&
-\beta(\hat x(t)-x^\ast(t))^T\left(J_{f^c,x}(\hat x(t),t)^T\hat\lambda(t)-J_{f^c,x}(x^\ast(t),t)^T\lambda^\ast(t)\right) \\
&
-\beta(\hat x(t)-x^\ast(t))^T
J_{f^{nc},x}(\hat x(t),t)^T
\big(\hat\lambda(t)-\lambda^\ast(t)\big) \\
&
+\beta
\left(f(\hat x(t),t)-f(x^\ast(t),t)\right)^T
\big(\hat\lambda(t)-\lambda^\ast(t)\big) \\
&-\beta\epsilon
\big(\lambda^\ast(t)-\lambda_{\mathrm{prior}}\big)^T
\big(\hat\lambda(t)-\lambda^\ast(t)\big)
+ {\sigma}_{\eta}\|\hat z(t)-z^\ast(t)\|_{\eta},
\end{aligned}
$$
in which
$$
\begin{aligned}
& -(\hat x(t)-x^\ast(t))^T\left(J_{f^c,x}(\hat x(t),t)^T\hat\lambda(t)-J_{f^c,x}(x^\ast(t),t)^T\lambda^\ast(t)\right) \\
& -(\hat x(t)-x^\ast(t))^T
J_{f^{nc},x}(\hat x(t),t)^T
\big(\hat\lambda(t)-\lambda^\ast(t)\big) \\
&+\left(f(\hat x(t),t)-f(x^\ast(t),t) \right)^T
\big(\hat\lambda(t)-\lambda^\ast(t)\big) \\
=\ &
-\hat\lambda(t)^T
\left(
f^c(x^\ast(t),t)-f(\hat x(t),t)
-J_{f^c,x}(\hat x(t),t)(x^\ast(t)-\hat x(t))
\right) \\
&
-\lambda^\ast(t)^T
\left(
f^c(\hat x(t),t)-f(x^\ast(t),t)
-J_{f^c,x}(x^\ast(t),t)(\hat x(t)-x^\ast(t))
\right) \\
&+\big(\hat\lambda(t)-\lambda^\ast(t)\big)^T\left(f^{nc}(\hat x(t),t)+ J_{f^{nc},x}(\hat x(t),t)(x^\ast(t)-\hat x(t))-f^{nc}(x^\ast(t),t)
\right).
\end{aligned}
$$
Since $f^c(x^\ast(t),t)-f(\hat x(t),t)
-J_{f^c,x}(\hat x(t),t)(x^\ast(t)-\hat x(t))\geq 0$ by the convexity of $f^c$, and
$$
\begin{aligned}
&\lambda^\ast(t)^T\left(f^c(\hat x(t),t)-f(x^\ast(t),t)
-J_{f^c,x}(x^\ast(t),t)(\hat x(t)-x^\ast(t))\right) \\
=\ &
(\hat x(t)-x^\ast(t))^T\left(\frac{1}{2}\sum_{i=1}^m \lambda^\ast_i(t) B_{f^c_i}(t)\right)(\hat x(t)-x^\ast(t)),
\end{aligned}
$$
where we denote $B_{f^c_i}(t):=\overline{H}_{f^c_i}(\hat x(t)-x^\ast(t),t)$, we can see that
\begin{equation}\label{eq:proof_cont_time_bound:temp}
\begin{aligned}
&\frac{1}{2}\frac{d}{dt}\|\hat z(t)-z^\ast(t)\|_{\eta}^2 \\
\leq\ &
-\beta(\hat x(t)-x^\ast(t))^T
\left(B_{\mathcal{L}^{nc}}(t)+\frac{1}{2}\sum_{i=1}^m\lambda^\ast_i(t)B_{f^c_i}(t)\right)
(\hat x(t)-x^\ast(t)) \\
&-\beta\epsilon
\big\|\hat\lambda(t)-\lambda^\ast(t)\big\|^2
-\beta\epsilon
\big(\lambda^\ast(t)-\lambda_{\mathrm{prior}}\big)^T
\big(\hat\lambda(t)-\lambda^\ast(t)\big)
\\
&
-\beta\big(
\hat\lambda(t)-\lambda^\ast(t)\big)^T
\big(f^{nc}(x^\ast(t),t) \\
&\qquad\qquad
-f^{nc}(\hat x(t),t)- J_{f^{nc},x}(\hat x(t),t)(x^\ast(t)-\hat x(t))\big)
+ {\sigma}_{\eta}\|\hat z(t)-z^\ast(t)\|_{\eta}
.
\end{aligned}
\end{equation}

Let $\kappa_1$ be the Lipschitz constant of $\hat z(t)$ with respect to the norm $\|\cdot\|_{\eta}$. Define
$$
\tilde\Delta := \frac{1}{2(\kappa_1+\sigma)}\left(\delta-\max\left\{\left\|\hat z(0)-z^\ast(0)\right\|_{\eta}\,
,
\frac{\beta^{-1}{\sigma}_{\eta}+\sqrt{\eta}\epsilon{M}_{\lambda}}{\gamma(\delta,\eta,\epsilon)}\right\}\right)
$$
We prove by induction that
\begin{equation}\label{eq:proof_cont_time_bound:hyp}
\|\hat z(t)-z^\ast(t)\|_{\eta}
\leq
\max\left\{\|\hat z(0)-z^\ast(0)\|_{\eta}\,
,
\frac{\beta^{-1}{\sigma}_\eta+\sqrt{\eta}\epsilon {M}_{\lambda}}{\gamma(\delta,\eta,\epsilon)}\right\}.
\end{equation}
for $t\in\big[(k-1)\tilde\Delta,k\tilde\Delta\big]\cap[0,S]$ for each $k=1,\ldots,\big\lceil S/\tilde\Delta \big\rceil$. Obviously \cref{eq:proof_cont_time_bound:hyp} holds for $t=0$. Now assume that \cref{eq:proof_cont_time_bound:hyp} holds for $t=k\tilde \Delta$. Then we have
$$
\begin{aligned}
\|\hat z(t)-z^\ast(t)\|_{\eta}
&\leq
\big\|\hat z(t)-\hat z\big(k\tilde\Delta\big)\big\|_{\eta}
+\big\|\hat z\big(k\tilde\Delta\big)-z^\ast\big(k\tilde\Delta\big)\big\|_{\eta}
+\big\|z^\ast\big(k\tilde\Delta\big)-z^\ast(t)\big\|_{\eta} \\
&\leq (\kappa_1+\sigma_\eta)\tilde\Delta
+\big\|\hat z(k\tilde\Delta)-z^\ast(k\tilde\Delta)\big\|_{\eta}
< \delta
\end{aligned}
$$
for any $t\in\big[k\tilde\Delta,(k+1)\tilde\Delta\big]\cap[0,S]$. Therefore by the definition of $M_{nc}(\delta)$ and \cref{eq:curve_bound_f},
\begin{equation}\label{eq:proof_cont_time:curve_bound}
\begin{aligned}
&\left\|
f^{nc}(\hat x(t),t)+J_{f^{nc},x}(\hat x(t),t)(x^\ast(t)-\hat x(t))
-f^{nc}(x^\ast(t),t)\right\| \\
\leq\ & \frac{{M}_{nc}(\delta)}{2}
\left\|\hat x(t)-x^\ast(t)\right\|^2
\end{aligned}
\end{equation}
for $t\in\big[k\tilde\Delta,(k+1)\tilde\Delta\big]\cap[0,S]$. Moreover, by Young's inequality,
\begin{equation}\label{eq:proof_cont_time:young}
\begin{aligned}
& \left\|\hat x(t)-x^\ast(t)\right\|
\big\|
\hat\lambda(t)-\lambda^\ast(t)
\big\| \\
\leq\ &
\frac{1}{2}
\left(\sqrt{\eta}\left\|\hat x(t)-x^\ast(t)\right\|^2
+\frac{1}{\sqrt{\eta}}
\big\|
\hat\lambda(t)-\lambda^\ast(t)\big\|^2\right)
=
\frac{\sqrt{\eta}}{2}
\left\|\hat z(t)-z^\ast(t)\right\|_{\eta}^2.
\end{aligned}
\end{equation}
Combining \cref{eq:proof_cont_time:curve_bound} and \cref{eq:proof_cont_time:young} with \cref{eq:proof_cont_time_bound:temp}, we get, for ${t\in\big[k\tilde\Delta,(k+1)\tilde\Delta\big]\cap[0,S]}$,
\begin{equation}\label{eq:proof_cont_time_bound:gronwall_type}
\begin{aligned}
&\frac{1}{2}\frac{d}{dt}\|\hat z(t)-z^\ast(t)\|_{\eta}^2 
 \\
\leq\ &
-\beta(\hat x(t)-x^\ast(t))^T
\left(B_{\mathcal{L}^{nc}}(t)+\frac{1}{2}\sum_{i=1}^m\lambda^\ast_i(t)B_{f^c_i}(t)\right)
(\hat x(t)-x^\ast(t)) 
 \\
&-\beta\epsilon
\big\|\hat\lambda(t)-\lambda^\ast(t)\big\|^2
-\beta\epsilon
\big(\lambda^\ast(t)-\lambda_{\mathrm{prior}}\big)^T
\big(\hat\lambda(t)-\lambda^\ast(t)\big)
 \\
&
+\frac{\beta{M}_{nc}(\delta)}{2}\|\hat x(t)-x^\ast(t)\|^2
\big\|
\hat\lambda(t)-\lambda^\ast(t)
\big\| 
+ {\sigma}_\eta\|\hat z(t)-z^\ast(t)\|_{\eta}
 \\
\leq\ &
-\beta\left(\min\left\{{\Lambda}_m(\delta),\eta\epsilon\right\}
-\frac{\sqrt{\eta}}{4}\delta {M}_{nc}(\delta)\right)\|\hat z(t)-z^\ast(t)\|^2_{\eta} 
 \\
&+\beta\left(\beta^{-1}{\sigma}_\eta+\sqrt{\eta}\epsilon{M}_{\lambda}\right)\|\hat z(t)-z^\ast(t)\|_{\eta}.
\end{aligned}
\end{equation}
Then by the condition \cref{eq:cont_time_tracking_cond}, \cref{lemma:gronwall_type} implies
\begin{equation}\label{eq:proof_cont_time_bound:temp2}
\begin{aligned}
&\|\hat z(t)-z^\ast(t)\|_{\eta} \\
\leq\ & e^{-\beta\gamma(\delta,\eta,\epsilon) (t-k\tilde\Delta)}
\left(\big\|\hat z\big(k\tilde\Delta)-z^\ast\big(k\tilde\Delta\big)\big\|_{\eta}
-\frac{\beta^{-1}{\sigma}_\eta+\sqrt{\eta}\epsilon {M}_{\lambda}}
{\gamma(\delta,\eta,\epsilon)}\right) \\
& + \frac{\beta^{-1}{\sigma}_\eta+\sqrt{\eta}\epsilon {M}_{\lambda}}{\gamma(\delta,\eta,\epsilon)}
\end{aligned}
\end{equation}
for $t\in\big[k\tilde\Delta,(k+1)\tilde\Delta\big]\cap[0,S]$. Now, if $\big\|\hat z\big(k\tilde\Delta\big)-z^\ast\big(k\tilde\Delta\big)\big\|_{\eta}$ is less than or equal to $\big(\beta^{-1}{\sigma}_\eta+\sqrt{\eta}\epsilon{M}_{\lambda}\big)/\gamma(\delta,\eta,\epsilon)$, then \cref{eq:proof_cont_time_bound:temp2} shows that
$$
\|\hat z(t)-z^\ast(t)\|_{\eta}
\leq\frac{\beta^{-1}{\sigma}_\eta+\sqrt{\eta}\epsilon{M}_{\lambda}}{\gamma(\delta,\zeta,\nu,\epsilon)},\qquad t\in\big[k\tilde\Delta,(k+1)\tilde\Delta\big]\cap[0,S],
$$
while if $\big\|\hat z\big(k\tilde\Delta\big)-z^\ast\big(k\tilde\Delta\big)\big\|_{\eta}$ is greater than $\big(\beta^{-1}{\sigma}_\eta+\sqrt{\eta}\epsilon{M}_{\lambda}\big)/\gamma(\delta,\eta,\epsilon)$, then \cref{eq:proof_cont_time_bound:temp2} with $t=k\tilde\Delta$ and \cref{eq:proof_cont_time_bound:hyp} imply
$$
\|\hat z(t)-z^\ast(t)\|_{\eta}
\leq\big\|\hat z\big(k\tilde\Delta\big)-z^\ast\big(k\tilde\Delta\big)\big\|_{\eta}
\leq \|\hat z(0)-z^\ast(0)\|_{\eta}
$$
for $t\in\big[k\tilde\Delta,(k+1)\tilde\Delta\big]\cap[0,S]$, and we can see that \cref{eq:proof_cont_time_bound:hyp} holds for $t\in\big[k\tilde\Delta,(k+1)\tilde\Delta\big]\cap[0,S]$. By induction \cref{eq:proof_cont_time_bound:hyp} holds for all $t\in[0,S]$, and particularly we get
$$
\|\hat z(t)-z^\ast(t)\|_{\eta}<\delta
$$
for all $t\in[0,S]$. This suggests that \cref{eq:proof_cont_time_bound:gronwall_type} holds for all $t\in[0,S]$, and finally by \cref{lemma:gronwall_type}, we get the desired bound on $\|\hat z(t)-z^\ast(t)\|_{\eta}$.
\end{proof}

It is interesting to make a comparison between \cref{theorem:tracking_performance} and \cref{theorem:tracking_cont_time}. The Taylor expansion \cref{eq:proof_feas_param:taylor_rho} shows that
$$
\rho(\delta,\alpha,\eta,\epsilon)
=1-\alpha\gamma(\delta,\eta,\epsilon)+
O\big(\alpha^2\big).
$$
Therefore, if we set $\alpha=\Delta_T\beta$ in the condition \cref{eq:main_thm_cond1} and let $\Delta_T\rightarrow 0^+$, we will recover the condition \cref{eq:cont_time_tracking_cond} except the strictness of the inequality. Furthermore, for any $t\in[0,S]$, we have
$$
\begin{aligned}
\rho^{\lfloor t/\Delta_T\rfloor}(\delta,\Delta_T\beta,\eta,\epsilon)
=\ &\left(1-\Delta_T\beta\gamma(\delta,\eta,\epsilon)+
O\big(\Delta_T^2\big)\right)^{\lfloor t/\Delta_T\rfloor} \\
=\ &
e^{-\beta\gamma(\delta,\eta,\epsilon)t}+O(\Delta_T),
\end{aligned}
$$
from which we can also recover \cref{eq:cont_time_bound}. These observations partially justify that \cref{eq:cont_time_limit} indeed gives the correct continuous-time limit of the discrete-time algorithm \cref{eq:RegPPD}.

Notice that the continuous-time tracking error bound \cref{eq:cont_time_bound} shares a similar form with the discrete-time tracking error bound \cref{eq:main_tracking_error}: a constant term
\begin{equation}\label{eq:cont_time_bound2}
\frac{
\beta^{-1}\sigma_{\eta}
+\sqrt{\eta}\epsilon M_{\lambda}
}{\gamma(\delta,\eta,\epsilon)},
\end{equation}
which we still call the \emph{eventual tracking error bound} in the continuous-time limit, plus something that decays exponentially with $\tau$. The eventual tracking error bound can also be split into two parts, the first part $\beta^{-1}\sigma_{\eta}/\gamma(\delta,\eta,\epsilon)$ being proportional to $\sigma_\eta$, and the second part $\sqrt{\eta}\epsilon M_{\lambda}/\gamma(\delta,\eta,\epsilon)$ representing the discrepancy introduced by regularization.

\subsection{Feasible parameters}

As can be seen, the quantity $\gamma(\delta,\zeta,\epsilon)$ and the bound \cref{eq:cont_time_bound2} are much easier to analyze than $\rho(\delta,\alpha,\zeta,\epsilon)$ and \cref{eq:main_tracking_error2}. This enables us not only to discuss the existence of feasible parameters exist but also the structure of the optimal parameters.

\begin{theorem}\label{theorem:opt_param_cont_time}
Let
$$
\begin{aligned}
\mathscr{A}_{\mathrm{fp}}(\delta,\beta)
&:=\left\{(\eta,\epsilon)\in\mathbb{R}_{++}^2:
\beta^{-1}\sigma_{\eta}<\delta\, \gamma(\delta,\eta,\epsilon)-\sqrt{\eta}\epsilon{M}_{\lambda}
\right\}.
\end{aligned}
$$
\begin{enumerate}[itemindent=13.5pt,leftmargin=0pt]
\item Suppose $\Lambda_m\big(\bar\delta\big)>M_{\lambda} M_{nc}\big(\bar\delta\big)$ for some $\bar\delta>0$. Then the set
$$
\mathscr{S}_{\mathrm{fp}}
:=
\left\{(\delta,\beta,\eta,\epsilon)\in\mathbb{R}_{++}^4: (\eta,\epsilon)\in\mathscr{A}_{\mathrm{fp}}(\delta,\beta)\right\}
$$
is a nonempty open subset of $\mathbb{R}_{++}^4$.

\item Let $\beta>0$ and $\delta>0$ be fixed such that $\mathscr{A}_{\mathrm{fp}}(\delta,\beta)$ is nonempty, and suppose $M_{\lambda} M_{nc}(\delta)>0$. Then
the minimizer of \cref{eq:cont_time_bound2} over $(\eta,\epsilon)\in \mathscr{A}_{\mathrm{fp}}(\delta,\beta)$ exists and is unique, and is equal to $(\Lambda_m(\delta)/\epsilon^\ast,\epsilon^\ast)$ where $\epsilon^\ast$ is the unique minimizer of the unimodal function
$$
b_{\delta,\beta}(\epsilon)
:=\frac{\beta^{-1}\sigma_{\epsilon^{-1}\Lambda_m(\delta)}+\sqrt{\epsilon\Lambda_m(\delta)}M_{\lambda}}{\Lambda_m(\delta)-\delta M_{nc}(\delta)\sqrt{\epsilon^{-1}\Lambda_m(\delta)}/4},
\qquad
\epsilon>\frac{1}{\Lambda_m(\delta)}\left(\frac{\delta M_{nc}(\delta)}{4\Lambda_m(\delta)}\right)^2.
$$
\end{enumerate}
\end{theorem}
The proof is postponed to \Cref{sec:proof_opt_param_cont_time}.

The first part of \cref{theorem:opt_param_cont_time} is the continuous-time counterpart of \cref{theorem:feas_param}. Then in Part 2, we proved that when $\beta$ and $\delta$ are fixed, there exists a unique optimal $(\eta^\ast,\epsilon^\ast)$ that minimizes the eventual tracking error bound \cref{eq:cont_time_bound2}. We have also shown that $\eta^\ast$ is equal to $\Lambda_m(\delta)/\epsilon^\ast$, and that $\epsilon^\ast$ is the unique minimizer of the unimodal function $b_{\delta,\beta}(\epsilon)$. The latter result is a quantitatively characterization of the trade-off in choosing the regularization parameter $\epsilon$: more regularization makes the Lagrangian better conditioned in the dual variables, but introduces additional errors as a side effect.

We fix $\beta$ and $\delta$ in Part 2 of \cref{theorem:opt_param_cont_time}, as a larger $\beta$ or a smaller $\delta$ will always lead to a smaller bound, while in applications $\beta$ usually cannot be arbitrarily chosen because of practical limitations (e.g. computation or communication delays), and an excessively small $\delta$ can also result in violating the condition \cref{eq:cont_time_tracking_cond}.

It should be noted that the optimal $(\eta^\ast,\epsilon^\ast)$ that minimizes the bound \cref{eq:cont_time_bound2} may not be the optimal parameters that minimizes the tracking error itself; in fact, \cref{eq:cont_time_bound2} is only an upper bound (which might be loose in certain situations). However, the analysis presented here will still be of value and can serve as a guide for choosing the parameters in practice.

\begin{remark}
In the second part of \cref{theorem:opt_param_cont_time}, we only consider the case where neither $M_{\lambda}$ and $M_{nc}(\delta)$ is zero. If one of $M_{\lambda}$ and $M_{nc}(\delta)$ is zero, the optimal $(\eta^\ast,\epsilon^\ast)$ may not satisfy the structure stated in \cref{theorem:opt_param_cont_time}, but the analysis is similar and no harder which we omit here.
\end{remark}

\subsection{Isolation of the KKT trajectory}
In \Cref{sec:ProblemFormulation}, we remarked that there could be multiple trajectories of KKT points, and $z^\ast(t)=(x^\ast(t),\lambda^\ast(t))$ is only one of these trajectories that is chosen arbitrarily. Then we analyzed the tracking performance of the algorithm \cref{eq:RegPPD} and its continuous-time limit \cref{eq:cont_time_limit}, and showed that under the conditions \cref{eq:main_thm_cond} or \cref{eq:cont_time_tracking_cond} a bounded tracking error can be achieved. On the other hand, if the KKT trajectory $z^\ast(t)$ bifurcates into two or more branches at some time $\tilde t\in[0,S]$ and these branches become far away as time proceeds, then we have no way to identify from \cref{theorem:tracking_performance} or \cref{theorem:tracking_cont_time} which trajectory the algorithm will track. It is also possible that two KKT trajectories come very close to each other at some time and we cannot distinguish by theory which trajectory the algorithm will track afterwards. Fortunately, as the following theorems show, such possibilities will not occur in some sense under certain conditions.

\begin{subtheorem}{theorem}
\begin{theorem}\label{theorem:isolate_KKTa}
Suppose for some $\delta>0$ and $\eta>0$,
\begin{equation}\label{eq:cond_isolate_KKT}
{\Lambda}_m(\delta)
-\frac{\sqrt{\eta}}{2}\delta {M}_{nc}(\delta)>0.
\end{equation}
Then there is no KKT point in the set
$$
\left\{z=(x,\lambda):0<\|z-z^\ast(t)\|_{\eta}\leq\delta,\,x\neq x^\ast(t)\right\}
$$
for each $t\in[0,S]$.

In particular, \cref{eq:cond_isolate_KKT} holds if the condition \cref{eq:cont_time_tracking_cond} holds for some $\delta\leq 2\eta^{-1/2}M_{\lambda}$.
\end{theorem}
\begin{theorem}\label{theorem:isolate_KKTb}
Let $z_{(i)}^\ast(t)=\big(x_{(i)}^\ast(t),\lambda_{(i)}^\ast(t)\big)$, $i=1,2$ be two Lipschitz continuous trajectories of KKT points of \cref{eq:main_problem} over $t\in[0,S]$, and for each $i=1,2$, define $\Lambda^{(i)}_m(\delta)$, $M^{(i)}_{nc}(\delta)$ for the associated trajectories by \cref{eq:def_Lambda_m} and \cref{eq:def_Mnc}. Suppose there exist $\delta_{(1)}>0$, $\delta_{(2)}>0$ and $\eta>0$ such that
\begin{equation}\label{eq:cond_isolate_KKT_2}
\Lambda^{(i)}_m\big(\delta_{(i)}\big)
-\frac{\sqrt{\eta}}{2}\delta_{(i)}M^{(i)}_{nc}\big(\delta_{(i)}\big)>0
\end{equation}
for $i=1,2$, and for some $t\in[0,S]$, $x_{(1)}^\ast(t)\neq x_{(2)}^\ast(t)$. Then
$$
\big\|z^\ast_{(1)}(t)-z^\ast_{(2)}(t)\big\|_{\eta}>\delta_{(1)}+\delta_{(2)}.
$$
\end{theorem}
\end{subtheorem}
\begin{proof}[Proof of \cref{theorem:isolate_KKTa}]
Suppose for some $t\in[0,S]$, there is another KKT point $z^+=(x^+,\lambda^+)$ satisfying
$$
0<\|z^+-z^\ast(t)\|_{\eta}\leq\delta
\qquad\textrm{and}\qquad
x^+\neq x^\ast(t).
$$
By the KKT condition \cref{eq:KKT_form1:primal}, we have
$$
0\geq
(x^+-x^\ast(t))^T
\big(
-\nabla_x \mathcal{L}^{nc}(x^\ast(t),\lambda^\ast(t),t)
-J_{f^c,x}(x^\ast(t),t)^T\lambda^\ast(t)
\big),
$$
and
$$
\begin{aligned}
0 \geq\ &
(x^\ast(t)-x^+)^T\big(-\nabla_x \mathcal{L}^{nc}(x^+,\lambda^\ast(t),t) \\
&\qquad
-J_{f^c,x}(x^+,t)^T\lambda^+
+
J_{f^{nc},x}(x^+,t)^T
\big(\lambda^\ast(t)-\lambda^+\big)\big).
\end{aligned}
$$
Thus
$$
\begin{aligned}
0\geq\, &
(x^+ -x^\ast(t))^T
\big(
\nabla_x \mathcal{L}^{nc}(x^+,\lambda^\ast(t),t)
-\nabla_x \mathcal{L}^{nc}(x^\ast(t),\lambda^\ast(t),t) \\
&+J_{f^c,x}(x^+,t)^T\lambda^+ -J_{f^c,x}(x^\ast(t),t)^T\lambda^\ast(t)
+
J_{f^{nc},x}(x^+,t)^T
\big(
\lambda^+-\lambda^\ast(t)\big)\big) \\
=\ &
(x^+-x^\ast(t))^T B_{\mathcal{L}^{nc}}(t)(x^+ -x^\ast(t))
\\
&+(x^+-x^\ast(t))^T\left(J_{f^c,x}(x^+,t)^T\lambda^+ -J_{f^c,x}(x^\ast(t),t)\lambda^\ast(t)\right) \\
&+(x^+-x^\ast(t))^T
J_{f^{nc},x}(x^+,t)^T
\big(
\lambda^+-\lambda^\ast(t)\big),
\end{aligned}
$$
where $B_{\mathcal{L}^{nc}}(t):=\overline{H}_{\mathcal{L}^{nc}}(x^+ -x^\ast(t),t)$. Also,
$$
(\lambda^+-\lambda^\ast(t))^T
\left(f(x^+,t)-f(x^\ast(t),t)\right)
=
-{\lambda^+}^T f(x^\ast(t),t)
-{\lambda^\ast(t)}^T f(x^+,t)\geq 0
$$
by the complementary slackness condition. Therefore
$$
\begin{aligned}
0\geq\ &(x^+-x^\ast(t))^T B_{\mathcal{L}^{nc}}(t)(x^+ -x^\ast(t))
\\
&+(x^+-x^\ast(t))^T\left(J_{f^c,x}(x^+,t)^T\lambda^+ -J_{f^c,x}(x^\ast(t),t)\lambda^\ast(t)\right) \\
&
+\big(
\lambda^+ -\lambda^\ast(t)\big)^T
\left(
J_{f^{nc},x}(x^+,t)(x^+ - x^\ast(t))
-
\left(
f(x^+,t) - f(x^\ast(t),t)\right)
\right).
\end{aligned}
$$
Notice that
$$
\begin{aligned}
&(x^+-x^\ast(t))^T\left(J_{f^c,x}(x^+,t)^T\lambda^+ -J_{f^c,x}(x^\ast(t),t)\lambda^\ast(t)\right) \\
&
+\big(
\lambda^+ - \lambda^\ast(t)\big)^T
\left(
J_{f^{nc},x}(x^+,t)(x^+ - x^\ast(t))
-
\left(
f(x^+,t) - f(x^\ast(t),t)\right)
\right) \\
=\ &
{\lambda^+}^T
\left(
f^c(x^\ast(t),t)-f^c(x^+,t)
-J_{f^c,x}(x^+,t)(x^\ast(t)-x^+)
\right) \\
&
+\lambda^\ast(t)^T
\left(
f^c(x^+,t)-f^c(x^\ast(t),t)
-J_{f^c,x}(x^\ast(t),t)(x^+ -x^\ast(t))
\right) \\
&
-
\big(
\lambda^+-\lambda^\ast(t)\big)^T
\left(
f^{nc}(x^+,t)+J_{f^{nc},x}(x^+,t)(x^\ast(t)-x^+)
-f^{nc}(x^\ast(t),t)\right) \\
\geq\ &
(x^+-x^\ast(t))^T
\left(\frac{1}{2}\sum_{i=1}^m \lambda^\ast_i(t)B_{f^c_i}(t)\right)(x^+-x^\ast(t)) \\
&
-\big(
\lambda^+-\lambda^\ast(t)\big)^T
\left(
f^{nc}(x^+,t)+J_{f^{nc},x}(x^+,t)(x^\ast(t)-x^+)
-f^{nc}(x^\ast(t),t)\right),
\end{aligned}
$$
where we denote $B_{f^c_i}(t):=\overline{H}_{f^c_i}(x^+-x^\ast(t),t)$. Therefore
$$
\begin{aligned}
0\geq\ &
(x^+-x^\ast(t))^T
\left(B_{\mathcal{L}^{nc}}(t)
+\frac{1}{2}\sum_{i=1}^m \lambda^\ast_i(t)B_{f^c_i}(t)\right)(x^+ -x^\ast(t)) \\
&
-\big(\lambda^+-\lambda^\ast(t)\big)^T\left(
f^{nc}(x^+,t)+J_{f^{nc},x}(x^+,t)(x^\ast(t)-x^+)
-f^{nc}(x^\ast(t),t)\right) \\
\geq\ &
{\Lambda}_m(\delta)\|x^+-x^\ast(t)\|^2
-\frac{{M}_{nc}(\delta)}{2}
\|x^+-x^\ast(t)\|^2
\big\|
\lambda^+-\lambda^\ast(t)
\big\| \\
\geq\ &
\left({\Lambda}_m(\delta)
-\frac{\sqrt{\eta}}{2}\delta {M}_{f^{nc}}(\delta)\right)
\|x^+-x^\ast(t)\|^2.
\end{aligned}
$$
However, if \cref{eq:cond_isolate_KKT} holds, the right-hand side of the above inequality is then positive, leading to a contradiction.

Now we prove that \cref{eq:cond_isolate_KKT} holds if \cref{eq:cont_time_tracking_cond} holds for some $\delta\leq 2\eta^{-1/2}M_{\lambda}$. We have
$$
\min\left\{\Lambda_m(\delta),\eta\epsilon\right\}
\leq\frac{\Lambda_m(\delta)+\eta\epsilon}{2},
$$
and so
$$
\begin{aligned}
\Lambda_m(\delta)
\geq 2\min\left\{\Lambda_m(\delta),\eta\epsilon\right\}-\eta\epsilon.
\end{aligned}
$$
On the other hand, \cref{eq:cont_time_tracking_cond} implies that
$$
\begin{aligned}
\min\left\{\Lambda_m(\delta),\eta\epsilon\right\}
&> \frac{\sqrt{\eta}}{4}\delta M_{nc}(\delta)
+\delta^{-1}
\left(\beta^{-1}\sigma_\eta+\sqrt{\eta}\epsilon M_{\lambda}\right) \\
&>\frac{\sqrt{\eta}}{4}\delta M_{nc}(\delta)
+\delta^{-1}\sqrt{\eta}\epsilon M_{\lambda}.
\end{aligned}
$$
Since $\delta\leq 2\eta^{-1/2}M_{\lambda}$, we have $\delta^{-1}\sqrt{\eta}\epsilon M_{\lambda}\geq \eta\epsilon/2$, and so
$$
\begin{aligned}
\Lambda_m(\delta)
&>
2
\left(\frac{\sqrt{\eta}}{4}\delta M_{nc}(\delta)
+\delta^{-1}\sqrt{\eta}\epsilon M_{\lambda}\right)-\eta\epsilon
\geq
\frac{\sqrt{\eta}}{2}\delta M_{nc}(\delta).
\end{aligned}
$$
This completes the proof.
\end{proof}

\begin{proof}[Proof of \cref{theorem:isolate_KKTb}]
For notational simplicity we temporarily denote $z_{(i)}^\ast(t)=\big(x_{(i)}^\ast(t),\lambda_{(i)}^\ast(t)\big)$ by $z_{(i)}=\big(x_{(i)},\lambda_{(i)}\big)$. Let
$$
\tilde\theta_1 := \frac{\delta_{(1)}}{\delta_{(1)}+\delta_{(2)}},\qquad
\tilde\theta_2 := \frac{\delta_{(2)}}{\delta_{(1)}+\delta_{(2)}},\qquad
\tilde z:=\tilde\theta_2 z_{(1)}
+\tilde\theta_1 z_{(2)}.
$$
Then it can be checked that
\begin{equation}\label{eq:proof_isolate_KKT:tilde_z}
z_{(2)}-z_{(1)}
=-\frac{z_{(1)}-\tilde z}{\tilde\theta_1}
=\frac{z_{(2)}-\tilde z}{\tilde\theta_2}
\end{equation}
and $\big\|\tilde z-z_{(1)}\big\|_{\eta}
\leq\delta_{(1)},
\big\|\tilde z-z_{(2)}\big\|_{\eta}
\leq\delta_{(2)}$. 
Now by \cref{eq:KKT_form1:primal}, we have
$$
\begin{aligned}
0
& \geq
\left(x_{(2)}-x_{(1)}\right)^T
\left(
-\nabla_x c\big(x_{(1)},t\big)
-
J_{f,x}\big(x_{(1)},t\big)^T
\lambda^\ast_{(1)}\right), \\
0
& \geq
\left(x_{(1)}-x_{(2)}\right)^T
\left(
-\nabla_x c\big(x_{(2)},t\big)
-
J_{f,x}\big(x_{(2)},t\big)^T
\lambda^\ast_{(2)}\right).
\end{aligned}
$$
Taking their sum, we get
\begin{equation}\label{eq:proof_isolate_KKT:part2_primal}
\begin{aligned}
0\geq &
\left(x_{(2)}-x_{(1)}\right)^T\!\!
\Big(\!\!
\left(\nabla_x c\big(x_{(2)},t\big)
-\nabla_x c\big(\tilde x,t\big)\right)
\!-\!
\left(\nabla_x c\big(x_{(1)},t\big)
-\nabla_x c\big(\tilde x,t\big)\right)
 \\
&
+\!\!
\left(
J_{f,x}\big(x_{(2)},t\big)^T
\lambda_{(2)}
\!-\!
J_{f,x}\big(\tilde x,t\big)^T
\tilde\lambda\right) 
\!-\!
\left(
J_{f,x}\big(x_{(1)},t\big)^T
\lambda_{(1)}
\!-\!
J_{f,x}\big(\tilde x,t\big)^T
\tilde\lambda
\right)\!\!
\Big).
\end{aligned}
\end{equation}
For the dual variables, by \cref{eq:KKT_form1:dual1},
\begin{equation}\label{eq:proof_isolate_KKT:part2_dual1}
\begin{aligned}
0\geq\ &
-\left(\lambda_{(2)}-\lambda_{(1)}\right)^T
\left(f\big(x_{(2)},t\big)-f\big(x_{(1)},t\big)\right) 
\\
=\ &
-\left(\lambda_{(2)}-\lambda_{(1)}\right)^T
\left(f\big(x_{(2)},t\big)-f\big(\tilde x,t\big)
-\left(f\big(x_{(1)},t\big)-f\big(\tilde x,t\big)\right)\right).
\end{aligned}
\end{equation}
Now, for each $i=1,2$, we have
\begin{equation}\label{eq:proof_isolate_KKT:part2_temp1}
\begin{aligned}
&\left(x_{(i)}-\tilde x\right)^T
\big(\nabla_x c\big(x_{(i)},t\big)
-\nabla_x c\big(\tilde x,t\big)
+
J_{f,x}\big(x_{(i)},t\big)^T
\lambda_{(i)}
-
J_{f,x}\big(\tilde x,t\big)^T
\tilde\lambda\big)
\\
&
-\big(
\lambda_{(i)}-\tilde\lambda\big)^T
\left(
f\big(x_{(i)},t\big) - f\big(\tilde x,t\big)\right)
\\
=\ &
\left(x_{(i)}-\tilde x\right)^T
\left(\nabla_x\mathcal{L}^{nc}\big(x_{(i)},\lambda_{(i)},t\big)
-\nabla_x\mathcal{L}^{nc}\big(\tilde x,\lambda_{(i)},t\big)\right)
\\
&
+\big(\lambda_{(i)}-\tilde\lambda\big)^T
\left(f^{nc}\big(\tilde x,t\big)+J_{f^{nc},x}\big(\tilde x,t\big)\left(x_{(i)}-\tilde x\right)-f^{nc}\big(x_{(i)},t\big)\right)
\\
&
+\lambda_{(i)}^T
\left(f^{c}\big(\tilde x,t\big)-f^{c}\big(x_{(i)},t\big)
-J_{f^c}\big(x_{(i)},t\big)^T\left(\tilde x-x_{(i)}\right)\right)
\\
&
+\tilde\lambda^T
\left(
f^{c}\big(x_{(i)},t\big)-f^{c}\big(\tilde x,t\big)
-J_{f^c}\big(\tilde x,t\big)^T\left(x_{(i)}-\tilde x\right)
\right)
\\
\geq\ &
\left(x_{(i)}-\tilde x\right)^T
B^{(i)}\left(x_{(i)}-\tilde x\right)
-\frac{M_{nc}^{(i)}\big(\delta_{(i)}\big)}{2}
\left\|x_{(i)}-\tilde x\right\|^2
\big\|
\lambda_{(i)}-\tilde\lambda
\big\|,
\end{aligned}
\end{equation}
where
$$
\begin{aligned}
B^{(i)}
:=\ &\int_0^1
\nabla_{xx}^2\mathcal{L}^{nc}\big(x_{(i)}+\theta\big(\tilde x-x_{(i)}\big),\lambda_{(i)},t\big)\,d\theta \\
&
+\sum_{j=1}^m
\lambda_{(i),j}
\int_{0}^1(1-\theta)
\nabla_{xx}^2 f_{j}^c\big(x_{(i)}+\theta\big(\tilde x-x_{(i)}\big),t\big)\,d\theta.
\end{aligned}
$$
By summing \cref{eq:proof_isolate_KKT:part2_primal} and \cref{eq:proof_isolate_KKT:part2_dual1}, and plugging in \cref{eq:proof_isolate_KKT:tilde_z} and \cref{eq:proof_isolate_KKT:part2_temp1}, we can see that
$$
\begin{aligned}
0\geq\ &
\sum_{i=1,2}\frac{1}{\tilde\theta_i}
\left(
\left(x_{(i)}-\tilde x\right)^T
B^{(i)}\left(x_{(i)}-\tilde x\right)
-\frac{M_{nc}^{(i)}\big(\delta_{(i)}\big)}{2}
\left\|x_{(i)}-\tilde x\right\|^2
\big\|
\lambda_{(i)}-\tilde\lambda\big\|
\right) \\
\geq\ &
\sum_{i=1,2}\tilde\theta_i
\left(\Lambda_m^{(i)}\big(\delta_{(i)}\big)
-\frac{\sqrt{\eta}}{2}\delta_{(i)}M_{nc}\big(\delta_{(i)}\big)\right)\big\|x_{(2)}-x_{(1)}\big\|^2.
\end{aligned}
$$
But \cref{eq:cond_isolate_KKT_2} then implies that the right-hand side of the above inequality is positive, leading to a contradiction.
\end{proof}

\begin{remark}\label{remark:isolate_KKT}
It should be noted that the condition \cref{eq:cond_isolate_KKT} does not exclude the possibility that at time $t$, there exists $\lambda^+\neq\lambda^\ast(t)$ such that $(x^\ast(t),\lambda^+)$ is also a KKT point of \cref{eq:main_problem_cont} and
$\|(x^\ast(t),\lambda^+)-(x^\ast(t),\lambda^\ast)\|_\eta\leq\delta$, unless we also assume that the optimal dual variable associated with $x^\ast(t)$ is unique at time $t$. A typical constraint qualification  that guarantees the uniqueness of the optimal Lagrange multiplier is the linear independence constraint qualification (LICQ) \cite{wachsmuth2013licq}, which cannot be directly used in our setting but can be possibly checked if we write \cref{eq:main_problem_cont} in some alternative formulation.
\end{remark}

Let us further assume that the optimal dual variable associated with $x^\ast(t)$ is unique for all $t\in[0,S]$ as \cref{remark:isolate_KKT} points out. Then \cref{theorem:isolate_KKTa,theorem:isolate_KKTb} show that, under certain conditions, the KKT points for a given time instant will always be isolated. Especially, when the condition \cref{eq:cont_time_tracking_cond} is satisfied for some $\delta$ that is not too large, there is no ambiguity in which of the KKT trajectories will be tracked by the continuous-time algorithm \cref{eq:cont_time_limit}.

\section{Numerical example}
\label{sec:results}

In this section we illustrate a numerical example that is extracted from a real-world application in power system operation.

The goal is to solve the following time-varying optimization problem
$$
\begin{aligned}
\min_{p\in\mathbb{R}^n,q\in\mathbb{R}^n}\quad &
\sum_{i=1}^n c_{p,i}\big (p_i-{p}^{\mathrm{PV}}_{\tau,i} \big)^2
+c_{q,i} q_i^2 \\
\textrm{s.t.}\quad &
{V}_{\min}\leq V_j(p,q,p^{\mathrm{L}}_{\tau},q^{\mathrm{L}}_{\tau}) \leq {V}_{\max},\quad j=1,\ldots,m, \\
&(p_i,q_i)\in\mathcal{X}_{\tau,i},\qquad i=1,\ldots,n.
\end{aligned}
$$
This optimization problem is solved at each $\tau$ to optimally operate the $n$ inverters each of which is connected to a photovoltaic  panel in a distribution feeder. Here the decision variables $p$ and $q$ represent the real and reactive power injections of the $n$ inverters. For the cost function, $p^{\mathrm{PV}}_{\tau,i}$ is the maximum available real power of the photovoltaic panel connected to inverter $i$ at time $\tau$; and $c_{p,i}$ and $c_{q,i}$ are positive constants for each $i$. For the constraints, $p^{\mathrm{L}}_\tau\in\mathbb{R}^m$ and $q^{\mathrm{L}}_\tau\in\mathbb{R}^m$ represent the time-varying real and reactive loads, $V_j:\mathbb{R}^n\times\mathbb{R}^n\times\mathbb{R}^m\times\mathbb{R}^m\rightarrow\mathbb{R}$ is the implicit function derived from the power flow equations that represents the mapping from power injections to the (normalized) voltage magnitude at bus $j$ in a modified IEEE 37-node test feeder \cite{wang2018explicit,schneider2018analytic}, $V_{\min}$ and $V_{\max}$ are the bounds on the (normalized) voltage magnitudes, and $\mathcal{X}_{\tau,i}$ is given by
$$
\mathcal{X}_{\tau,i}
=\left\{(p,q)\in\mathbb{R}^2:
0\leq p\leq p^{\mathrm{PV}}_{\tau,i}, p^2+q^2\leq S_{i,\max}^2
\right\},
$$
where $S_{i,\max}$ is a positive constant for each $i=1,\ldots,n$. In this numerical example we have $n=18$, $m=36$, $V_{\min}=0.95$, $V_{\max}=1.05$, $c_{p,i}=3$ and $c_{q,i}=1$ for all $i=1,\ldots,n$, $S_{3,\max}=1$, $S_{18,\max}=3.5$, $S_{i,\max}=2$ for $i\neq 3,18$, $q^{\mathrm{L}}_\tau=0.5p^{\mathrm{L}}_\tau$ for each $\tau$. \Cref{fig:profiles} demonstrates the curves of total photovoltaic generation $\sum_i p^{\mathrm{PV}}_{\tau,i}$, individual loads $p^{\mathrm{L}}_\tau$ and total load $\sum_j p^{\mathrm{L}}_{\tau,j}$ in the time domain, obtained from real-world data \cite{Bank13}. The time period is from 09:00 to 15:00, and we set $\Delta_T=1\,\mathrm{s}$ so that $T=21600$. For more detailed setting we refer to \cite
{tang2018feedback}. For the algorithm \cref{eq:RegPPD}, we set $\epsilon=10^{-5}$, $\eta=0.75/\epsilon$, $\alpha=0.04$ and $\lambda_{\mathrm{prior}}=0$.

\begin{figure}[tbhp]
    \centering
    \includegraphics[width=.85\textwidth]{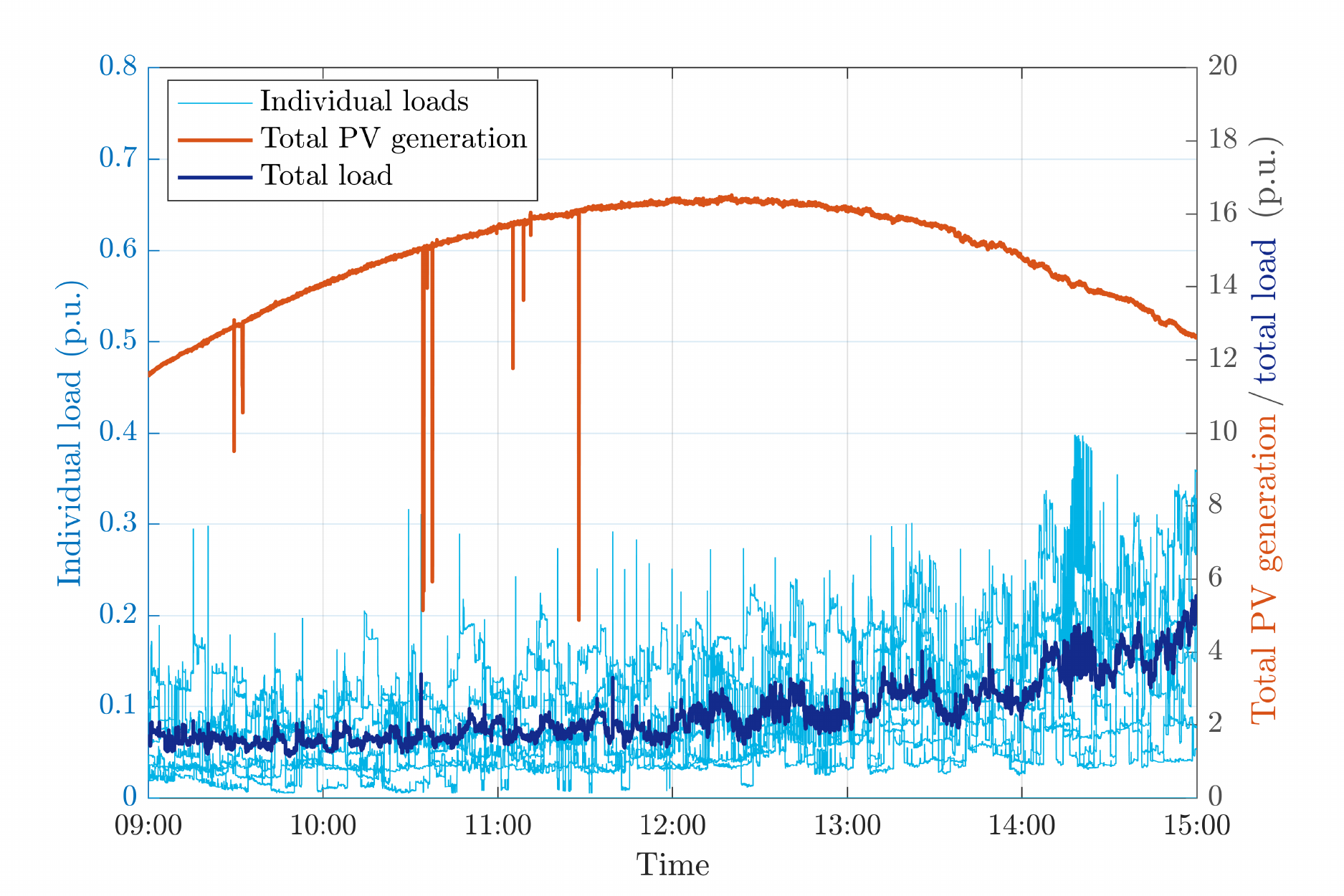}
    \caption{Illustration of individual loads $p^{\mathrm{L}}_\tau$, total load $\sum_j p^{\mathrm{L}}_{\tau,j}$ and total photovoltaic (PV) generation $\sum_i p^{\mathrm{PV}}_{\tau,i}$.}
    \label{fig:profiles}
\end{figure}

\begin{figure}[tbhp]
    \centering
    \includegraphics[width=.9\textwidth]{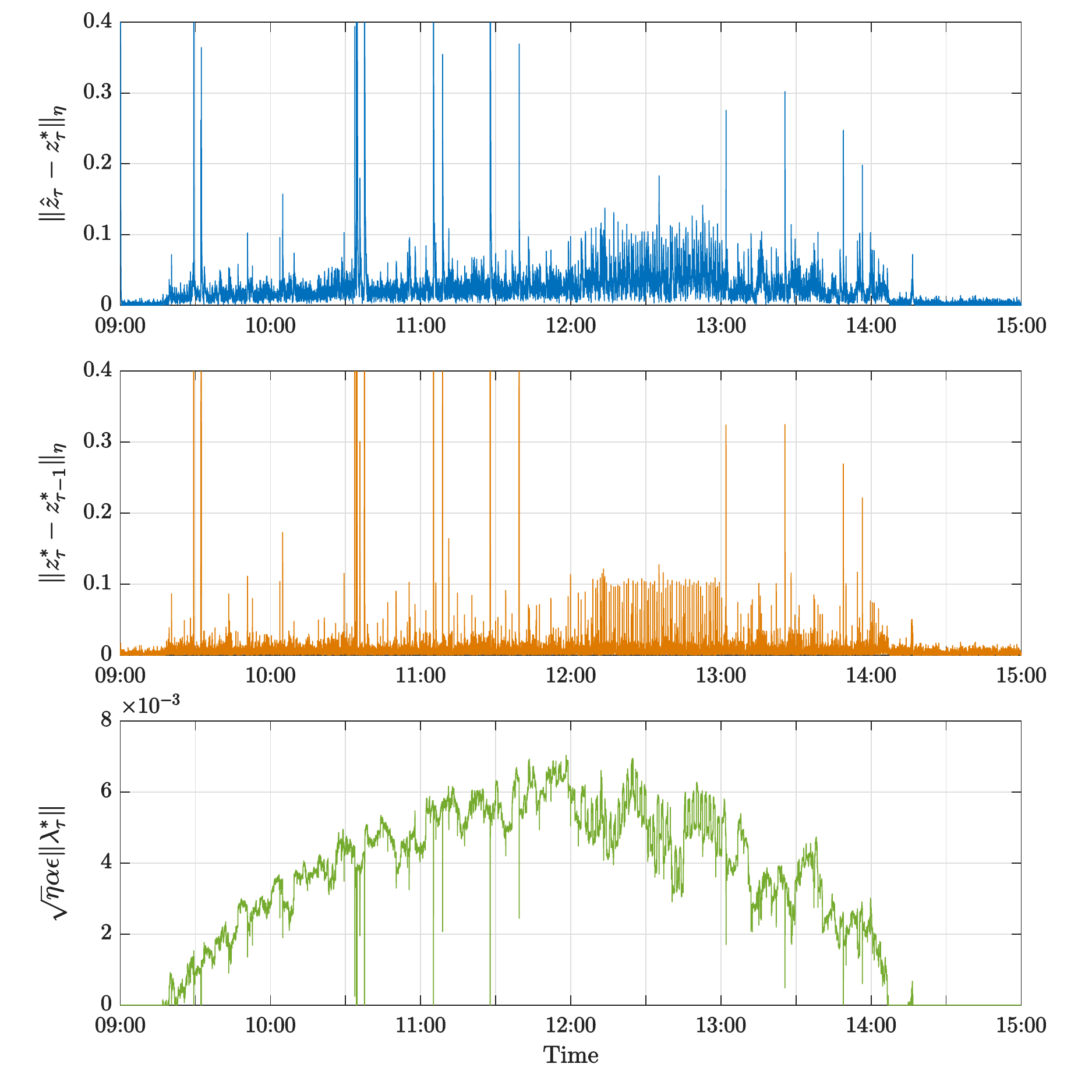}
    \caption{Illustrations of $\left\|\hat z_\tau-z^\ast_\tau\right\|_\eta$, $\big\|z^\ast_\tau-z^\ast_{\tau-1}\big\|_\eta$ and $\sqrt{\eta}\alpha\epsilon\big\|\lambda^\ast_\tau\big\|$.}
    \label{fig:tracking_error}
\end{figure}

\Cref{fig:tracking_error} shows the tracking error $\left\|\hat z_\tau-z^\ast_\tau\right\|_\eta$
from the proposed algorithm \cref{eq:RegPPD}, together with $\|z^\ast_{\tau}-z^\ast_{\tau-1}\|_\eta$ and $\sqrt{\eta}\alpha\epsilon\big\|\lambda^\ast_\tau\big\|$. It can be seen that apart from some spikes, the tracking error is bounded below $0.2$ for all $\tau$. We also have the following statistics:
$$
\frac{1}{T}\sum_{\tau}\|\hat z_\tau-z^\ast_\tau\|_\eta=2.56\times 10^{-2},
\qquad
\frac{1}{T}\sum_{\tau}\frac{\|\hat z_\tau-z^\ast_\tau\|_\eta}{\|z^\ast_\tau\|_\eta}=7.02\times 10^{-3}.
$$
As a comparison, we have $T^{-1}\sum_\tau \|z^\ast_\tau-z^\ast_{\tau-1}\|_\eta=8.55\times 10^{-3}$. Moreover, the illustrations seem to suggest that $\left\|\hat z_\tau-z^\ast_\tau\right\|_\eta$ is strongly correlated with $\|z^\ast_\tau-z^\ast_{\tau-1}\|_\eta\approx \big\|\frac{d}{dt}{z}^\ast(\tau\Delta_T)\big\|\Delta_T$ and $\sqrt{\eta}\alpha\epsilon\big\|\lambda^\ast_\tau\big\|$ in a way similar to the eventual tracking error bound \cref{eq:main_tracking_error2}.

\begin{figure}[tbhp]
    \centering
    \includegraphics[width=.85\textwidth]{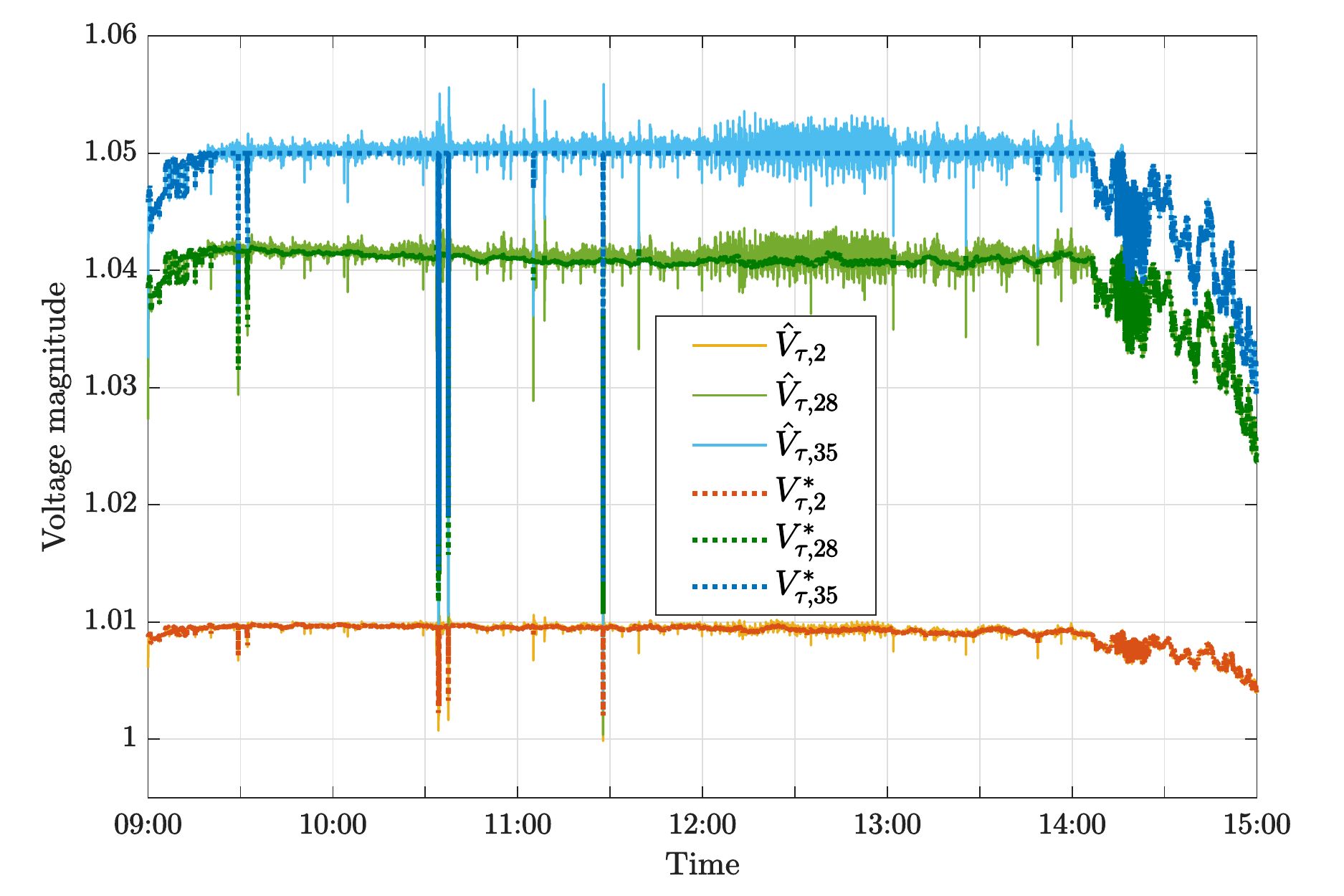}
    \caption{$\hat V_{\tau,j}$ and $V^\ast_{\tau,j}$ for $j=2,28,35$.}
    \label{fig:voltages}
\end{figure}

\Cref{fig:voltages} compares
$$
\hat V_{\tau,j}:=V_j(\hat p_\tau,\hat q_\tau,p^\mathrm{L}_\tau,q^\mathrm{L}_\tau)
$$
with
$$
V^\ast_{\tau,j}:=V_j(p^\ast_\tau,q^\ast_\tau,p^\mathrm{L}_\tau,q^\mathrm{L}_\tau)
$$
for $j=2,28$ and $35$. It can be seen that the constraint $V_{35}(p,q,p^\mathrm{L}_\tau,q^\mathrm{L}_\tau)\leq V_{\max}$ is slightly violated for certain time instants. The violation is very small though, as simulation gives
$$
\frac{1}{T}\sum_{\tau,j}
\left(\left[\hat V_{\tau,j}-V_{\max}\right]_+
+\left[V_{\min}-\hat V_{\tau,j}\right]_+\right)
=3.67\times 10^{-4},
$$
where $[\cdot]_+$ denotes the positive part of a real number. The violation is partly a consequence of introducing regularization, which drives the dual variables towards zero and leads to an underestimation of the optimal dual variables.

We emphasize that this numerical example is derived from real-world scenarios, and the theoretical assumptions may not apply here. Nevertheless, the simulations results seem to be in accordance with the theory presented in the text.

\section{Conclusions and future directions}
\label{sec:conclusions}

In this paper, we conducted a comprehensive study on the regularized primal-dual gradient method and its continuous-time counterpart for time-varying nonconvex optimization. Sufficient conditions that guarantee bounded tracking error were derived for the proposed discrete-time algorithm. A continuous-time version of the algorithm, formulated as a system of differential inclusions, was also considered and analyzed. Implications of these analytical results were discussed, and a numerical example was presented to illustrate the performance of the proposed algorithm in a real-world application.

We make the following remarks on possible generalizations and directions to explore:

(1) In the formulation \cref{eq:main_problem_cont}, only explicit inequality constraints are considered. It turns out that equality constraints can be handled similarly, and all the theoretical results can be readily generalized. We omitted the associated discussion in this paper to avoid tedious but uninspiring derivations.

(2) Another possible and potentially useful generalization of \cref{eq:main_problem_cont} is
$$
\begin{aligned}
\min_{x \in \mathbb{R}^n}\quad & c(x,t)+h(x,t) \\
\textrm{s.t.}\quad & f(x,t)\leq 0,
\end{aligned}
$$
where $h:\mathbb{R}^n\times\mathbb{R}\rightarrow\mathbb{R}\cup\{+\infty\}$ and $h(\cdot,t)$ is a closed proper convex function for each $t$. The primal update \cref{eq:RegPPD:primal} will then take the form
$$
\hat x_\tau = \mathrm{prox}_{\alpha h_\tau}
\left[
\hat x_{\tau-1}
-\alpha \left(\nabla c_\tau(\hat x_{\tau-1})
+
J_{f_\tau}(\hat x_{\tau-1})^T
\hat\lambda_{\tau-1}\right)
\right],
$$
where $\mathrm{prox}$ denotes the \emph{proximal operator} \cite{parikh2014proximal}, and $h_\tau(x):=h(x,\tau\Delta_T)$. As proximal operators share many similar properties with projection operators, it is expected that most results in this paper can be generalized as well.

(3) In this paper we assumed that the KKT trajectory $z^\ast(t)$ is Lipschitz continuous over $[0,S]$. In reality, it is possible that the KKT trajectories are only of bounded variation, and sometimes it is necessary to deal with jumps. So far, it has remained unclear how to develop a time-varying optimization algorithm that can handle this issue without sacrificing efficiency.

(4) In this paper we defined the tracking error by the distance from $\hat z_\tau$ to the true KKT solution $z^\ast_\tau$. In some situations, this definition could be inappropriate for characterizing the sub-optimality of $\hat z_\tau$, and existing literature has proposed other metrics for evaluating the tracking error such as the dynamic regret \cite{mokhtari2016online,yi2016tracking} or the fixed-point residual error \cite{simonetto2017time}. It would be interesting to investigate the tracking performance of the regularized primal-dual algorithm under such metrics, especially to see whether weaker conditions for bounded tracking error can be derived.

\appendix

\section{Proof of \texorpdfstring{\cref{lemma:approx_contraction}}{Lemma \ref{lemma:approx_contraction}}}
\label{sec:proof_approx_contraction}
The KKT conditions \cref{eq:KKT_form1} can be written equivalently in the form of a fixed-point equation
\begin{subequations}\label{eq:KKT_form2}
\begin{align}
x^\ast(t)
&=\mathcal{P}_{\mathcal{X}(t)}
\left[x^\ast(t)
-\alpha\left(\nabla_x c(x^\ast(t),t)
+
J_{f,x}(x^\ast(t),t)^T
\lambda^\ast(t)\right)\right],
\label{eq:KKT_form2:primal}\\
\lambda^\ast(t)
&=\mathcal{P}_{\mathbb{R}^m_+}
\left[\lambda^\ast(t)
+\beta f(x^\ast(t),t)\right],
\label{eq:KKT_form2:dual1}
\end{align}
\end{subequations}
where $\alpha$ and $\beta$ are any positive real number.

Now we have
$$
\begin{aligned}
&
\nabla c_\tau(\hat x_{\tau-1})-\nabla c_\tau(x^\ast_\tau)
+
J_{f_\tau}(\hat x_{\tau-1})^T
\hat\lambda_{\tau-1}
-
J_{f_\tau}(x^\ast_{\tau})^T
\lambda^\ast_{\tau}
\\
=\, &
\nabla_x \mathcal{L}^{nc}_\tau(\hat x_{\tau-1},\lambda^\ast_\tau)
-\nabla_x \mathcal{L}^{nc}_\tau(x^\ast_{\tau},\lambda^\ast_\tau)
+
J_{f^{nc}_\tau}(\hat x_{\tau-1})^T
\big(
\hat\lambda_{\tau-1}-\lambda^\ast_\tau\big)
\\
&+J_{f^c_\tau}(\hat x_{\tau-1})^T\hat\lambda_{\tau-1}
-J_{f^c_\tau}(x^\ast_{\tau})^T\lambda^\ast_\tau\\
=\, &
B_{\mathcal{L}^{nc}_\tau}(\hat x_{\tau-1}-x^\ast_\tau)
+
J_{f^{nc}_\tau}(\hat x_{\tau-1})^T
\big(
\hat\lambda_{\tau-1}-\lambda^\ast_\tau
\big)
+J_{f^c_\tau}(\hat x_{\tau-1})^T\hat\lambda_{\tau-1}
-J_{f^c_\tau}(x^\ast_{\tau})^T\lambda^\ast_\tau,
\end{aligned}
$$
where, recalling the definition of $\overline{H}_{\mathcal{L}^{nc}}$ in~\cref{eq:def_HL} and advocating the Fundamental Theorem of Calculus,
we defined $B_{\mathcal{L}^{nc}_\tau}$ as $B_{\mathcal{L}^{nc}_\tau}
:= \overline{H}_{\mathcal{L}^{nc}}(\hat x_{\tau-1}-x^\ast_\tau\,,\tau\Delta_T)$. 
Then by \cref{eq:RegPPD} and \cref{eq:KKT_form2:primal}, and using the nonexpansiveness of projection onto convex sets, we get
\begin{equation}\label{eq:proof_approx_contraction:primal_diff}
\begin{aligned}
&\left\|\hat x_{\tau}-x^\ast_\tau\right\|^2 
\\
\leq\ &
\Big\|
\left(I-\alpha B_{\mathcal{L}^{nc}_\tau}\right)(\hat x_{\tau-1}-x^\ast_\tau) 
\\
&\quad
-\alpha
\left(
J_{f^{nc}_\tau}(\hat x_{\tau-1})^T
\big(\hat\lambda_{\tau-1}-\lambda^\ast_\tau\big)
+J_{f^c_\tau}(\hat x_{\tau-1})^T\hat\lambda_{\tau-1}
-J_{f^c_\tau}(x^\ast_{\tau})^T\lambda^\ast_\tau
\right)
\Big\|^2.
\end{aligned}
\end{equation}
We have,
\begin{equation}\label{eq:proof_approx_contraction:temp1}
\begin{aligned}
&\left\|
J_{f^{nc}_\tau}(\hat x_{\tau-1})^T
\big(\hat\lambda_{\tau-1}-\lambda^\ast_\tau\big)
+J_{f^c_\tau}(\hat x_{\tau-1})^T\hat\lambda_{\tau-1}
-J_{f^c_\tau}(x^\ast_{\tau})^T\lambda^\ast_\tau\right\|
\\
=\ &
\left\|
J_{f_\tau}(\hat x_{\tau-1})^T
\big(\hat\lambda_{\tau-1}-\lambda^\ast_\tau\big)
+\left(J_{f^c_\tau}(\hat x_{\tau-1})-J_{f^c_\tau}(x^\ast_{\tau})\right)^T
\lambda^\ast_\tau\right\|
\\
\leq\ &
\left\|
J_{f_\tau}(\hat x_{\tau-1})^T
\big(
\hat\lambda_{\tau-1}-\lambda^\ast_\tau\big)\right\|
+ \left\|J_{f^c_\tau}(\hat x_{\tau-1})-J_{f^c_\tau}(x^\ast_{\tau})
\right\|
\sup_{t\in[0,S]}\left\|\lambda^\ast(t)\right\|
 \\
\leq\ &
L_f(\delta)\big\|
\hat\lambda_{\tau-1}-\lambda^\ast_\tau\big\|
+M_c(\delta)\left\|\hat x_{\tau-1}-x^\ast_\tau\right\|
\sup_{t\in[0,S]}\left\|\lambda^\ast(t)\right\|,
\end{aligned}
\end{equation}
where we used \cref{eq:curve_bound_J} and the definitions of $M_c(\delta)$ and $L_f(\delta)$.

For the dual variables, by \cref{eq:RegPPD:dual1} and \cref{eq:KKT_form2:dual1}, and using the nonexpansiveness of projection onto convex sets,
$$
\begin{aligned}
&\big\|
\hat\lambda_\tau-\lambda^\ast_\tau
\big\|^2 \\
=\ &
\big\|
(1-\eta\alpha\epsilon)
\big(
\hat\lambda_{\tau-1}-\lambda^\ast_\tau\big)
-\eta\alpha\epsilon
\big(
\lambda^\ast_\tau-\lambda_{\mathrm{prior}}\big)
+\eta\alpha
\left(
f_\tau(\hat x_{\tau-1})-f_\tau(x^\ast_\tau)\right)
\big\|^2 \\
=\ &
\left\|(1-\eta\alpha\epsilon)
\big(
\hat\lambda_{\tau-1}-\lambda^\ast_\tau\big)
-\eta\alpha\epsilon
\big(
\lambda^\ast_\tau-\lambda_{\mathrm{prior}}\big)
\right\|^2
+\eta^2\alpha^2
\left\|
f_\tau(\hat x_{\tau-1})-f_\tau(x^\ast_\tau)\right\|^2 \\
&+2\eta\alpha
\left((1-\eta\alpha\epsilon)
\big(
\hat\lambda_{\tau-1}-\lambda^\ast_\tau\big)
-\eta\alpha\epsilon
\big(
\lambda^\ast_\tau-\lambda_{\mathrm{prior}}\big)
\right)^T
\left(
f_\tau(\hat x_{\tau-1})-f_\tau(x^\ast_\tau)\right).
\end{aligned}
$$
Noting that $\|\hat x_{\tau-1}-x^\ast_\tau\|\leq \left\|\hat z_{\tau-1}-z^\ast_\tau\right\|_{\eta}\leq\delta$, by the definitions of $M_{\lambda}$ and $L_f(\delta)$, we get
\begin{equation}\label{eq:proof_approx_contraction:dual_diff}
\begin{aligned}
&\big\|
\hat\lambda_\tau-\lambda^\ast_\tau
\big\|^2 \\
\leq\ &
\left(
(1-\eta\alpha\epsilon)
\big\|
\hat\lambda_{\tau-1}-\lambda^\ast_\tau
\big\|
+\eta\alpha\epsilon M_{\lambda}
\right)^2+\eta^2\alpha^2 L_f^2(\delta)\left\|\hat x_{\tau-1}-x^\ast_\tau\right\|^2
\\
&+2\eta\alpha(1-\eta\alpha\epsilon)
\big(
\hat\lambda_{\tau-1}-\lambda^\ast_\tau
\big)^T
\left(
f_\tau(\hat x_{\tau-1})-f_\tau(x^\ast_\tau)
\right)
\\
&+2\eta^2\alpha^2\epsilon M_{\lambda} L_f(\delta)\left\|\hat x_{\tau-1}-x^\ast_\tau\right\|.
\end{aligned}
\end{equation}
By the convexity of the components of $f^c_{\tau}$ and noting that $\hat\lambda_{\tau-1}\in\mathbb{R}^m_+$, we have
$$
\hat\lambda_{\tau-1}^T
\left(f^{c}_\tau(\hat x_{\tau-1})+J_{f^c_\tau}(\hat x_{\tau-1})(x^\ast_\tau-\hat x_{\tau-1})-f^{c}_\tau(x^\ast_\tau)\right)
\leq 0.
$$
Next, recall the definition of $\overline{H}_{f^c_i}$ in~\cref{eq:def_Hfc}, and notice that by using the Fundamental Theorem of Calculus and integration by parts, we have
$$
u^T \left [ \overline{H}_{f^c_i}(u,\tau\Delta_T) \right ] u = 2 \left (f^c_{\tau, i} (x^\ast_\tau + u) - f^c_{\tau, i} (x^\ast_\tau) - u^T \nabla f^c_{\tau, i} (x^\ast_\tau) \right )
$$
for any $u \in \mathbb{R}^n$. Thus,
$$
\begin{aligned}
& {\lambda^\ast_\tau}^T
\left(
f^{c}_\tau(x^\ast_\tau)
+J_{f^c_\tau}(x^\ast_{\tau})(\hat x_{\tau-1}-x^\ast_\tau)
-f^{c}_\tau(\hat x_{\tau-1})\right) \\
 &
\hspace{.3cm}= -\frac{1}{2}(\hat x_{\tau-1}-x^\ast_\tau)^T
\left(\sum_{i=1}^m \lambda^\ast_{\tau,i}
B_{f^c_{\tau,i}}\right)(\hat x_{\tau-1}-x^\ast_\tau),
\end{aligned}
$$
where $B_{f^c_{\tau,i}}:=\overline{H}_{f^c_i}(\hat x_{\tau-1}-x^\ast_\tau,\tau\Delta_T)$. Therefore,
\begin{equation}\label{eq:proof_approx_contraction:temp2}
\begin{aligned}
&
\big(
\hat\lambda_{\tau-1}-\lambda^\ast_\tau
\big)^T
\left(
f_\tau(\hat x_{\tau-1})-f_\tau(x^\ast_\tau)
\right)
\\
&\quad
-(\hat x_{\tau-1}-x^\ast_\tau)^T
\Big(
J_{f^{nc}_\tau}(\hat x_{\tau-1})^T
\big(
\hat\lambda_{\tau-1}-\lambda^\ast_\tau\big) \\
&\qquad\qquad\qquad\qquad
+J_{f^c_\tau}(\hat x_{\tau-1})^T\hat\lambda_{\tau-1}
-J_{f^c_\tau}(x^\ast_{\tau})^T\lambda^\ast_\tau
\Big)
\\
=\ &
\big(
\hat\lambda_{\tau-1}-\lambda^\ast_\tau
\big)^T
\left(
f^{nc}_\tau(\hat x_{\tau-1})
+J_{f^{nc}_\tau}(\hat x_{\tau-1})(x^\ast_\tau-\hat x_{\tau-1})
-f^{nc}_\tau(x^\ast_\tau)\right)
\\
&\quad
+ \hat\lambda_{t-1}^T
\left(f^{c}_\tau(\hat x_{\tau-1})+J_{f^c_\tau}(\hat x_{\tau-1})(x^\ast_\tau-\hat x_{\tau-1})-f^{c}_\tau(x^\ast_\tau)\right)
\\
&\quad
+{\lambda^\ast_\tau}^T
\left(
f^{c}_\tau(x^\ast_\tau)
+J_{f^c_\tau}(x^\ast_{\tau})(\hat x_{\tau-1}-x^\ast_\tau)
-f^{c}_\tau(\hat x_{\tau-1})\right)
\\
\leq\ &
\frac{M_{nc}(\delta)}{2}\left\|\hat x_{\tau-1}-x^\ast_\tau\right\|^2
\big\|
\hat\lambda_{\tau-1}-\lambda^\ast_\tau
\big\|
\\
&-\frac{1}{2}(\hat x_{\tau-1}-x^\ast_\tau)^T
\left(\sum_{i=1}^m \lambda^\ast_{\tau,i}
B_{f^c_{\tau,i}}\right)(\hat x_{\tau-1}-x^\ast_\tau)
\\
\leq\ &
\frac{\sqrt{\eta}}{4}\delta M_{nc}(\delta)\left\|\hat z_{\tau-1}-z^\ast_\tau\right\|_{\eta}^2
-\frac{1}{2}(\hat x_{\tau-1}-x^\ast_\tau)^T
\left(\sum_{i=1}^m \lambda^\ast_{\tau,i}
B_{f^c_{\tau,i}}\right)(\hat x_{\tau-1}-x^\ast_\tau),
\end{aligned}
\end{equation}
where we used \cref{eq:curve_bound_f} and the definition of $M_{nc}(\delta)$ in the second step, and used
$$
\begin{aligned}
&\left\|\hat x_{\tau-1}-x^\ast_\tau\right\|
\big\|
\hat\lambda_{\tau-1}-\lambda^\ast_\tau
\big\| \\
\leq\ &
\frac{1}{2}
\left(
\sqrt{\eta}\left\|\hat x_{\tau-1}-x^\ast_\tau\right\|^2
+\frac{1}{\sqrt{\eta}}
\big\|
\hat\lambda_{\tau-1}-\lambda^\ast_\tau
\big\|^2
\right)
=
\frac{\sqrt{\eta}}{2}\left\|\hat z_{\tau-1}-z^\ast_\tau\right\|_{\eta}^2
\end{aligned}
$$
in the last step.

Now we take the sum of \cref{eq:proof_approx_contraction:primal_diff} and \cref{eq:proof_approx_contraction:dual_diff} and use \cref{eq:proof_approx_contraction:temp1} to bound $\left\|\hat z_\tau-z^\ast_\tau\right\|_{\eta}^2$ by
\begin{equation}\label{eq:proof_approx_contraction:z_diff_1}
\begin{aligned}
&\left\|\hat z_\tau-z^\ast_\tau\right\|_{\eta}^2
=
\|\hat x_\tau-x^\ast_\tau\|^2
+\eta^{-1} \big\|
\hat\lambda_\tau - \lambda^\ast_\tau
\big\|^2
\\
\leq\ &
(\hat x_{\tau-1}-x^\ast_\tau)^T\left(I-\alpha B_{\mathcal{L}^{nc}_\tau}\right)^2(\hat x_{\tau-1}-x^\ast_\tau)
\\
&+
\alpha^2\left(L_f(\delta)\big\|
\hat\lambda_{\tau-1}-\lambda^\ast_\tau
\big\|
+M_c(\delta)\left\|\hat x_{\tau-1}-x^\ast_\tau\right\|
\sup_{t\in[0,S]}\left\|\lambda^\ast(t)\right\|\right)^2
\\
&-2\alpha
(\hat x_{\tau-1}-x^\ast_\tau)^T
\left(I-\alpha B_{\mathcal{L}^{nc}_\tau}\right)
\big(
J_{f^{nc}_\tau}(\hat x_{\tau-1})^T
\big(
\hat\lambda_{\tau-1}-\lambda^\ast_\tau
\big)
\\
&\qquad\qquad\qquad\qquad\qquad\qquad\qquad
+J_{f^c_\tau}(\hat x_{\tau-1})^T\hat\lambda_{\tau-1}
-J_{f^c_\tau}(x^\ast_{\tau})^T\lambda^\ast_\tau
\big)
\\
&+
\eta^{-1}\left(
(1-\eta\alpha\epsilon)
\big\|
\hat\lambda_{\tau-1}-\lambda^\ast_\tau
\big\|
+\eta\alpha\epsilon M_{\lambda}
\right)^2
+\eta\alpha^2 L_f^2(\delta)\left\|\hat x_{\tau-1}-x^\ast_\tau\right\|^2
\\
&+2\alpha(1-\eta\alpha\epsilon)
\big(
\hat\lambda_{\tau-1}-\lambda^\ast_\tau
\big)^T
\left(
f_\tau(\hat x_{\tau-1})-f_\tau(x^\ast_\tau)\right)
\\
&+2\eta\alpha^2\epsilon M_{\lambda} L_f(\delta)\left\|\hat x_{\tau-1}-x^\ast_\tau\right\|.
\end{aligned}
\end{equation}
Notice that
\begin{equation}\label{eq:proof_approx_contraction:temp3}
\begin{aligned}
&\alpha^2\left(L_f(\delta)\big\|
\hat\lambda_{\tau-1}-\lambda^\ast_\tau
\big\|
+M_c(\delta)\left\|\hat x_{\tau-1}-x^\ast_\tau\right\|\sup_{t\in[0,S]}\left\|\lambda^\ast(t)\right\|\right)^2
\\
&\quad
+\eta\alpha^2 L_f^2(\delta)\left\|\hat x_{\tau-1}-x^\ast_\tau\right\|^2
\\
\leq\ &
\alpha^2
\left(\sqrt{\eta}L_f(\delta)
+M_c(\delta)\sup_{t\in[0,S]}\left\|\lambda^\ast(t)\right\|\right)^2
\left\|\hat z_{\tau-1}-z^\ast_\tau\right\|_{\eta}^2
\\
=\ &
\alpha^2 D^2(\delta,\eta)\left\|\hat z_{\tau-1}-z^\ast_\tau\right\|_{\eta}^2.
\end{aligned}
\end{equation}
Moreover,
$$
\begin{aligned}
&(1-\eta\alpha\epsilon)
\big(
\hat\lambda_{\tau-1}-\lambda^\ast_\tau 
\big)^T
\left(
f_\tau(\hat x_{\tau-1})-f_\tau(x^\ast_\tau)\right)
 \\
&\quad
-
(\hat x_{\tau-1}-x^\ast_\tau)^T
\left(I-\alpha B_{\mathcal{L}^{nc}_\tau}\right)
\big(
J_{f^{nc}_\tau}(\hat x_{\tau-1})^T
\big(
\hat\lambda_{\tau-1}-\lambda^\ast_\tau
\big)
 \\
&\qquad\qquad\qquad\qquad\qquad\qquad\qquad
+J_{f^c_\tau}(\hat x_{\tau-1})^T\hat\lambda_{\tau-1}
-J_{f^c_\tau}(x^\ast_{\tau})^T\lambda^\ast_\tau
\big)
 \\
=\ &
(1-\eta\alpha\epsilon)\Big[
\big(
\hat\lambda_{\tau-1}-\lambda^\ast_\tau
\big)^T
\left(
f_\tau(\hat x_{\tau-1})-f_\tau(x^\ast_\tau)\right)
 \\
&\quad\ \ 
-(\hat x_{\tau-1}-x^\ast_\tau)^T
\left(
J_{f^{nc}_\tau}(\hat x_{\tau-1})^T
\big(
\hat\lambda_{\tau-1}-\lambda^\ast_\tau
\big)
+J_{f^c_\tau}(\hat x_{\tau-1})^T\hat\lambda_{\tau-1}
-J_{f^c_\tau}(x^\ast_{\tau})^T\lambda^\ast_\tau\right)
\!\!\Big]
 \\
&\quad
-\alpha(\hat x_{\tau-1}-x^\ast_\tau)^T
\left(\eta\epsilon I-B_{\mathcal{L}^{nc}_\tau}\right)
\big(
J_{f^{nc}_\tau}(\hat x_{\tau-1})^T
\big(
\hat\lambda_{\tau-1}-\lambda^\ast_\tau
\big)
 \\
&\qquad\qquad\qquad\qquad\qquad\qquad\qquad
+J_{f^c_\tau}(\hat x_{\tau-1})^T\hat\lambda_{\tau-1}
-J_{f^c_\tau}(x^\ast_{\tau})^T\lambda^\ast_\tau
\big)
 \\
\leq\ &
(1-\eta\alpha\epsilon)
\!\left[\!
\frac{\sqrt{\eta}}{4}\delta M_{nc}(\delta)
\!\left\|\hat z_{\tau-1}-z^\ast_\tau\right\|_{\eta}^2
-\frac{1}{2}(\hat x_{\tau-1}-x^\ast_\tau)^T
\!\!\left(\sum_{i=1}^m \lambda^\ast_{\tau,i}
B_{f^c_{\tau,i}}\right)\!
(\hat x_{\tau-1}-x^\ast_\tau)
\!\right]
 \\
&\quad
+\alpha\left\|\eta\epsilon I-B_{\mathcal{L}^{nc}_\tau}\right\|
\left\|\hat x_{\tau-1}-x^\ast_\tau\right\|
 \\
&\quad\quad\quad\times
\left(L_f(\delta)\big\|
\hat\lambda_{\tau-1}-\lambda^\ast_\tau
\big\|
+M_c(\delta)\left\|\hat x_{\tau-1}-x^\ast_\tau\right\|\sup_{t\in[0,S]}\left\|\lambda^\ast(t)\right\|\right)
 \\
\leq\ &
(1-\eta\alpha\epsilon)
\!\left[\!
\frac{\sqrt{\eta}}{4}\delta M_{nc}(\delta)
\!\left\|\hat z_{\tau-1}-z^\ast_\tau\right\|_{\eta}^2
-\frac{1}{2}(\hat x_{\tau-1}-x^\ast_\tau)^T
\!\!\left(\sum_{i=1}^m \lambda^\ast_{\tau,i}
B_{f^c_{\tau,i}}\right)\!
(\hat x_{\tau-1}-x^\ast_\tau)
\!\right]
 \\
&\quad
+\alpha\left\|\eta\epsilon I-B_{\mathcal{L}^{nc}_\tau}\right\|
 D(\delta,\eta)
\left\|\hat z_{\tau-1}-z^\ast_\tau\right\|_{\eta}^2.
\end{aligned}
$$
Therefore by plugging \cref{eq:proof_approx_contraction:temp3} and the above inequality into 
\cref{eq:proof_approx_contraction:z_diff_1}, we get
$$
\begin{aligned}
& \left\|\hat z_\tau-z^\ast_\tau\right\|_{\eta}^2 \\
\leq\ &
(\hat x_{\tau-1}-x^\ast_\tau)^T
\left[\left(I-\alpha B_{\mathcal{L}^{nc}_\tau}\right)^2
-\alpha(1-\eta\alpha\epsilon)\sum_{i=1}^m \lambda^\ast_{\tau,i}
B_{f^c_{\tau,i}}\right](\hat x_{\tau-1}-x^\ast_\tau) \\
&
+(1-\eta\alpha\epsilon)^2 \eta^{-1}
\big\|
\hat \lambda_{\tau-1}-\lambda^\ast_\tau
\big\|^2 
+ \alpha^2 D^2(\delta,\eta)\left\|\hat z_{\tau-1}-z^\ast_\tau\right\|_{\eta}^2 \\
&
+\alpha(1-\eta\alpha\epsilon)\frac{\sqrt{\eta}}{2}\delta M_{nc}(\delta)\left\|\hat x_{\tau-1}-x^\ast_\tau\right\|^2
+
2\alpha\left\|\eta\epsilon I-B_{\mathcal{L}^{nc}_\tau}\right\|
D(\delta,\eta)
\left\|\hat z_{\tau-1}-z^\ast_\tau\right\|_{\eta}^2
\\
&+2\sqrt{\eta}\alpha\epsilon M_{\lambda}
\left(
\frac{1-\eta\alpha\epsilon}{\sqrt{\eta}}
\big\|
\hat \lambda_{\tau-1}-\lambda^\ast_\tau
\big\|
+\sqrt{\eta}\alpha L_f(\delta)
\left\|\hat x_{\tau-1}-x^\ast_\tau\right\|
\right) \\
&
+\eta\alpha^2\epsilon^2 M_{\lambda}^2.
\end{aligned}
$$
It's not hard to see that
$$
\begin{aligned}
&\left(
\frac{1-\eta\alpha\epsilon}{\sqrt{\eta}}
\big\|
\hat \lambda_{\tau-1}-\lambda^\ast_\tau
\big\|
+\sqrt{\eta}\alpha L_f(\delta)
\left\|\hat x_{\tau-1}-x^\ast_\tau\right\|
\right) \\
\leq\ &
\max\left\{1-\eta\alpha\epsilon,\,
\sqrt{\eta}\alpha L_f(\delta)\right\}
\left\|\hat z_{\tau-1}-z^\ast_\tau\right\|_{\eta},
\end{aligned}
$$
and by the definition of $\rho(\delta,\alpha,\eta,\epsilon)$ and $\kappa(\delta,\alpha,\eta,\epsilon)$, we get
$$
\begin{aligned}
&\left\|\hat z_\tau-z^\ast_\tau\right\|_{\eta}^2 \\
\leq\ &
\rho^2(\delta,\alpha,\eta,\epsilon)
\left\|\hat z_{\tau-1}-z^\ast_\tau\right\|_{\eta}^2
+\eta\alpha^2\epsilon^2 M_{\lambda}^2
\\
&+2\sqrt{\eta}\alpha\epsilon
M_{\lambda}\cdot
\max\left\{1-\eta\alpha\epsilon,
\sqrt{\eta}\alpha L_f(\delta)\right\}
\left\|\hat z_{\tau-1}-z^\ast_\tau\right\|_{\eta} \\
\leq\ &
\left(
\rho(\delta,\alpha,\eta,\epsilon)
\left\|\hat z_{\tau-1}-z^\ast_\tau\right\|_{\eta}
+\kappa(\delta,\alpha,\eta,\epsilon)
\sqrt{\eta}\alpha\epsilon M_{\lambda}
\right)^2,
\end{aligned}
$$
which is just \cref{eq:approx_contraction}.

Now let $\delta$, $\eta$ and $\epsilon$ be fixed. With the help of \cref{lemma:sup_continuous}, it can be shown that, if we temporarily allow $\alpha$ to take arbitrary values in $\mathbb{R}$, then the function $\alpha\mapsto \rho^{(\mathrm{P})}(\delta,\alpha,\eta,\epsilon)$ is a continuous function over $\alpha\in\mathbb{R}$, and so
$$
\lim_{\alpha\rightarrow 0^+}\rho^{(\mathrm{P})}(\delta,\alpha,\alpha\eta,\epsilon)= \rho^{(\mathrm{P})}(\delta,0,\eta,\epsilon)=1.
$$
Then by the definition of $\rho(\delta,\alpha,\eta,\epsilon)$, it's straightforward to get $\rho(\delta,\alpha,\eta,\epsilon)\rightarrow 1$ when $\alpha\rightarrow 0^+$, which further leads to
$$
\lim_{\alpha\rightarrow 0^+}\kappa(\delta,\alpha,\eta,\epsilon)
= 1.
$$
We also have
$$
\begin{aligned}
\frac{\max\left\{1-\eta\alpha\epsilon,\sqrt{\eta}\alpha
L_f(\delta)\right\}}
{\rho(\delta,\alpha,\eta,\epsilon)}
\leq\ &
\frac{1-\eta\alpha\epsilon+\alpha D(\delta,\eta)}
{\rho(\delta,\alpha,\eta,\epsilon)} \\
\leq\ &
\frac{\sqrt{2\left((1-\alpha\eta\epsilon)^2+\alpha^2 D^2(\delta,\eta)\right)}}{\rho(\delta,\alpha,\eta,\epsilon)}
\leq\sqrt{2}
\end{aligned}
$$
which implies that $\kappa(\delta,\alpha,\eta,\epsilon)\leq\sqrt{2}$.

\section{Proof of \texorpdfstring{\cref{lemma:sup_continuous}}{Lemma \ref{lemma:sup_continuous}}}
\label{sec:proof_sup_continuous}

\paragraph{Part 1} Let $a\in\mathbb{R}$ be arbitrary, and consider the set
$$
\begin{aligned}
g^{-1}[(a,+\infty)]&=\{y\in Y: g(y)>a\} \\
&=\{y\in Y:\exists x\in X\textrm{ such that }f(x,y)>a\} \\
&=\bigcup_{x\in X}\{y\in Y: f(x,y)>a\}.
\end{aligned}
$$
The continuity of $f$ implies that $\{y\in Y: f(x,y)>a\}$ is open for each $x\in X$, and so $g^{-1}[(a,+\infty)]$ is open. By the arbitrariness of $a\in\mathbb{R}$, we see that $g$ is lower semicontinuous.

Now let $y^0\in Y$ be arbitrary, and let $(y_n)_{n\in\mathbb{N}}$ be any sequence in $Y$ such that $y_n\rightarrow y^0$ and $\lim_{n\rightarrow\infty} g(y_n)$ exists. Since $X$ is compact and $f$ is continuous, we can see that for any $n$, there is some $x_n\in X$ such that $g(y_n)=f(x_n,y_n)$. By the compactness of $X$, we can find a subsequence $(x_{k_n})_{n\in\mathbb{N}}$ such that $x_{k_n}\rightarrow x^0$ for some $x^0\in X$ as $n\rightarrow\infty$. Then
$$
\lim_{n\rightarrow\infty} g(y_n)
=\lim_{n\rightarrow\infty} f(x_n,y_n)
=\lim_{n\rightarrow\infty} f(x_{k_n},y_{k_n})=f(x^0,y^0)
\leq g(y^0),
$$
where the third equality follows from the continuity of $f$, and the last inequality follows from the definition of $g$. By the arbitrariness of the sequence $(y_n)_{n\in\mathbb{N}}$, we get
$$
\limsup_{y\rightarrow y^0} g(y)\leq g(y^0),
$$
where $y^0\in Y$ is arbitrary. Now we can conclude that $f$ is upper semicontinuous, and thus continuous on $Y$.

\paragraph{Part 2} Now suppose $f:B_n(R)\times V\rightarrow\mathbb{R}$ is continuous. Define the auxiliary function $\tilde f:B_n(R)\times (0,R)\times V\rightarrow\mathbb{R}$ by
$$
\tilde f(u,r,v)=f\left(\mathcal{P}_{B_n(r)}(u), v\right).
$$
The function $(u,r)\mapsto \mathcal{P}_{B_n(r)}(u)$ is continuous (in fact Lipschitz), as for any $(u_1,r_1)$ and $(u_2,r_2)$ in $B_n(R)\times (0,R)$, we have
$$
\begin{aligned}
&\left\|\mathcal{P}_{B_n(r_1)}(u_1)
-\mathcal{P}_{B_n(r_2)}(u_2)\right\| \\
\leq\ &
\left\|\mathcal{P}_{B_n(r_1)}(u_1)-\mathcal{P}_{B_n(r_1)}(u_2)\right\|+
\left\|\mathcal{P}_{B_n(r_1)}(u_2)
-\mathcal{P}_{B_n(r_2)}(u_2)\right\| \\
\leq\ &
\|u_1-u_2\|+|r_1-r_2|.
\end{aligned}
$$
Therefore $\tilde f$ is also a continuous function. Moreover,
$$
g(r,v)=\sup_{u:\|u\|\leq r} f(u,v)
=\sup_{u\in B_n(R)} \tilde f(u,r,v).
$$
By the compactness of $B_n(R)$ and the first part of \cref{lemma:sup_continuous}, we conclude that $g$ is continuous.

\section{Proof of \texorpdfstring{\cref{lemma:gronwall_type}}{Lemma \ref{lemma:gronwall_type}}}
\label{sec:proof_gronwall_type}

The proof is directly based on the following lemma.
\begin{lemma}[{\cite{willett1965discrete}}]\label{lemma:gronwall_type_general}
Let $I$ be a closed interval with zero as left endpoint. Let $u(t)$ be a continuous nonnegative function that satisfies the integral inequality
$$
u(t)\leq u_0+\int_{0}^t w(s)u^p(s)\,ds
$$
where $w(t)$ is a continuous nonnegative function on $I$. For $0\leq p <1$ we have
$$
u(t)
\leq\left(u_0^{1-p}
+(1-p)\int_0^t w(s)\,ds
\right)^{\frac{1}{1-\beta}},
$$
and for $p=1$, we have
$$
u(t)\leq u_0\exp\int_0^t w(s)\,ds.
$$
\end{lemma}

Let us define $u(t)=e^{2\beta t}v^2(t)$. Then
$$
\begin{aligned}
\frac{d}{dt}u(t)
&=2\beta e^{2\beta t}v^2(t)
+e^{2\beta t}\frac{d}{dt}\left(v^2(t)\right) \\
&\leq
2\beta e^{2\beta t}v^2(t)
+2e^{2\beta t}(\alpha v(t)-\beta v^2(t)) \\
&=2\alpha e^{2\beta t}v(t)=2\alpha e^{\beta t}\sqrt{u(t)}
\end{aligned}
$$
for almost all $t\in [0,S]$. Therefore by \cref{lemma:gronwall_type_general},
$$
u(t)\leq\left( \sqrt{u(0)}+\alpha\int_0^t e^{\beta s}\,ds\right)^{2}
=\left( \sqrt{u(0)}+\frac{\alpha}{\beta}(e^{\beta t}-1)\right)^{2},
$$
and by the definition of $u(t)$, we get the desired result.

\section{Proof of \texorpdfstring{\cref{theorem:cont_time_limit}}{Theorem \ref{theorem:cont_time_limit}}}
\label{sec:proof_cont_time_limit}

The following lemma provides conditions for the closedness of the normal cone of a time-varying convex set.
\begin{lemma}\label{lemma:subdifferential_closed}
Suppose $\mathcal{C}:[0,S]\rightarrow2^{\mathbb{R}^p}$ is a $\kappa$-Lipschitz set-valued map, and $\mathcal{C}(t)$ is closed and convex for each $t\in[0,S]$. Then the set
$$
\left\{(z,t,y)\in\mathbb{R}^p\times[0,S]\times\mathbb{R}^p:
z\in\mathcal{C}(t),y\in N_{\mathcal{C}(t)}(z)
\right\}
$$
is closed; in other words, the set-valued map $(z,t)\mapsto N_{\mathcal{C}(t)}(z)$ is closed.
\end{lemma}
\begin{proof}[Proof of \cref{lemma:subdifferential_closed}]
Denote
$$
A=\left\{(z,t,y)\in\mathbb{R}^p\times[0,S]\times\mathbb{R}^p:
z\in\mathcal{C}(t),y\in N_{\mathcal{C}(t)}(z)
\right\}.$$
Let $(z_k,t_k,y_k),\,k\in\mathbb{N}$ be a sequence in $A$ that converges to some $(z,t,y)$. We then have
$$
\begin{aligned}
\inf_{u\in\mathcal{C}(t)}\|z-u\|
&\leq
\|z-z_k\|
+\inf_{u\in\mathcal{C}(t)}\left\|
z_k-u\right\|
\leq
\|z-z_k\|
+d_H\left(\mathcal{C}(t),\mathcal{C}(t_k)\right) \\
&\leq \|z-z_k\|+\kappa|t-t_k|
\end{aligned}
$$
where $k\in\mathbb{N}$ is arbitrary. By letting $k\rightarrow\infty$ we get $\inf_{u\in\mathcal{C}(t)}\|z-u\|=0$, and since $\mathcal{C}(t)$ is closed, we have $z\in\mathcal{C}(t)$.

Now let $w\in\mathcal{C}(t)$ be arbitrary, and denote $w_k=\mathcal{P}_{\mathcal{C}(t_k)}(w)$.
We have $\|w_k-w\|\leq\kappa|t-t_k|$ as $\mathcal{C}$ is $\kappa$-Lipschitz. Then
$$
y_k^T
(w-z_k)
\leq
y_k^T
(w-w_k)
+y_k^T
(w_k-z_k)
\leq
y_k^T
(w-w_k)
\leq \kappa\|y_k\||t-t_k|,
$$
where we used $y_k^T(w_k-z_k)\leq 0$ since $y_k\in N_{\mathcal{C}(t_k)}(z_k)$ and $w_k\in\mathcal{C}(t_k)$. By letting $k\rightarrow\infty$, we get $y^T(w-z)\leq 0$. By the arbitrariness of $w\in\mathcal{C}(t)$, we see that $y\in N_{\mathcal{C}(t)}(z)$, and therefore $(z,t,y)\in A$.
\end{proof}

The following lemma constitutes the core step of the proof of \cref{theorem:cont_time_limit}, which has been studied in the literature on perturbed sweeping process \cite{castaing1993evolution,adly2014convex}. We provide its proof here for completeness.

\begin{lemma}\label{lemma:discrete_to_cont}
Suppose $\Phi:\mathbb{R}^p\times[0,S]\rightarrow\mathbb{R}^p$ and $\mathcal{C}:[0,S]\rightarrow2^{\mathbb{R}^p}$, and that
\begin{enumerate}[itemindent=13.5pt,leftmargin=0pt]
\item $\mathcal{C}$ is $\kappa_1$-Lipschitz, and for each $t\in[0,S]$, $\mathcal{C}(t)$ is closed and convex,
\item $\Phi$ is continuous when restricted to the set $\bigcup_{t\in[0,S]}\mathcal{C}(t)\times[0,S]$, and there exists some $\kappa_2>0$ such that
$$
\|\Phi(z,t)\|\leq \kappa_2(1+\|z\|),\qquad
\forall (z,t)\in\bigcup_{t\in[0,S]}\mathcal{C}(t)\times[0,S].
$$
\end{enumerate}
Let $\hat z_0 \in \mathcal{C}(0)$ be arbitrary, and for each $K\in\mathbb{N}$, Define $\hat z^{(K)}_\tau,\,\tau\in\{0,1,2,\ldots,K\}$ by
$$
\begin{aligned}
\hat z^{(K)}_0
&= \hat z_0, \\
\hat z^{(K)}_\tau
&= \mathcal{P}_{\mathcal{C}(\tau\Delta_K)}
\left[ \hat z^{(K)}_{\tau-1} + \Delta_K \Phi\Big(\hat z^{(K)}_{\tau-1},\tau\Delta_K\Big)
\right],\\
\end{aligned}
$$
where $\Delta_K:=S/K$, and for $t\in[0,S]$, define
\begin{equation}\label{eq:discrete_path_lemma}
\hat z^{(K)}(t)
=\frac{\tau\Delta_K -t}{\Delta_K}\hat z^{(K)}_{\tau-1} + \frac{t-(\tau-1)\Delta_K}{\Delta_K}\hat z^{(K)}_{\tau}
\end{equation}
if $t\in[(\tau-1)\Delta_K, \tau\Delta_K]$. Then, if we keep $S$ constant and let $K\rightarrow\infty$, the sequence $\big(\hat z^{(K)}\big)_{K\in\mathbb{N}}$ defined in \cref{eq:discrete_path_lemma} has a convergent subsequence, and any convergent subsequence converges uniformly to a Lipschitz continuous $\hat z$ that satisfies
\begin{equation}\label{eq:diff_incl_lemma}
\begin{aligned}
\hat z(0)&=\hat z_0, \\
-\frac{d}{dt}\hat z(t)+\Phi(\hat z(t),t)
&\in N_{\mathcal{C}(t)}(\hat z(t)),\quad \forall t\in[0,S]\,a.e.
\end{aligned}
\end{equation}
\end{lemma}
\begin{proof}[Proof of \cref{lemma:discrete_to_cont}]
Let $A=\bigcup_{t\in[0,S]}\mathcal{C}$. For each $\tau\geq 1$, let
$$
\begin{aligned}
u_\tau^{(K)} &= \mathcal{P}_{\mathcal{C}(\tau\Delta_K)}\left[\hat z_{\tau-1}^{(K)}\right].
\end{aligned}
$$
Since $\left\|u_\tau^{(K)}-\hat z_{\tau-1}^{(K)}\right\|\leq\kappa_1\Delta_K$ by the $\kappa_1$-Lipschitz continuity of $\mathcal{C}$, we get
\begin{equation}\label{eq:first_order:sweeping_lemma_temp}
\begin{aligned}
\left\|\hat z^{(K)}_\tau-\hat z^{(K)}_{\tau-1}\right\|
\leq\ &
\left\|\mathcal{P}_{\mathcal{C}(\tau\Delta_K)}
\left[\hat z^{(K)}_{\tau-1}
+\Delta_K\Phi\Big(\hat z^{(K)}_{\tau-1},\tau\Delta_K\Big)\right]
-\mathcal{P}_{\mathcal{C}(\tau\Delta_K)}
\left[\hat z^{(K)}_{\tau-1}\right]\right\|
\\
& +
\left\|
\mathcal{P}_{\mathcal{C}(\tau\Delta_K)}
\left[\hat z^{(K)}_{\tau-1}\right]
-\hat z^{(K)}_{\tau-1}
\right\|
\\
\leq\ &
\Delta_K\left\|\Phi\Big(\hat z^{(K)}_{\tau-1},\tau\Delta_K\Big)\right\|
+\left\|u^{(K)}_\tau
-\hat z_{\tau-1}^{(K)}\right\| 
\\
\leq\ &
\Delta_K(\kappa_1+\kappa_2)
+\Delta_K\kappa_2\left\|\hat z^{(K)}_{\tau-1}\right\|,
\end{aligned}
\end{equation}
and so
$$
\left\|\hat z^{(K)}_\tau\right\|
\leq\Delta_K(\kappa_1+\kappa_2)
+(1+\Delta_K\kappa_2)\left\|\hat z^{(K)}_{\tau-1}\right\|
$$
for any $\tau=1,\ldots,K$. By induction we can see that
$$
\left\|\hat z^{(K)}_\tau\right\|
\leq \left(1+\Delta_K\kappa_2\right)^\tau
\left(\left\|\hat z_0\right\|
+\frac{\kappa_1+\kappa_2}{\kappa_2}\right)
-\frac{\kappa_1+\kappa_2}{\kappa_2}
$$
holds for all $\tau=0,\ldots,K$, and since $\Delta_K=S/K$, we get
$$
\begin{aligned}
\left\|\hat z^{(K)}_\tau\right\|
&\leq \left(1+\frac{S\kappa_2}{K}\right)^K
\left(\left\|\hat z_0\right\|
+\frac{(\kappa_1+\kappa_2)}{\kappa_2}\right)
-\frac{\kappa_1+\kappa_2}{\kappa_2} \\
&\leq e^{S\kappa_2}\left(\left\|\hat z_0\right\|
+\frac{(\kappa_1+\kappa_2)}{\kappa_2}\right)
-\frac{\kappa_1+\kappa_2}{\kappa_2}
=: \kappa_3
\end{aligned}
$$
for any $\tau=0,1,\ldots,K$. By plugging it back to \cref{eq:first_order:sweeping_lemma_temp}, we have
$$
\left\|\hat z^{(K)}_\tau-\hat z^{(K)}_{\tau-1}\right\|
\leq \Delta_K(\kappa_1+\kappa_2+\kappa_2\kappa_3),
$$
and consequently
$$
\left\|\frac{d}{dt}\hat z^{(K)}(t)\right\|
=\Delta_K^{-1}\left\|\hat z^{(K)}_{\lceil t/\Delta_K\rceil}-\hat z^{(K)}_{\lfloor t/\Delta_K\rfloor}\right\|
\leq \kappa_1+\kappa_2+\kappa_2\kappa_3
=:\tilde\ell 
$$
for almost every $t\in[0,S]$.

Let $D\hat z_i^{(K)}\in L^\infty([0,S])$ denote the weak derivative of the $i$'th entry of $\hat z^{(K)}$ for each $i=1,\ldots,p$. Then the sequence $\big(D\hat z_i^{(K)}\big)_{K\in\mathbb{N}}$ lies in the ball
$$
B=\left\{f\in L^\infty([0,S]):\esssup_{t\in[0,S]}|f(t)| \leq \tilde{\ell}\right\}.
$$
The Banach--Alaoglu theorem \cite{rudin1991functional} indicates that $B$ is weak* sequentially compact, and so $\big(D\hat z_i^{(K)}\big)_{K\in\mathbb{N}}$ has a convergent subsequence with respect to the weak* topology. We extract an arbitrary convergent subsequence and still denote it by $\big(D\hat z_i^{(K)}\big)_{K\in\mathbb{N}}$. Then $D\hat z^{(K)}_i\overset{w^\ast}{\rightarrow} q_i$ for some $q_i\in B$, or in other words,
$$
\int_0^T u(t)q_i(t)\,dt=\lim_{K\rightarrow\infty}\int_0^T
u(t)D\hat z_i^{(K)}(t)\,dt
$$
for all $u\in L^1([0,S])$. Consequently $\hat z^{(K)}$ converges uniformly to $\hat z$ given by
$$
\hat z(t)=\hat z_0+\int_0^t
q(s)\,ds
$$
where $q:[0,S]\rightarrow \mathbb{R}^{p}$ is the vector-valued function with entries $g_1,\ldots,g_{p}$.

Next we prove that any convergent subsequence of $\big(\hat z^{(K)}\big)_{K\in\mathbb{N}}$ converges to a limit that satisfies the differential inclusions \cref{eq:diff_incl_lemma}. We still use $\hat z$ to denote the limit of an arbitrary convergent subsequence of $\big(\hat z^{(K)}\big)_{K\in\mathbb{N}}$, and without loss of generality we assume $\hat z^{(K)}\rightarrow\hat z$ by extracting the subsequence. Since $\big(\hat z^{(K)}\big)_{K\in\mathbb{N}}$ is equi-Lipschitz, the convergence $\hat z^{(K)}\rightarrow\hat z$ is uniform, $\hat z$ is $\tilde{\ell}$-Lipschitz, and $D\hat z^{(K)}\overset{w^\ast}{\rightarrow} D\hat z$. Define
$$
\delta^{(K)}(t)=\left\lfloor\frac{t}{\Delta_K}\right\rfloor\Delta_K,\qquad
\theta^{(K)}(t)=\left\lceil\frac{t}{\Delta_K}\right\rceil\Delta_K.
$$
Then we have for almost all $t\in[0,S]$,
$$
\begin{aligned}
&-\frac{d}{dt}\hat z^{(K)}(t)+\Phi\Big(\hat z^{(K)}\big(\delta^{(K)}(t)\big),\theta^{(K)}(t)\Big) \\
=\ &
\frac{1}{\Delta_K}\left(
\hat z^{(K)}\left(\delta^{(K)}(t)\right)
+\Delta_K\Phi\Big(\hat z^{(K)}\big(\delta^{(K)}(t)\big),\theta^{(K)}(t)\Big)-\hat z^{(K)}\big(\theta^{(K)}(t)\big)
\right) \\
\in\ & N_{\mathcal{C}\left(\theta^{(K)}(t)\right)} \big(\hat z^{(K)}\big(\theta^{(K)}(t)\big)\big).
\end{aligned}
$$
In addition, for almost all $t\in[0,S]$,
$$
\left\|-\frac{d}{dt}\hat z^{(K)}(t)+\Phi\Big(\hat z^{(K)}\big(\delta^{(K)}(t)\big),\theta^{(K)}(t)\Big)\right\|
\leq \tilde{\ell}+\kappa_2(1+\kappa_3),
$$
and so
$$
\begin{aligned}
&-\frac{d}{dt}\hat z^{(K)}(t)+\Phi\Big(\hat z^{(K)}\big(\delta^{(K)}(t)\big),\theta^{(K)}(t)\Big) \\
\in\ &
N_{\mathcal{C}\big(\theta^{(K)}(t)\big)} \big(\hat z^{(K)}\big(\theta^{(K)}(t)\big)\big)
\cap B_p\big(\tilde{\ell}+\kappa_2(1+\kappa_3)\big).
\end{aligned}
$$
We denote
$$
F(z,t)=N_{\mathcal{C}\big(\theta^{(K)}(t)\big)} \big(\hat z^{(K)}\big(\theta^{(K)}(t)\big)\big)
\cap B_p\big(\tilde{\ell}+\kappa_2(1+\kappa_3)\big)
\qquad (z,t)\in\mathbb{R}^p\times[0,S].
$$
By \cref{lemma:subdifferential_closed} and \cite[Proposition 1.4.9]{aubin2009set}, the set-valued map $F$ is upper semicontinuous. Noticing that
$$
\lim_{K\rightarrow\infty}\delta^{(K)}(t)=\lim_{K\rightarrow\infty}\theta^{(K)}(t)=t,
$$
$$
\lim_{K\rightarrow\infty}\Phi\Big(\hat z^{(K)}(\delta^{(K)}(t)),\theta^{(K)}(t)\Big)=\Phi(\hat z(t),t),
$$
by \cite[Theorem 7.2.2]{aubin2009set}, we can conclude that, for almost all $t\in[0,S]$,
$$
-\frac{d}{dt}\hat z(t)+\Phi(\hat z(t),t)\in F(\hat z(t),t),
$$
which implies \cref{eq:diff_incl_lemma}.
\end{proof}

Now we are ready to finish the proof of \cref{theorem:cont_time_limit}. For each $t\in[0,S]$ and $z=(x,\lambda)\in\mathbb{R}^n\times\mathbb{R}^m$, we define
\begin{equation}\label{eq:def_Phi_map}
\Phi(z,t)
:=\beta\begin{bmatrix}
-\nabla_x c(x,t)-J_{f,x}(x,t)^T\lambda\\
\eta\left(f^{in}(x,t)-\epsilon(\lambda-\lambda_{\mathrm{prior}})\right)
\end{bmatrix}
\end{equation}
and
$$
\mathcal{C}(t)
:=\mathcal{X}(t)\times\mathbb{R}^m_+.
$$
The iterations \cref{eq:RegPPD_cont_setting} can then be formulated as
$$
\begin{aligned}
\hat z^{(T)}_0 &= \hat z_0, \\
\hat z^{(T)}_\tau &= \mathcal{P}_{\mathcal{C}(\tau\Delta_T)}
\left[\hat z^{(T)}_{\tau-1}+\Delta_T\Phi\Big(\hat z^{(T)}_{\tau-1},\tau\Delta_T\Big)\right].
\end{aligned}
$$
We check the conditions of \cref{lemma:discrete_to_cont} as follows:
\begin{enumerate}[itemindent=13.5pt,leftmargin=0pt]
\item $\mathcal{C}$ is $\kappa_1$-Lipschtz as $\mathcal{X}$ is $\kappa_1$-Lipschtz.
\item $\Phi$ is obviously continuous on $\bigcup_{t\in[0,S]}\mathcal{X}(t)\times[0,S]$. Moreover,
$$
\begin{aligned}
\left\|\Phi(z,t)\right\|
\leq\ &
\beta
\left(\|\nabla_x c(x,t)\|
+\left\|
J_{f,x}(x,t)\right\|
\left\|
\lambda
\right\|\right) \\
&+\eta\beta
\left\|
f(x,t)\right\|+
\eta\beta\epsilon\left\|
\lambda\right\|
+\eta\beta\epsilon\left\|
\lambda_{\mathrm{prior}}\right\|.
\end{aligned}
$$
Let $x_{\mathrm{aux}}\in\bigcup_{t\in[0,S]}\mathcal{X}(t)$ be arbitrary, and 
$$
\begin{aligned}
\kappa_3 &:= \sup\left\{
\left\|
J_{f,x}(x,t)\right\|:(x,t)\in\bigcup_{t\in[0,S]}\mathcal{X}(t)\times[0,S]\right\}, \\
\kappa_4 &:=
\sup_{t\in[0,S]}\left\|
f(x_{\mathrm{aux}},t)\right\|,
\end{aligned}
$$
both of which are finite. Then
$$
\left\|
f(x,t)\right\|
\leq \kappa_4+\kappa_3(\|x\|+\|x_{\mathrm{aux}}\|).
$$
By \cref{eq:cont_time_limit:nabla_c_linear_growth} and noticing that $\|x\|\leq\|z\|$ and $\|\lambda\|\leq\|z\|$, we get
$$
\begin{aligned}
\|\Phi(z,t)\|
\leq\ &\beta
(\kappa_2(1+\|z\|)
+\kappa_3\|z\|) \\
&+\eta\beta\left(\kappa_4+\kappa_3(\|z\|+\|x_{\mathrm{aux}}\|)\right)+
\eta\beta\epsilon\|z\|
+\eta\beta\epsilon\left\|\lambda_{\mathrm{prior}}\right\| \\
\leq\ & \kappa_5(1+\|z\|),
\end{aligned}
$$
where $\kappa_5$ satisfies
$$
\begin{aligned}
\kappa_5
&\geq\beta\kappa_2
+\eta\beta(\kappa_4+\kappa_3\|x_{\mathrm{aux}}\|)
+\eta\beta\epsilon\|\lambda_{\mathrm{prior}}\|, \\
\kappa_5
&\geq \beta(\kappa_2+\kappa_3)
+\eta\beta\kappa_3+\eta\beta\epsilon.
\end{aligned}
$$
\end{enumerate}
By \cref{lemma:discrete_to_cont}, the sequence of trajectories defined by
\cref{eq:discrete_path_lemma} [and consequently \cref{eq:discrete_interp_path}] then has convergent subsequences each of which converges to some Lipschitz continuous solution to \cref{eq:diff_incl_lemma}.

\section{Proof of \texorpdfstring{\cref{theorem:opt_param_cont_time}}{Theorem \ref{theorem:opt_param_cont_time}}}
\label{sec:proof_opt_param_cont_time}
We first prove that $\sigma_{\eta}$ is a continuous function of $\eta$ over $\eta\in\mathbb{R}_{++}$. Let $t_1,t_2\in[0,S]$ with $t_1\neq t_2$ be given, and define
$$
\begin{aligned}
\overline{v}(x;t_1,t_2)
:=\ &
\frac{\|z^\ast(t_2)-z^\ast(t_1)\|_{x^{-2}}}{|t_2-t_1|} \\
=\ &
\frac{\left(
\left\|x^\ast(t_2)-x^\ast(t_1)\right\|^2
+x^2
\left\|\lambda^\ast(t_2)-\lambda^\ast(t_1)\right\|^2\right)^{1/2}}{|t_2-t_1|}
\end{aligned}
$$
for $x\in\mathbb{R}_{++}$. Obviously $\overline{v}(x;t_1,t_2)$ is a convex function of $x$ over $x\in\mathbb{R}_{++}$. Then since $\sigma_{x^{-2}}$ is the supremum of $v(x;t_1,t_2)$ over $\{(t_1,t_2)\in[0,S]^2: t_1\neq t_2\}$, $\sigma_{x^{-2}}$ is also a convex function of $x$ over $x\in\mathbb{R}_{++}$. As the domain of $\sigma_{x^{-2}}$ is $\mathbb{R}_{++}$ which is open, we can conclude that $\sigma_{x^{-2}}$ is a continuous function of $x$ over $x\in\mathbb{R}_{++}$, and consequently $\sigma_{\eta}$ is a continuous function of $\eta$ over $\eta\in\mathbb{R}_{++}$.

\paragraph{Part 1.}
The proof uses the same approach as in proving \cref{theorem:feas_param}.
Let $R> \bar\delta$ be arbitrary, and define
$$
f_R(\delta,\beta,\eta,\epsilon)
=\delta\gamma(\delta,\eta,\epsilon)
-\sqrt{\eta}\epsilon M_{\lambda} - \beta^{-1}\sigma_{\eta}.
$$
We consider two cases.
\begin{enumerate}[itemindent=13.5pt,leftmargin=0pt]
\item $M_{\lambda}\neq 0$: Let $\delta_0=\bar\delta$ and
$$
\begin{aligned}
\eta_0 =\left(\frac{2\Lambda_m(\delta_0)}{\delta_0 M_{nc}(\delta_0)}\right)^2,\qquad 
\epsilon_0 = \frac{\Lambda_m(\delta_0)}{\eta_0}.
\end{aligned}
$$
We then have
$$
f_R(\delta_0,\beta,\eta_0,\epsilon_0)
=\frac{\bar\delta}{2}
\left(\Lambda_m\big(\bar\delta)
-M_{\lambda}M_{nc}\big(\bar\delta\big)\right)
-\beta^{-1}\eta_{\eta_0}
$$
Since $\Lambda_m\big(\bar\delta)
>M_{\lambda}M_{nc}\big(\bar\delta\big)$, we can find sufficiently large $\beta_0$ so that $f_R(\delta_0,\beta_0,\eta_0,\epsilon_0)$ is greater than $0$, implying that $\mathscr{S}_{\mathrm{fp}}$ is nonempty.

\item $M_{\lambda}=0$: Let $\eta_0>0$ be arbitrary, and let $\epsilon_0=\eta_0^{-1}\Lambda_m(\delta_0)$. Then
$$
\gamma(\delta,\eta_0,\epsilon_0)=\Lambda_m(\delta)-\frac{\sqrt{\eta_0}}{4}\delta M_{nc}(\delta).
$$
by the monotonicity of $\Lambda_m(\delta)$ and $M_{nc}(\delta)$,
$$
\lim_{\delta\rightarrow 0^+} \gamma(\delta,\eta_0,\epsilon_0)=\lim_{\delta\rightarrow 0^+}\Lambda_m(\delta)\geq\Lambda_m\left(\bar\delta\right)>0.
$$
Therefore there exists some $\delta_0\in\left(0,\bar\delta\right]$ such that $\gamma(\delta_0,\eta_0,\epsilon_0)>0$, and we have
$$
f_R(\delta_0,\beta,\eta_0,\epsilon_0)
=
\delta_0 \gamma(\delta_0,\eta_0,\epsilon_0) -\beta^{-1}\sigma_{\eta_0}.
$$
Therefore we can find sufficiently large $\beta_0>0$ such that $f_R(\delta_0,\beta_0,\eta_0,\epsilon_0)$ is positive, and consequently $\mathscr{S}_{\mathrm{fp}}$ is non-empty.
\end{enumerate}

Finally, by \cref{lemma:sup_continuous} and the continuity of the function $\eta\mapsto\sigma_{\eta}$, it can be seen that $f_R(\delta,\beta,\eta,\epsilon)$ is a continuous function over $(\delta,\beta,\eta,\epsilon)\in(0,R)\times\mathbb{R}_{++}^3$. Therefore the set
$$
\mathscr{S}_{\mathrm{fp}}
\cap\left((0,R)\times\mathbb{R}_{++}^3\right)=
\{(\delta,\beta,\eta,\epsilon)\in(0,R)\times\mathbb{R}_{++}^3: f_R(\delta,\beta,\eta,\epsilon)>0\}
$$
is an open subset of $\mathbb{R}_{++}^4$, and consequently
$$
\mathscr{S}_{\mathrm{fp}}
=\bigcup_{R>\bar\delta}
\mathscr{S}_{\mathrm{fp}}
\cap\left((0,R)\times\mathbb{R}_{++}^3\right)
$$
is an open subset of $\mathbb{R}_{++}^4$.

\paragraph{Part 2.}
Denote
$$
\begin{aligned}
g_0(\eta,\epsilon)
&:=\frac{
\beta^{-1}\sigma_{\eta}
+\sqrt{\eta}\epsilon M_{\lambda}
}{\gamma(\delta,\eta,\epsilon)}, \\
g_1(\eta,\epsilon)
&:=\delta\, \gamma(\delta,\eta,\epsilon)-\sqrt{\eta}\epsilon{M}_{\lambda}-\beta^{-1}\sigma_{\eta},
\end{aligned}
$$
where $(\eta,\epsilon)\in\mathbb{R}_{++}^2$. Obviously $g_0(\eta,\epsilon)=\delta$ if $g_1(\eta,\epsilon)=0$, and $g_0(\eta,\epsilon)<\delta$ if $g_1(\eta,\epsilon)>0$. It can also be seen that $g_1(\eta,\epsilon)$ is a continuous function over $(\eta,\epsilon)\in\mathbb{R}_{++}^2$.

Now let $M>0$ be arbitrary such that
$$
\frac{2\Lambda_m(\delta)}{M}\leq 
\min\left\{1,\frac{M_{\lambda}}{\delta^2}\right\},
$$
and let $(\eta,\epsilon)\in\mathbb{R}_{++}^2$ be arbitrary such that $\eta+\epsilon\geq M$. Consider the following two cases:
\begin{enumerate}[itemindent=13.5pt,leftmargin=0pt]
\item $\eta\epsilon\geq\Lambda_m(\delta)$. Then
$$
\begin{aligned}
g_1(\eta,\epsilon)
&=\delta \Lambda_m(\delta)
-\sqrt{\left(\frac{\delta^2 M_{nc}(\delta)}{4}\sqrt{\eta}
+\sqrt{\eta}\epsilon M_{\lambda}\right)^2}
-\beta^{-1}\sigma_{\eta} \\
&\leq \delta \Lambda_m(\delta)
-\sqrt{\frac{\delta^4 M_{nc}^2(\delta)}{16}\eta
+M_{\lambda}^2\eta\epsilon^2}-\beta^{-1}\sigma_{\eta} \\
&\leq \delta \Lambda_m(\delta)
-\sqrt{\frac{\delta^4 M_{nc}^2(\delta)}{16}\eta
+M_{\lambda}^2\Lambda_m(\delta)\epsilon}-\beta^{-1}\sigma_{\eta} \\
&\leq
\delta \Lambda_m(\delta)
-\min\left\{\frac{\delta^2 M_{nc}(\delta)}{4},
M_{\lambda}\sqrt{\Lambda_m(\delta)}\right\}
\sqrt{M}-\beta^{-1}\sigma_{\eta}.
\end{aligned}
$$
\item $\eta\epsilon<\Lambda_m(\delta)$. In this case, since
$$
\epsilon+\frac{\Lambda_m(\delta)}{\epsilon}
>\epsilon+\eta\geq M,
$$
we must have $\epsilon\geq M/2$ or $\epsilon^{-1}\Lambda_m(\delta)\geq M/2$. In the former case, we have $\eta<\epsilon^{-1}\Lambda(\delta)\leq 2\Lambda(\delta)/M$, and by the choice of $M$ we have $\eta\leq 1$ and $\delta\sqrt{\eta}\leq M_{\lambda}$. Thus
$$
\begin{aligned}
g_1(\eta,\epsilon)
&=\delta\eta\epsilon
-\frac{\sqrt{\eta}}{4}\delta^2 M_{nc}(\delta)
-\sqrt{\eta}\epsilon M_{\lambda}-\beta^{-1}\sigma_{\eta} \\
&\leq
\sqrt{\eta}\epsilon
(\delta\sqrt{\eta}-M_{\lambda})
-\beta^{-1}\sigma_{\eta}
\leq-\beta^{-1}\sigma_1.
\end{aligned}
$$
For the latter case, we have $\eta\geq M-\epsilon\geq M-2\Lambda_m(\delta)/M\geq M-1$, and therefore
$$
\begin{aligned}
g_1(\eta,\epsilon)
&=\delta\eta\epsilon
-\frac{\sqrt{\eta}}{4}\delta^2 M_{nc}(\delta)
-\sqrt{\eta}\epsilon M_{\lambda}-\beta^{-1}\sigma_{\eta} \\
&\leq\delta\Lambda_m(\delta)
-\frac{\sqrt{M-1}}{4}\delta^2 M_{nc}(\delta).
\end{aligned}
$$
\end{enumerate}
Summarizing these results, we get by the arbitrariness of $M$ that
$$
\limsup_{\eta+\epsilon\rightarrow+\infty}
g_1(\eta,\epsilon)<0.
$$
Therefore the set $g_1^{-1}[\mathbb{R}_+]$ is a bounded subset of $\mathbb{R}^2_{++}$.

Now let $(\tilde\eta,\tilde\epsilon)$ be any boundary point of the set $g_1^{-1}[\mathbb{R}_+]$ in $\mathbb{R}^2$, and let $(\eta_k,\epsilon_k)_{k\in\mathbb{N}}$ be a sequence in $g_1^{-1}[\mathbb{R}_+]$ that converges to $(\tilde\eta,\tilde\epsilon)$. Obviously
$$
\liminf_{k\rightarrow\infty}
g_1(\eta_k,\epsilon_k) \geq 0,
$$
and by the continuity of $g_1$ on $\mathbb{R}_{++}^2$, we further have $\lim_{k\rightarrow\infty}g_1(\eta_k,\epsilon_k)=g_1(\tilde\eta,\tilde\epsilon)=0$ if $(\tilde\eta,\tilde\epsilon)\in\mathbb{R}_{++}^2$. Since
$$
\limsup_{
\substack{(\eta,\epsilon)\rightarrow(0,\tilde\epsilon) \\ (\eta,\epsilon)\in\mathbb{R}_{++}^2}}
g_1(\eta,\epsilon) =
\limsup_{\eta\rightarrow 0^+}\left(-\beta^{-1}\sigma_\eta\right)<0,
$$
we see that $\tilde\eta\neq 0$, and then since
$$
\lim_{
\substack{(\eta,\epsilon)\rightarrow(\tilde\eta,0) \\ (\eta,\epsilon)\in\mathbb{R}_{++}^2}}
g_1(\eta,\epsilon) =
-\frac{\sqrt{\tilde\eta}}{4}\delta^2 M_{nc}(\delta)-\beta^{-1}\sigma_{\tilde\eta}<0,
$$
we see that $\tilde\epsilon\neq 0$. Therefore the boundary point of $g_1^{-1}[\mathbb{R}_+]$ in $\mathbb{R}^2$ are all in $\mathbb{R}_{++}^2$. By the continuity of $g_1$, we can conclude that $g_1^{-1}[\mathbb{R}_+]$ is a closed subset of $\mathbb{R}^2$. Together with the boundedness shown above, we have shown that $g_1^{-1}[\mathbb{R}_+]$ is compact.

By the continuity of $g_0(\eta,\epsilon)$ over ${(\eta,\epsilon)\in g_1^{-1}[\mathbb{R}_+]}$ and the compactness of the set $g_1^{-1}[\mathbb{R}_+]$, the minimum of $g_0(\eta,\epsilon)$ over $(\eta,\epsilon)\in g_1^{-1}[\mathbb{R}_+]$ is achieved by some $(\eta^\ast,\epsilon^\ast)\in g_1^{-1}[\mathbb{R}_+]$. By assumption, the set $$
\mathscr{A}_{\mathrm{fp}}(\delta,\beta)=g_1^{-1}[\mathbb{R}_{++}]\subseteq g_1^{-1}[\mathbb{R}_+]
$$
is nonempty, and for $(\eta,\epsilon)\in g_1^{-1}[\mathbb{R}_{++}]$ we have $g_0(\eta,\epsilon)<\delta$, while for $(\eta,\epsilon)\in g_1^{-1}[\{0\}]$ we have $g_0(\eta,\epsilon)=\delta$. Therefore $(\eta^\ast,\epsilon^\ast)$ must be in the set $g_1^{-1}[\mathbb{R}_{++}]$.

Next we show that $\epsilon^\ast=\Lambda_m(\delta^\ast)/\eta^\ast$. It's not hard to check that the function $\epsilon\mapsto g_1(\eta^\ast,\epsilon)$ is continuous over $\epsilon\in\mathbb{R}_{++}$, is monotonic when $\epsilon\leq \Lambda_m(\delta)/\eta^\ast$, and is decreasing when $\epsilon\geq \Lambda_m(\delta)/\eta^\ast$. Thus the set $$
{A_{\eta^\ast}:=\{\epsilon>0: g_1(\eta^\ast,\epsilon)>0\}}
$$
is an open interval in $\mathbb{R}_{++}$. We have
$$
\frac{d}{d\epsilon}g_0(\eta^\ast,\epsilon)
=
-\frac{\eta^\ast\left(\delta M_{\lambda} M_{nc}(\delta)/4
+\beta^{-1}\sigma_{\eta^\ast}\right)}{\left(\eta^\ast\epsilon-\sqrt{\eta^\ast}\delta M_{nc}(\delta)/4\right)^2}<0
$$
when $\epsilon\in A_{\eta^\ast}\cap\left(0,\Lambda_m(\delta)/\eta^\ast\right)$, and
$$
\frac{d}{d\epsilon}g_0(\eta^\ast,\epsilon)
=
\frac{\sqrt{\eta^\ast}M_{\lambda}}{\Lambda_m(\delta)-\sqrt{\eta^\ast}\delta M_{nc}(\delta)/4}> 0
$$
when $\epsilon\in A_{\eta^\ast}\cap
\left(\Lambda_m(\delta)/\eta^\ast,+\infty\right)$. Since the minimum of $g_0(\eta^\ast,\epsilon)$ over $\epsilon\in A_{\eta^\ast}$ is achieved, we must have $\epsilon^\ast=\Lambda_m(\delta)/\eta^\ast$. This conclusion has a further implication: If we define the function
$$
\begin{aligned}
\tilde g_0(x)
&=g_0\left(x^{-2},x^2\Lambda_m(\delta)\right) =\frac{\beta^{-1}\sigma_{x^{-2}}+x\Lambda_m(\delta)M_{\lambda}}{\Lambda_m(\delta)-x^{-1}\delta M_{nc}(\delta)/4},\\
\tilde g_1(x)
&=g_1\left(x^{-2},x^2\Lambda_m(\delta)\right)
\end{aligned}
$$
for $x\in(x_l,+\infty)$, where
$$
x_l:=\frac{\delta M_{nc}(\delta)}{4\Lambda_m(\delta)},
$$
then
$$
\sqrt{\frac{\epsilon^\ast}{\Lambda_m(\delta)}}
=\frac{1}{\sqrt{\eta^\ast}}
\in
\argmin_{x\in \tilde g_1^{-1}[\mathbb{R}_{++}]}
\tilde g_0(x).
$$
Finally, we prove that $\tilde g_0(x)$ has a unique minimizer over $(x_l,+\infty)$. Since $\sigma_{x^{-2}}$ is a convex and nondecreasing function of $x$, it is absolutely continuous and admits a weak derivative $u(x)$ which is nonnegative and nondecreasing in $x$. Then we see that $\tilde g_0(x)$ is also absolutely continuous, whose weak derivative is given by
$$
D\tilde g_0(x)
=
\frac{\beta^{-1}u(x)+\Lambda_m(\delta)M_{\lambda}
-\frac{\beta^{-1}x^{-1}\sigma_{x^{-2}}+\Lambda_m(\delta)M_{\lambda}}{x(\Lambda_m(\delta)-x^{-1}\delta M_{nc}(\delta)/4)}\frac{\delta M_{nc}(\delta)}{4}}
{\Lambda_m(\delta)-x^{-1}\delta M_{nc}(\delta)/4}.
$$
We can see that $x^{-1}\sigma_{x^{-2}}$ is equal to
$$
\sup_{\substack{t_1,t_2\in[0,S],\\t_1\neq t_2}}
\frac{\left(
x^{-2}\left\|x^\ast(t_2)-x^\ast(t_1)\right\|^2
+
\left\|\lambda^\ast(t_2)-\lambda^\ast(t_1)\right\|^2
\right)^{1/2}}{|t_2-t_1|},
$$
showing that $x^{-1}\sigma_{x^{-2}}$ is nonincreasing in $x$. Then we can easily verify that
$$
\frac{\beta^{-1}x^{-1}\sigma_{x^{-2}}+\Lambda_m(\delta)M_{\lambda}}{x(\Lambda_m(\delta)-x^{-1}\delta M_{nc}(\delta)/4)}
$$
is a strictly decreasing function of $x$, as the numerator is nonincreasing in $x$ and the denominator is positive and strictly increasing in $x$ for $x>\delta M_{nc}(\delta)/(4\Lambda_m(\delta))$. Moreover, we have
$$
\begin{aligned}
&\lim_{x\rightarrow+\infty}
\left(
\beta^{-1}u(x)+\Lambda_m(\delta)M_{\lambda}
-\frac{\beta^{-1}x^{-1}\sigma_{x^{-2}}+\Lambda_m(\delta)M_{\lambda}}{x(\Lambda_m(\delta)-x^{-1}\delta M_{nc}(\delta)/4)}\frac{\delta M_{nc}(\delta)}{4}
\right) \\
=\ &
\beta^{-1}\lim_{x\rightarrow+\infty}u(x)
+\Lambda_m(\delta) M_{\lambda}>0,
\end{aligned}
$$
while
$$
\lim_{x\rightarrow x_l^+}
\left(
\beta^{-1}u(x)+\Lambda_m(\delta)M_{\lambda}
-\frac{\beta^{-1}x^{-1}\sigma_{x^{-2}}+\Lambda_m(\delta)M_{\lambda}}{x(\Lambda_m(\delta)-x^{-1}\delta M_{nc}(\delta)/4)}\frac{\delta M_{nc}(\delta)}{4}
\right) 
=-\infty.
$$
Therefore there exists a unique $x^\ast\in\left(x_l,+\infty\right)$ such that $D\tilde g_0(x)<0$ for $x\in(x_l,x^\ast)$ and $D\tilde g_0(x)>0$ for $x\in(x^\ast,+\infty)$. We then see that $\tilde g_0(x)$ is a unimodal function over $x\in(x_l,+\infty)$ with the unique minimizer $x^\ast$ . Therefore
$$
\{x^\ast\}=
\argmin_{x\in (x_l,+\infty)}
\tilde g_0(x).
$$
Since $\tilde g_1^{-1}[\mathbb{R}_{++}]\subseteq(x_l,+\infty)$, and both $\tilde g_1^{-1}[\mathbb{R}_{++}]$ and $(x_l,+\infty)$ are open subsets of $\mathbb{R}_{++}$, we have
$$
\argmin_{x\in \tilde g_1^{-1}[\mathbb{R}_{++}]}
\tilde g_0(x)
\subseteq 
\argmin_{x\in (x_l,+\infty)}
\tilde g_0(x),
$$
from which we can conclude that $x^\ast=1/\sqrt{\eta^\ast}=\sqrt{\epsilon^\ast/\Lambda_m(\delta)}$.

The unimodality of $b_{\delta,\beta}(\epsilon)$ can be obtained by noting that $b_{\delta,\beta}(\epsilon)=\tilde g_0\big(\sqrt{\epsilon/\Lambda_m(\delta)}\big)$ and that $\epsilon\mapsto\sqrt{\epsilon/\Lambda_m(\delta)}$ is a strictly increasing function.

\bibliographystyle{siamplain}
\bibliography{biblio.bib,biblio2.bib}
\end{document}